\documentclass{article}
\usepackage[utf8]{inputenc}
\usepackage[textwidth=16cm,textheight=22cm,centering]{geometry}

\usepackage{amsmath, amssymb, amstext, amsfonts, textcomp, amsxtra, amsbsy, amsgen, amsopn, amscd, mathrsfs, amsthm, latexsym, array}
\usepackage{bbm, dsfont}
\usepackage{color}

\usepackage[colorlinks]{hyperref}

\newtheorem{theorem}{Theorem}[section]
\newtheorem{proposition}{Proposition}[section]
\newtheorem{lemma}      [theorem]{Lemma}
\newtheorem{corollary}  [theorem]{Corollary}

\theoremstyle{definition}
\newtheorem{definition}{Definition}[section]
\theoremstyle{remark}
\newtheorem{remark}{Remark}[section]
\theoremstyle{example}
\newtheorem{example}{Example}[section]

\newcommand{\intL}{{\rm L}}			
\newcommand{\Cont}{{\rm C}}			
\newcommand{\Sch}{\mathcal{S}}		
\newcommand{\Ht}{{\rm Ht}}			
\newcommand{\Tr}{{\rm Tr}}			
\newcommand{\Nr}{{\rm Nr}}			
\newcommand{\sgn}{{\rm sgn}}			

\newcommand{\Mellin}[2][]{\mathfrak{M}_{#1}[#2]}				

\newcommand{\norm}[1][\cdot]{\lvert #1 \rvert}                  
\newcommand{\extnorm}[1]{\left\lvert #1 \right\rvert}           
\newcommand{\Norm}[1][\cdot]{\lVert #1 \rVert}                  
\newcommand{\extNorm}[1]{\left\lVert #1 \right\rVert}           
\newcommand{\Pairing}[2]{\langle #1, #2 \rangle}                
\newcommand{\ProjF}{{\rm F}}

\newcommand{\gp}[1]{\mathbf{#1}}        
\newcommand{\GL}{{\rm GL}}
\newcommand{\PGL}{{\rm PGL}}
\newcommand{\PO}{{\rm PO}}
\newcommand{\SL}{{\rm SL}}

\newcommand{\N}{\mathbb{N}}
\newcommand{\Z}{\mathbb{Z}}
\newcommand{\Mat}{{\rm M}}

\newcommand{\Q}{\mathbb{Q}}     
\newcommand{\R}{\mathbb{R}}     
\newcommand{\C}{\mathbb{C}}     
\newcommand{\E}{\mathbf{E}}
\newcommand{\F}{\mathbf{F}}
\newcommand{\A}{\mathbb{A}}     
\newcommand{\vo}{\mathfrak{o}}  
\newcommand{\vp}{\mathfrak{p}}  
\newcommand{\idl}[1]{\mathfrak{#1}}
\newcommand{\vO}{\mathcal{O}}

\newcommand{\rpL}{{\rm L}}		
\newcommand{\rpR}{{\rm R}}		
\newcommand{\Bas}{\mathcal{B}}	
\newcommand{\Ind}{{\rm Ind}}		

\newcommand{\fin}{{\rm fin}}		
\newcommand{\eis}{{\rm E}}		
\newcommand{\RamC}{\vartheta}	
\newcommand{\Cond}{\mathbf{C}}	

\newcommand{\fa}{\mathbf{a}}
\newcommand{\fb}{\mathbf{b}}

\newcommand{\Vol}{{\rm Vol}}		
\makeatletter

\newcommand{\Rmnum}[1]{\expandafter\@slowromancap\romannumeral #1@}
\makeatother

\newcommand{\BC}{{\mathbb {C}}}

\newcommand{\BF}{{\mathbb {F}}}

\newcommand{\CC}{{\mathcal {C}}}

\newcommand{\CS}{{\mathcal {S}}}

\newcommand{\Fo}{{\mathfrak {o}}}
\newcommand{\Fp}{{\mathfrak {p}}}

\newcommand{\RI}{{\mathrm {I}}}

\def\wb{\overline} 

\def\vpi{\varpi}

\def\p{\prime}

\def\Cl{\mathrm{Cl}}
\def\Tr{\mathrm{Tr}}
\def\Nr{\mathrm{Nr}}
\def\fin{\mathrm{fin}}

\renewcommand{\Im}{{\mathrm{Im}}}

\renewcommand{\Re}{{\mathrm{Re}}}

\newcommand{\Res}{{\mathrm{Res}}}

\newcommand{\SO}{{\mathrm{SO}}}

\newcommand{\tr}{{\mathrm{tr}}}

\newcommand{\ud}{\,\mathrm{d}}
\newcommand{\vol}{{\mathrm{vol}}}

\newcommand{\wt}{\widetilde}

\newcommand{\bs}{\backslash}

\newtheorem{thm}{Theorem}[subsection]
\newtheorem{defin}[thm]{Definition}

\newtheorem{lem}[thm]{Lemma}
\newtheorem{cor}[thm]{Corollary}

\title{Bias of Root Numbers for Hilbert Newforms of Cubic Level}
\author{Zhilin Luo \quad Qinghua Pi \quad Han Wu}

\begin{document}

\maketitle

\begin{abstract}
	We give a general formula of the bias of root numbers for Hilbert modular newforms of cubic level. Explicit calculation is given when the base field is $\Q, \Q(\sqrt{2}), \Q(\sqrt{5})$ and the level is the cube of certain rational integers. This complements a previous result of the second author and extends the bias phenomenon to the number fields. Our method is based on Jacquet-Zagier's trace formula, and the explicit calculation works generally for all real quadratic fields of narrow class number one and for rational cubic levels.
\end{abstract}

\tableofcontents

\section{Introduction}

	\subsection{Newforms with Prescribed Root Number and Bias Phenomenon}
	
	Let $H_k^{\pm}(N)$ be the subspace of newforms of weight $2k$ for $\Gamma_0(N)$ with global root number $\pm 1$. It is discovered by Iwaniec-Luo-Sarnak in \cite{ILS00} that these two families of modular forms have different behaviour with respect to the distribution of the low lying zeros of the associated $L$-functions. As a preliminary result, they obtained the asymptotic formulas of $\dim H_k^{\pm}(N)$ \cite[Corollary 2.14]{ILS00} for square-free $N$. These formulas were considerably improved to exact ones by K.Martin \cite{Ma18}, who ran to this problem with a different motivation and also discovered an interesting bias phenomenon in favour of $H^+$ by establishing (we only present the case $N > 3$ for brevity) the non-negativity of
	$$ B(k,N) := \dim_{\C} H_k^+(N) - \dim_{\C} H_k^-(N) = c_N h(D_{-N}) - \delta(k,1), $$
	where $h(D_{-N})$ is the class number of $\Q(\sqrt{-N})$, $\delta$ is the Kronecker symbol and
	$$ c_N = \left\{ \begin{matrix} 1/2 & \quad \text{if } N \equiv 1,2 (4) \\ 1 & \text{if } N \equiv 7 (8) \\ 2 & \text{if } N \equiv 3 (8) \end{matrix} \right. . $$
	Such formulas were recently generalized to the case $k>1$ and $N=M^3$ for square-free $M$ by Pi-Qi \cite{PQ21}. Various methods have been applied to establish such formulas. The method of Iwaniec-Luo-Sarnak is via Petersson trace formula, which was refined and generalized by Pi-Qi with some new analytic input by Balkanova, Frolenkov and (Shenhui) Liu. The method of Martin is via a simple trace formula due to Yamauchi, which he remarked to be possibly messy for general level $N$.
	
	In a previous work of the third author \cite{W19}, the deduction of the Selberg trace formula for $\GL_2$ over a general number field $\F$ via applying the Rankin-Selberg method to the automorphic kernel function, an idea due to Jacquet-Zagier \cite{JZ87}, is completed. In a subsequent joint work of the third author \cite{CWZ21}, some variant of that method for $\SL_2$ was refined to extend the prime geodesic theorems to principal congruence subgroup case. An effective method of explicit computation of certain terms in the geometric side, namely the orbital integrals for $\F$-elliptic terms, is developed. In this paper, we continue the development of Jacquet-Zagier's method by applying it to the computation of $B(k,N^3)$ with $N$ square-free, as well as its generalization to Hilbert modular forms of cubic level. We encounter two major difficulties:
\begin{itemize}
	\item One difficulty is the localization of holomorphic forms on the spectral side, especially when weight $2$ appears. Overcoming this difficulty in the classical approach via trace formulas consists of sophisticated extension of the class of test functions, since the matrix coefficients of a discrete series is not smooth with compact support in general. We avoid such sophistication by a limiting process with a family of test functions in the space of Schwartz functions, whose construction is inspired by Hecke's treatment of weight $2$ Eisenstein series \cite{He27}. In the process we need the crucial analytic continuation of a certain Dirichlet series, whose coefficients are special values of some generalized Zagier's $L$-function. To this end, an interesting application of the trace formula in the inverse direction is performed. This is the major innovation in our method. We hope it to have general utility in the future applications of Jacquet-Zagier's trace formula.
	\item The other difficulty is the explicit computation of the special values $L(1,\eta_{\E/\F})$, where $\F$ is the base field, $\E/\F$ runs over the relevant quadratic field extensions appearing on the geometric side, and $\eta_{\E/\F}$ is the associated quadratic Hecke character. For example, in the case $\F$ is a real quadratic number field, we only know how to calculate $L(1,\eta_{\E/\F})$ for \emph{biquadratic} fields $\E$, due to our limited knowledge in computational algebraic number theory. This excludes the explicit computation of bias for certain cubic levels. For example over $\F=\Q(\sqrt{2})$, the level $N=(3+\sqrt{2})^3$ is excluded in our explicit formulas of bias, since $\F(\sqrt{3+\sqrt{2}})$ is not biquadratic.
\end{itemize}   
	Nevertheless, our main formula in Theorem \ref{MainGenF} below is written in the general setting. We hope it to be useful in determining all bias phenomena in the future.
	
\begin{remark}
	It would be worth to compare our method with the method of taking pseudo-coefficients (at the archimedean places) directly in the trace formula (call it PC for simplicity). It is true that PC is simpler than the ours for the computation of bias for cubic level, since the geometric side of the relevant trace formula will consist of only elliptic orbital integrals. This is no longer true if we try to compute the bias for square-free level, in which case one can not take matrix coefficients as test function at the ramified primes because these coefficients for the Steinberg representations are not integrable. The natural choice in this case is a pseudo-coefficient inspired from the Atkin-Lehner theory. Up to normalization, it is given by a characteristic function
	$$ f_{\vp}(g) = \mathbbm{1}_{\gp{K}_0[\vp]} \left( \begin{pmatrix} & 1 \\ \varpi_{\vp} & \end{pmatrix}^{-1} g \right), $$
	where $\gp{K}_0[\vp]$ is the Iwarohi subgroup at $\vp$. In this case, especially for $k=1$, the terms \cite[(6.36) \& (6.37)]{GJ79} are no longer vanishing and difficult to compute, since their expression in terms of the orbital integrals or trace distributions are not available. Our method should still work, although some extra difficulties appear. Precisely, Lemma \ref{SimpleRSF} (1) below no longer applies to our test function, and we need to handle the terms $I_{\infty}''(s)$ and $I_{\infty}'''(s)$. The techniques to handle them are different from those designed for elliptic orbital integrals, which are presented in the current paper. We reserve it for another paper.
\end{remark}	

\begin{remark}
	To the best of our knowledge, the dimension formulas in the number field case, namely the analogue for Hilbert modular forms of $\dim_{\C} H_k^+(N) + \dim_{\C} H_k^-(N)$ are not known. It looks to us that handling them with our method has the same level of difficulty as the bias problem for square-free level mentioned in the previous remark. We will try to compute them in the same reserved paper, and study the relevant equidistribution problem. However, we don't think that our method can compute the dimensions of newforms including holomorphic Eisenstein series (for even power level only), which looks like a problem of different nature.
\end{remark}

\begin{remark}
	We also plan to investigate the possible application of Jacquet-Zagier's trace formula to the first moment of the symmetric square $L(s, f, \mathrm{Sym})$, for which several results (e.g. Theorem \ref{AutoKernSpecDSch}) contained in the current paper are fundamental. For this reason, they are established in a form stronger than what we actually need in this paper.
\end{remark}

	\subsection{A Generalization of Zagier's $L$-Function}
	
	To a discriminant $\delta \in \Z$ (that is $\delta = b^2-4ac$ for some $a,b,c \in \Z$), Siegel\footnote{Siegel seems to be the first to introduce this $L$-series (see the discussion in \cite[\S 2]{SY13}). But in the literature many people call it Zagier's $L$-function due to his work \cite{Za76}.} associated an $L$-series
\begin{equation}
	L(s, \delta) := \frac{\zeta(2s)}{\zeta(s)} \sum_{q=1}^{\infty} \rho_q(\delta) q^{-s}, \quad \rho_q(\delta) := \extnorm{ \left\{ x \pmod{2q} \ \middle| \ x^2 \equiv \delta \pmod{4q} \right\} }.
\label{ZagLQ}
\end{equation} 
	If $\delta$ is non-zero, we may write $\delta = Dl^2$ with $D$ a fundamental discriminant. Then we have the factorization (see \cite[(4) \& (7)]{SY13})
	$$ L(s,\delta) = L(s, \chi_D) \cdot \prod_{p^k \parallel l} p^{\left( \frac{1}{2}-s \right)k} \frac{\left( p^{\left( \frac{1}{2}-s \right) (k+1)} - p^{\left( s-\frac{1}{2} \right) (k+1)} \right) - \chi_D(p) p^{-\frac{1}{2}} \left( p^{\left( \frac{1}{2}-s \right) k} - p^{\left( s-\frac{1}{2} \right) k} \right) }{p^{\frac{1}{2}-s} - p^{s-\frac{1}{2}}}, $$
	where $\chi_D$ is the unique quadratic character associated with the quadratic extension $\Q[\sqrt{D}]/\Q$.
	
	For our purpose, we need a generalization to number field, whose arithmetic meaning, namely an analogue of the expression (\ref{ZagLQ}), remains mysterious to us.

\begin{definition}
	Let $\idl{J}$ be a non-zero integral ideal of $\vo$. A $\idl{J}$-discriminant is a $\delta \in \idl{J}^{-2}$ such that $\delta = b^2 - 4a$ for some $a \in \idl{J}^{-2}, b \in \idl{J}^{-1}$.
\end{definition}

\noindent Let $\delta$ be a non-zero $\idl{J}$-discriminant. Let $\E = \F[\sqrt{\delta}]$ be the corresponding quadratic (algebra) extension of $\F$, whose discriminant ideal is $D_{\E/\F}$. Let $\eta_{\E}$ be the quadratic Hecke character associated with $\E/\F$. We have $\delta \vo = D_{\E/\F} l^2$ for a unique fractional ideal $l \subset \idl{J}^{-1}$. (To see the existence, choose any $x \in \idl{J}$ so that $\vo[\sqrt{\delta x^2}] \subset \vo_{\E}$ is a quadratic order, and compute its discriminant at every finite place $\vp$.) We define
\begin{align}
	L \left( s+\frac{1}{2},\delta; \idl{J} \right) &= \Nr(\idl{J})^{-s-\frac{1}{2}} \cdot L \left( s+\frac{1}{2}, \eta_{\E} \right) \cdot \nonumber \\
	&\quad \prod_{\vp^k \parallel l} q^{-s(k+k_{\vp})} \frac{\left( q^{s (k+k_{\vp}+1)} - q^{-s (k+k_{\vp}+1)} \right) - \eta_{\E}(\varpi_{\vp}) q^{-\frac{1}{2}} \left( q^{s (k+k_{\vp})} - q^{-s (k+k_{\vp})} \right) }{q^s - q^{-s}}, \label{ZagLF}
\end{align}
	where we have denoted $q = \Nr(\vp), k_{\vp} = \mathrm{ord}_{\vp}(\idl{J})$ on the right hand side, and we have the formula
	$$ \eta_{\E}(\varpi_{\vp}) = \left\{ \begin{matrix} -1 & \text{if } \E/\F \text{ is unramified at } \vp \\ 0 & \text{if } \E/\F \text{ is ramified at } \vp \\ 1 & \text{if } \E/\F \text{ is split at } \vp \end{matrix} \right. . $$
	If $\idl{N}$ is a (square-free) integral ideal, we also denote a variant as
\begin{align}
	L^{(\idl{N})} \left( s+\frac{1}{2},\delta; \idl{J} \right) &= \Nr(\idl{J})^{-s-\frac{1}{2}} \cdot L \left( s+\frac{1}{2}, \eta_{\E} \right) \cdot \nonumber \\
	&\quad \prod_{\vp^k \parallel l, \vp \nmid \idl{N}} q^{-s(k+k_{\vp})} \frac{\left( q^{s (k+k_{\vp}+1)} - q^{-s (k+k_{\vp}+1)} \right) - \eta_{\E}(\varpi_{\vp}) q^{-\frac{1}{2}} \left( q^{s (k+k_{\vp})} - q^{-s (k+k_{\vp})} \right) }{q^s - q^{-s}}. \label{ZagLFVar}
\end{align}

\begin{lemma}
	For any $\epsilon > 0$ and as $\idl{J}, \delta \in \idl{J}^{-2}$ vary, we have
	$$ \extnorm{L(1,\delta;\idl{J})}, \extnorm{L^{(\idl{N})}(1,\delta;\idl{J})} \ll_{\epsilon} \Nr(\idl{J})^{-1+2\epsilon} \prod_{v \mid \infty} \norm[\delta]_v^{\epsilon}. $$
\label{GenZagLBd}
\end{lemma}
\begin{proof}
	Since $\vo \supset \delta \idl{J}^2 = D_{\E/\F} (l \idl{J})^2$, we have $\Nr(D_{\E/\F}) \leq \Nr(\delta \idl{J}^2)$. Applying the convex bound of $L$-functions (see \cite[(5.21)]{IK04}), we get
	$$ \extnorm{L(1,\delta;\idl{J})} \leq \Nr(\idl{J})^{-1} \extnorm{L(1,\eta_{\E})} \prod_{\vp^k \parallel l} (1+3q^{-1}) \ll_{\epsilon} \Nr(\idl{J})^{-1} \Nr(D_{\E/\F})^{\epsilon} \Nr(\delta \idl{J}^2)^{\epsilon}, $$
	which implies readily the desired bound. The same proof applies to $L^{(\idl{N})}(1,\delta;\idl{J})$.
\end{proof}

	\subsection{Main Results}
	
	Recall that $\F$ is a totally real number field of degree $d$ with ring of integers $\vo$ and absolute discriminant $D_{\F}$. Let $\vec{k}=(k_v)_{v \mid \infty} \in \Z_{\geq 1}^d$ be a weight vector, and let $\idl{N}$ be a square-free ideal of $\vo$. Let $U_+$ be the group of totally positive units of $\F$, and choose $U_2$ as a system of representatives for $U_+/(\vo^{\times})^2$. Denote by $\mathcal{A}(\vec{k},\idl{N}^3)$ the set of Hilbert modular Hecke eigen-newforms on $\mathbb{H}^d$ with weight $2\vec{k}$ and level $\idl{N}^3$. Our first main result is a general formula for the bias of root numbers in $\mathcal{A}(\vec{k},\idl{N}^3)$ defined by
	$$ B(\vec{k},\idl{N}^3) := \sum_{f \in \mathcal{A}(\vec{k}, \idl{N}^3)} \varepsilon \left( \frac{1}{2},f \right). $$
	
\begin{theorem}
	(1) If $\idl{N}$ is not a square in the narrow class group $\Cl_+(\F)$ of $\F$, then $B(\vec{k},\idl{N}^3) = 0$.
	
\noindent (2) Otherwise, we choose an integral ideal $\idl{J}$ such that $\idl{J}^{-2} \idl{N}$ is principal with a totally positive generator $N$. For $x \in \F$ and $v \mid \infty$, we write $x_v$ for its image under the natural embedding $\F \to \F_v$. Then we have
\begin{align*}
	B(\vec{k},\idl{N}^3) &= 2 D_{\F}^{\frac{1}{2}} \sum_{u \in U_2} \sum_n 2^{-\mathbbm{1}_{n=0}} \left( \prod_{v \mid \infty} \frac{(4N_v)^{k_v}}{2\pi} \Re \left( \frac{1}{\left( \sqrt{(4uN-(nuN)^2)_v} +i (nuN)_v \right)^{2k_v-1}} \right) \right) \cdot \\
	&\quad L^{(\idl{N})}(1, (nuN)^2-4uN; \idl{J}) \cdot A(n,\idl{N}),
\end{align*} 
	where the inner sum is over $n \in \idl{J}/\{ \pm 1 \}$ so that 
\begin{itemize}
	\item $4uN-(nuN)^2$ is totally positive,
	\item and $\F[\sqrt{4uN-(nuN)^2}]$ is ramified at every prime ideal $\vp \mid \idl{N}$,
\end{itemize}	
	$L^{(\idl{N})}(s,\delta;\idl{J})$ is a generalization of Zagier's $L$-function defined in (\ref{ZagLFVar}), and
	$$ A(n,\idl{N}) := \prod_{\vp \mid \idl{N}} \left\{ \begin{matrix} \Nr(\vp)-1 & \text{if } n \in \vp \idl{J} \\ -1 & \text{otherwise} \end{matrix} \right. . $$
\label{MainGenF}
\end{theorem}

\begin{remark}
	Note we have the following equality
\begin{equation} 
	\norm[U_2] = \norm[\Cl_+(\F)] / \norm[\Cl(\F)]. 
\label{U2ClGpsRel}
\end{equation}
	To see this, let $\F_+$, resp. $P$, resp. $P_+$, resp. $I$ be the group of positive elements in $\F^{\times}$, resp. principal ideals, resp. principal ideals with totally positive generator, resp. ideals. On the one hand, we have an exact sequence
	$$ 1 \to U_+ \to \vo^{\times} \to \F^{\times}/\F_+ \to P/P_+ \to 1, $$
	since $P=\F^{\times}/\vo^{\times}$ and $P_+ = \F_+ \vo^{\times}/\vo^{\times}$ by definition. By the weak approximation theorem, we have $\F^{\times}/\F_+ \simeq (\Z/2\Z)^d$. By Dirichlet's theorem of units, we know $\vo^{\times} \simeq \Z/2\Z \oplus \Z^{d-1}$, hence $\vo^{\times}/(\vo^{\times})^2 \simeq (\Z/2\Z)^d$. Therefore $\norm[U_2] = \norm[\vo^{\times}/(\vo^{\times})^2]/\norm[\vo^{\times}/U_+] = \norm[\F^{\times}/\F_+]/\norm[\vo^{\times}/U_+]=\norm[P/P_+]$. On the other hand, we have another short exact sequence by definition
	$$ 1 \to P/P_+ \to I/P_+ = \Cl_+(\F) \to I/P = \Cl(\F) \to 1. $$
	The equality (\ref{U2ClGpsRel}) follows readily.
\end{remark}

	We would like to draw some concrete consequences of the above general formula. First consider the case $\F = \Q$. We shall not distinguish an ideal and its positive generator, since $\Q$ has narrow class number one. We have the following generalization of \cite{PQ21} which includes the case $k=1$.
	
\begin{corollary}
	Let $N > 1$ be a square-free natural number. Denote by $\varphi(\cdot)$ the Euler's totient function, and by $h(D)$ the class number of $\Q[\sqrt{D}]$ for a fundamental discriminant $D$.
\begin{itemize}
	\item[(1)] If $N > 3$, then we have
	$$ B(k,N^3) = \left\{ \begin{matrix} h(-N)\varphi(N) & \text{if } N \equiv 7 (8) \\ 2h(-N)\varphi(N) & \text{if } N \equiv 3 (8) \\ 2^{-1} h(-4N) \varphi(N) & \text{if } N \equiv 1,2 (4) \end{matrix} \right.. $$
	\item[(2)] If $N=2$, then we have
	$$ B(k,8) = \left\{ \begin{matrix} 0 & \text{if } k \equiv 0,1 (4) \\ 1 & \text{if } k \equiv 2,3 (4) \end{matrix} \right.. $$
	\item[(3)] If $N=3$, then we have
	$$ B(k,27) = \left\{ \begin{matrix} 1 & \text{if } k \equiv 0,1 (3) \\ 2 & \text{if } k \equiv 2 (3) \end{matrix} \right.. $$
\end{itemize}
\label{MainQ}
\end{corollary}

	In the number field setting, we first investigate the case $\F = \Q[\sqrt{2}]$, for which any rational square-free odd integer $N \geq 3$ is still square-free in the PID $\vo = \Z[\sqrt{2}]$.
	
\begin{corollary}
	Let $N \geq 3$ be a square-free natural odd number. Denote by $\varphi_{\F}(\cdot)$ the Euler's totient function over $\F$, such that
	$$ \varphi_{\F}(N) = \sideset{}{_{\vp \mid N}} \prod (\Nr(\vp)-1). $$
	Write $h(D)$ for the class number of $\Q[\sqrt{D}]$ for a fundamental discriminant $D$.
\begin{itemize}
	\item[(1)] If $N > 3$, then we have
	$$ B(\vec{k},N^3) = \varphi_{\F}(N) h(-N) h(-2N) \cdot \left\{ \begin{matrix} \frac{5}{2} & \text{if } N \equiv 3 (8) \\ 1 & \text{if } N \equiv 7 (8) \\ \frac{3}{4} & \text{if } N \equiv 1 (4) \end{matrix} \right.. $$
	\item[(2)] If $N = 3$, writing $\vec{k}=(k_1,k_2)$ then we have
	$$ B(\vec{k}, 27) = \left\{ \begin{matrix} 12 & \text{if } k_1,k_2 \equiv 2 (3) \\ 13 & \text{if } k_1,k_2 \equiv 0,1 (3) \\ 14 & \text{otherwise} \end{matrix} \right.. $$
\end{itemize} 
\label{MainQSq2}
\end{corollary}

	We turn to the case $\F = \Q[\sqrt{5}]$, for which any rational square-free integer $N \geq 2$ with $5 \nmid N$ is still square-free in the PID $\vo=\Z[(\sqrt{5}+1)/2]$. We do not give the details of proof since it is quite similar to that of Corollary \ref{MainQSq2}, but only point out a subtlety in the case of level $27$: the prime ideal $2$ is split in the quadratic extension $\F[\sqrt{-3}]/\F$, although it is inert in $\Q[\sqrt{-3}]/\Q$.
	
\begin{corollary}
	Let $N \geq 2$ be a square-free natural number with $5 \nmid N$. With the same notation as in Corollary \ref{MainQSq2}, we have:
\begin{itemize}
	\item[(1)] If $N > 3$, then we have
	$$ B(\vec{k},N^3) = \varphi_{\F}(N) h(-N) h(-5N) \cdot \left\{ \begin{matrix} \frac{1}{4} & \text{if } N \equiv 1,2 (4) \\ \frac{3}{2} & \text{if } N \equiv 3 (8) \\ 1 & \text{if } N \equiv 7 (8) \end{matrix} \right.. $$
	\item[(2)] If $N=2$, writing $\vec{k}=(k_1,k_2)$ then we have
	$$ B(\vec{k},8) = \left\{ \begin{matrix} 1 & \text{if } k_1,k_2 \equiv 0,1(4) \text{ or } 2,3 (4) \\ 2 & \text{otherwise} \end{matrix} \right.. $$
	\item[(3)] If $N=3$, writing $\vec{k}=(k_1,k_2)$ then we have
	$$ B(\vec{k},27) = \left\{ \begin{matrix} 4 & \text{if } k_1,k_2 \equiv 2 (3) \\ 5 & \text{if } k_1,k_2 \equiv 0,1 (3) \\ 6 & \text{otherwise} \end{matrix} \right.. $$
\end{itemize}
\end{corollary}

	\subsection{Notation and Convention}
	
		\subsubsection{Complex Analytic Notation}
	
	If $f$ is a meromorphic function around $s=s_0$, we denote the coefficients in its Laurent expansion by
	$$ f(s) = \sideset{}{_{-\infty < k < 0}} \sum \frac{f^{(k)}(s_0)}{(-k)!} (s-s_0)^k + \sideset{}{_{k \geq 0}} \sum \frac{f^{(k)}(s_0)}{k!} (s-s_0)^k. $$

		\subsubsection{Number Theoretic Notation}
		
	Throughout the paper, $\F$ is a (fixed) totally real number field with ring of integers $\vo$ and of degree $d=[\F : \Q]$. Let its absolute different (resp. discriminant) be $\idl{D}_{\F}$ (resp. $D_{\F}$). $V_{\F}$ (resp. $V_{\infty}$) denotes the set of places (resp. infinite places) of $\F$ and for any $v \in V_{\F}$, $\F_v$ is the completion of $\F$ with respect to the absolute value $\norm_v$ corresponding to $v$. $\A = \A_\F$ is the ring of adeles of $\F$, while $\A^{\times}$ denotes the group of ideles. We often write a finite place by $\vp$, which also denotes its corresponding prime ideal of $\vo$.
	
	We put the standard Tamagawa measure $dx = \sideset{}{_v} \prod dx_v$ on $\A$ (resp. $d^{\times}x = \sideset{}{_v} \prod d^{\times}x_v$ on $\A^{\times}$). We recall their constructions. Let $\Tr = \Tr_{\Q}^{\F}$ be the trace map, extended to $\A \to \A_{\Q}$. Let $\psi_{\Q}$ be the additive character of $\A_{\Q}$ trivial on $\Q$, restricting to the infinite place as
	$$ \Q_{\infty} = \R \to \C^{(1)}, \quad x \mapsto e^{2\pi i x}. $$
	We put $\psi = \psi_{\Q} \circ \Tr$, which decomposes as $\psi(x) = \sideset{}{_v} \prod \psi_v(x_v)$ for $x=(x_v)_v \in \A$. $dx_v$ is the additive Haar measure on $\F_v$, self-dual with respect to $\psi_v$. Precisely, if $\F_v = \R$, then $dx_v$ is the usual Lebesgue measure on $\R$; if $v = \vp < \infty$ such that $\vo_{\vp}$ is the valuation ring of $\F_{\vp}$ with prime ideal $\vp \vo_{\vp}$ and a fixed uniformizer $\varpi_{\vp}$, then $dx_{\vp}$ gives $\vo_{\vp}$ the mass $D_{\vp}^{-1/2}$, where $D_{\vp}$ is the local component at $\vp$ of the discriminant $D_\F$ of $\F/\Q$ such that $D_\F = \sideset{}{_{\vp < \infty}} \prod D_{\vp}$. Consequently, the quotient space $\F \backslash \A$ with the above measure quotient by the discrete measure on $\F$ admits the total mass $1$ \cite[Ch.\Rmnum{14} Prop.7]{Lan03}. Recall the local zeta-functions: if $\F_v = \R$, then $\zeta_v(s) = \Gamma_{\R}(s) = \pi^{-s/2} \Gamma(s/2)$; if $v=\vp < \infty$ then $\zeta_{\vp}(s) = (1-q_{\vp}^{-s})^{-1}$, where $q_{\vp} := \Nr(\vp)$ is the cardinality of $\vo/\vp$. We then define
	$$ d^{\times} x_v := \zeta_v(1) \frac{dx_v}{\norm[x]_v}. $$
	In particular, $\Vol(\vo_{\vp}^{\times}, d^{\times}x_{\vp}) = \Vol(\vo_{\vp}, dx_{\vp})$ for $\vp < \infty$. 

\noindent The complete Dedekind zeta-function satisfies the functional equation
	$$ \Lambda_{\F}(s) := \sideset{}{_{v \in V_{\F}}} \prod \zeta_v(s) = D_\F^{\frac{1}{2}-s} \Lambda_{\F}(1-s). $$
	
		\subsubsection{Automorphic Representation Theoretic Notation}
		
	We will work on algebraic groups $\GL_2$ and $\PGL_2$ over $\F$, the latter being the quotient of $\GL_2$ by its center over $\A$ or $\F_v$ in the category of locally compact groups. We put the \emph{hyperbolic measure} instead of the Tamagawa measure on $\GL_2$. We recall its definition. We pick the standard maximal connected compact subgroup $\gp{K} = \sideset{}{_v} \prod \gp{K}_v$ of $\GL_2(\A)$ by
	$$ \gp{K}_v = \left\{ \begin{matrix} \SO_2(\R) & \text{if } \F_v = \R \\ \GL_2(\vo_{\vp}) & \text{if } v = \vp < \infty \end{matrix} \right. , $$
and equip it with the Haar probability measure $d\kappa_v$. We define the following one-parameter algebraic subgroups of $\GL_2(\F_v)$
	$$ \gp{Z}_v = \gp{Z}(\F_v) = \left\{ z(u) := \begin{pmatrix} u & 0 \\ 0 & u \end{pmatrix} \ \middle| \ u \in \F_v^{\times} \right\}, $$
	$$ \gp{N}_v = \gp{N}(\F_v) = \left\{ n(x) := \begin{pmatrix} 1 & x \\ 0 & 1 \end{pmatrix} \ \middle| \ x \in \F_v \right\}, $$
	$$ \gp{A}_v = \gp{A}(\F_v) = \left\{ a(y) := \begin{pmatrix} y & 0 \\ 0 & 1 \end{pmatrix} \ \middle| \ y \in \F_v^{\times} \right\}, $$
and equip them with the Haar measures on $\F_v^{\times}, \F_v, \F_v^{\times}$ respectively. The hyperbolic Haar measure $dg_v$ on $\GL_2(\F_v)$ is the push-forward of the product measure $d^{\times}u \cdot dx \cdot d^{\times}y / \norm[y]_v \cdot d\kappa_v$ under the Iwasawa decomposition map
	$$ \gp{Z}_v \times \gp{N}_v \times \gp{A}_v \times \gp{K}_v \to \GL_2(\F_v), \quad (z(u), n(x), a(y), \kappa) \mapsto z(u) n(x) a(y) \kappa. $$
	At every place $v$, we define a \emph{height function} 
	$$ \Ht_v: \gp{Z}(\F_v) \backslash \GL_2(\F_v) / \gp{K}_v \to \R_{>0}, \quad \begin{pmatrix} t_1 & x \\ 0 & t_2 \end{pmatrix} \kappa \mapsto \extnorm{\frac{t_1}{t_2}}_v. $$
	Their tensor product $\Ht(g) := \sideset{}{_v} \prod \Ht_v(g_v)$ for $g = (g_v)_v \in \GL_2(\A)$ is the height function on $\GL_2(\A)$.
	
\begin{remark}
	Note that $\Vol(\GL_2(\vo_{\vp})) = D_{\vp}^{-\frac{3}{2}}$ for this measure.
\label{MaxCpMeas}
\end{remark}

\noindent Similarly, the hyperbolic Haar measure $d\bar{g}_v$ on $\PGL_2(\F_v)$ is the push-forward of the product measure $dx \cdot d^{\times}y / \norm[y]_v \cdot d\kappa_v$ under the composition map
	$$ \gp{N}_v \times \gp{A}_v \times \gp{K}_v \to \GL_2(\F_v) \to \PGL_2(\F_v), \quad (n(x), a(y), \kappa) \mapsto [n(x) a(y) \kappa]. $$
	We then define and equip the quotient space
	$$ [\PGL_2] := \gp{Z}(\A) \GL_2(\F) \backslash \GL_2(\A) = \PGL_2(\F) \backslash \PGL_2(\A) $$
with the product measure $d\bar{g} := \sideset{}{_v} \prod d\bar{g}_v$ on $\PGL_2(\A)$ quotient by the discrete measure on $\PGL_2(\F)$. 

	
	Consider $\F_v \in \{ \R,\C \}$. On $\PGL_2(\F_v)$, we define a norm by
	$$ \extNorm{\begin{pmatrix} x_1 & x_2 \\ x_3 & x_4 \end{pmatrix}} := \frac{\sum_{i=1}^4 \norm[x_i]^2}{\norm[x_1x_4-x_2x_3]} + \frac{\norm[x_1x_4-x_2x_3]}{\sum_{i=1}^4 \norm[x_i]^2}. $$
	Recall (see \cite[\S 7.1.2]{Wal88}) that a Schwartz function $\phi \in \Sch(\PGL_2(\F_v))$ is a smooth one such that for any $X, Y$ in the enveloping algebra of the complexified Lie algebra of $\GL_2(\F_v)$ and any $r \geq 0$
	$$ \sup_{g \in \PGL_2(\F_v)} \Norm[g]^r \extnorm{ \rpL(X) \rpR(Y) \phi (g) } < +\infty, $$
	where $\rpL$ and $\rpR$ are the left and right translation of $\PGL_2(\F_v)$ respectively. We define $\Sch(\PGL_2(\A_{\infty})) = \widehat{\otimes}_{v \mid \infty} \Sch(\PGL_2(\F_v))$ to be the topological completion of the tensor product. 
	
	We denote by $\RamC$ any constant towards the Ramanujan-Petersson conjecture for $\GL_2$. The current record $\RamC = 7/64$ is due to Kim-Sarnak \cite{KS02} over $\Q$ and Blomer-Brumley \cite{BB11} over general number fields.

\section{A Spectral Reciprocity for Trace Formula}

	\subsection{Spectral Decomposition and Fourier Inversion}
	
	In automorphic representation theory, the terminology ``spectral decomposition'' seems to have different meaning for different authors. In a former work \cite[\S 1.2 \& 1.3]{W17}, the third author tried to clarify this terminology as well as its difference with ``Fourier inversion''. In this subsection, we continue the discussion on this topic.
	
	Let $\gp{G}$ be a locally compact group, whose unitary dual $\widehat{\gp{G}}$ is equipped with the Fell topology. Let $(\rpR,V_{\rpR})$ be a unitary representation of $\gp{G}$ in a Hilbert space $V_{\rpR}$ with inner product $\Pairing{\cdot}{\cdot}_{\rpR}$. The \emph{spectral decomposition of the pairing} $\Pairing{\cdot}{\cdot}_{\rpR}$ or the \emph{Plancherel formula} is to find and establish:
	
\begin{itemize}
	\item[(1)] A Borel measure $d\mu=d\mu_{\rpR}$ on $\widehat{\gp{G}}$, called the \emph{Plancherel measure} for $\rpR$;
	\item[(2)] For each $\pi \in \widehat{\gp{G}}$ in the support of $d\mu$, a $\gp{G}$-intertwiner $\ProjF_{\pi}: V_{\rpR} \to V_{\pi}$ called the \emph{spectral projectors}, where $V_{\pi}$ equipped with its canonical norm $\Pairing{\cdot}{\cdot}_{\pi}$ is a (possibly infinite) direct sum of $\pi$, called the \emph{$\pi$-isotypic component of $\rpR$};
	\item[(3)] For every $v \in V_{\rpR}$, $\ProjF_{\pi}(v)$ is well-defined for $\pi$ outside a set with $d\mu$-measure $0$, and we have 
\begin{equation} 
	\Pairing{v_1}{v_2}_{\rpR} = \int_{\widehat{\gp{G}}} \Pairing{\ProjF_{\pi}(v_1)}{\ProjF_{\pi}(v_2)}_{\pi} d\mu(\pi). 
\label{PlancherelF}
\end{equation}
\end{itemize}
	
\noindent The pairing $\Pairing{\cdot}{\cdot}_{\rpR}$ is the natural one between $V_{\rpR}$ and its topological dual. For any (dense) subspace $V \subset V_{\rpR}$ with a finer topology (for example $V=V_{\rpR}^{\infty}$ can be the subspace of smooth vectors if $\gp{G}$ is a Lie group), and any functional $\ell \in V^{\vee}$ in the topological dual of $V$, one may ask for an analogue of (\ref{PlancherelF}) for the natural pairing $\Pairing{v}{\ell}$ for $v \in V$. In this case, we call the obvious adaptation of the above formalism of Plancherel formula leading to (\ref{PlancherelF}) the \emph{spectral decomposition of the pairing $\Pairing{v}{\ell}$ for $v \in V$}.

	Let $X$ be a topological space on which $\gp{G}$ acts from the right. Let $dx$ be a $\gp{G}$-invariant measure on $X$. We are now interested in the special case where $\rpR$ is the right translation on $V_{\rpR}=\intL^2(X,dx)$. In this special case, people are interested in the \emph{existence} of the adjoint operator $\ProjF_{\pi}^*$ of $\ProjF_{\pi}$ \emph{formally} satisfying
	$$ \Pairing{\ProjF_{\pi}^* \ProjF_{\pi}(v_1)}{v_2}_{\rpR} = \Pairing{\ProjF_{\pi}(v_1)}{\ProjF_{\pi}(v_2)}_{\pi}. $$
	If $d\mu$ is discrete at $\pi$, namely if $\pi$ appears as subrepresentations of $\rpR$, this adjoint operator $\ProjF_{\pi}^*$ may be simply regarded as the inclusion map. In general, both the domain and the image of $\ProjF_{\pi}^*$ are subtle questions. In practice, both $X$ and $\gp{G}$ often admit some smooth structures with which the $\gp{G}$ action is compatible. The smooth structure of $\gp{G}$ allows us to define subspace of smooth vectors $V_{\rpR}^{\infty}$ or $V_{\pi}^{\infty}$. A suitable version of Dixmier-Malliavin's theorem would imply $\ProjF_{\pi}(V_{\rpR}^{\infty}) \subset V_{\pi}^{\infty}$.
	
\begin{example}
	If $\gp{G}$ is a Lie group, we have differential operators associated to the Lie algebra of $\gp{G}$ and the meaning of a smooth vector is clear. If $\gp{G}$ is totally disconnected, a smooth vector admits an open subgroup of $\gp{G}$ as its stabilizer group.
\end{example}

\noindent During the establishment of (\ref{PlancherelF}), we naturally already obtain a version of 
	$$ \ProjF_{\pi}^*: V_{\pi}^{\infty} \to \Cont^{\infty}(X), $$
	as well as the validity of the formula
\begin{equation}
	v(x) = \int_{\widehat{\gp{G}}} \left(\ProjF_{\pi}^* \ProjF_{\pi} v\right)(x) d\mu(\pi), \quad v \in V_0
\label{FourInvF}
\end{equation}
	for a sufficiently large subspace $V_0 \subset V_{\rpR}^{\infty}$ with some strong convergence for $x \in X$ (normal or uniform), and possibly an estimation of the dominant of the right hand side as $x \in X$ varies:
\begin{equation}
	\int_{\widehat{\gp{G}}} \extnorm{\left(\ProjF_{\pi}^* \ProjF_{\pi} v\right)(x)} d\mu(\pi).
\label{DomFourInvF}
\end{equation} 
	
\begin{example}
	For $\gp{G}=\R$ and $X = \Z \backslash \R$ (resp. $\R$), we are in the setting of the classical Fourier analysis. The dual $\widehat{\R}=\R$, hence we may parametrize $(\pi,V_{\pi}) = (\pi_t, \C e_t)$ with $t \in \R$ and $\pi_t(x)e_t = e^{2\pi i xt} e_t$. We have $\ProjF_{\pi_t}^*(e_t)(x)=e^{2\pi i t x}$. We may take $V_0=\Cont^{\infty}(\Z \backslash \R)$ (resp. $V_0=\Sch(\R)$ the space of Schwartz functions, or its subspace of Gaussian type functions $P(x)e^{-ax^2}$ for polynomials $P$ and $a > 0$).
\end{example}

\begin{definition}
	We call the formula (\ref{FourInvF}) together with an estimation of the dominant integral (\ref{DomFourInvF}), as well as any extension of their validity to a larger subspace $V_0 \subset V \subset V_{\rpR}$, a \emph{Fourier inversion formula}. If $X = \gp{G}(\F) \backslash \gp{G}(\A)$ is the underlying space for automorphic forms for a reductive group $\gp{G}$, we also call (\ref{FourInvF})+(\ref{DomFourInvF}) an \emph{automorphic Fourier inversion} for $V$.
\end{definition}

\begin{remark}
	The formula (\ref{FourInvF}) for $V$ alone is also the spectral decomposition of the pairing $\Pairing{v}{\delta_x}$ for $v \in V$ and the Dirac mass $\delta_x$. This situation holds a particular status, because the spectral decomposition of many other interesting pairings, for example pairings making use of the Gan-Gross-Prasad conjectures in the automorphic setting, would be established based on applying this one together with Fubini's theorem, for which the estimation of (\ref{DomFourInvF}) is important. For this reason, we have added the estimation of (\ref{DomFourInvF}) in the definition of a Fourier inversion to distinguish it from spectral decomposition.
\end{remark}

\begin{remark}
	It seems that the spectral decomposition in the automorphic setting considered in the work of Moeglin-Waldspurger \cite[\S \Rmnum{6}]{MW95} only treats the Plancherel formula (\ref{PlancherelF}). In another source \cite[Theorem 1.1]{CP90}, where the expression ``spectral decomposition'' appeared explicitly, it seems also to only concern (\ref{PlancherelF}), although the subsequent \cite[Proposition 1.4]{CP90} looks like an attempt of establishing an automorphic Fourier inversion (\ref{FourInvF}) without a clear mention of $\ProjF_{\pi}^*$ and with an incomplete proof (see \cite[\S 4]{W17} for more details). Moreover, one can not say, as in the proof of \cite[Proposition 1.4]{CP90} that (\ref{FourInvF}) holds in the sense of $\intL^2$ for $v \in V_{\rpR}^{\infty}$ directly from (\ref{PlancherelF}), because one does not \emph{a priori} know if (partial integral version of) the right hand side of (\ref{FourInvF}) lies in $\intL^2(X,dx)$.
\end{remark}

\begin{remark}
	There are other authors who understand ``spectral decomposition'' as the spectral resolution of the hyperbolic Laplacian in the classical setting of Modular forms on the upper half plane (see \cite[\S 4 \& 7]{Iw02}). The statement of the relevant theorems does correspond to what we mean by automorphic Fourier inversion here. But the viewpoint does not suit well for the representation theory. For example, the decomposition according to the Hecke operators, i.e., the finite components of an automorphic representation, is not reflected in this viewpoint.
\end{remark}

\begin{remark}
	Our terminology ``automorphic Fourier inversion'' also reminds us that the state-of-art of such formulas in the automorphic setting has not reached the same depth as in the classical Fourier analysis in $\intL^2(\R)$. Hence problems towards this direction are open.
\end{remark}

    \subsection{Automorphic Fourier Inversion and Pre-trace Formula}
    
    The third author studied the automorphic Fourier inversion for $\GL_2$ over general number fields in the former works \cite[\S 2.6]{W14}, \cite{W17} and \cite[Theorem 2.3]{Wu11}. We recall the special case for $\PGL_2$. Let $(\rpR,V_{\rpR}) = (\rpR, \intL^2([\PGL_2], d\bar{g}))$ be the right regular representation.
    
\begin{definition}
	Let $\varphi \in V_{\rpR}^{\infty}$ be a smooth vector represented by a smooth function on $\GL_2(\A)$. If for any $X$ in the universal enveloping algebra of the Lie algebra of $\GL_2(\A_{\infty})$, we have
	$$ \extnorm{ \rpR(X).\varphi(g) } \ll_{\epsilon} \Ht(g)^{1/2-\epsilon} $$
	for any $\epsilon > 0$ sufficiently small, uniformly in $g$ lying in a/any \emph{Siegel domain}, then we call $\varphi$ ``nice''.
\label{NiceFDef}
\end{definition}

\begin{theorem}{(\emph{automorphic Fourier inversion formula})}
	For ``nice'' $\varphi \in V_{\rpR}^{\infty}$, we have a decomposition 
\begin{align*}
	\varphi(g) &= \sideset{}{_{\pi \text{ cuspidal}}} \sum \sideset{}{_{e \in \Bas(\pi)}} \sum \Pairing{\varphi}{e} e(g) \\
	&+ \sum_{\chi \in \widehat{\R_+ \F^{\times} \backslash \A^{\times}}} \sum_{f \in \Bas(\chi,\chi^{-1})} \int_{-\infty}^{\infty} \Pairing{\varphi}{\eis(i\tau,f)} \eis(i\tau,f)(g) \frac{d\tau}{4\pi} \\
	&+ \frac{1}{\Vol([\PGL_2])} \sideset{}{_{\substack{ \chi \in \widehat{\F^{\times} \backslash \A^{\times}} \\ \chi^2 = 1 }}} \sum \int_{[\PGL_2]} \varphi(x) \overline{\chi(\det x)} dx \cdot \chi(\det g)
\end{align*}
	with normal convergence in $ [\PGL_2] $. Here, $\Bas(\pi)$ (resp. $\Bas(\chi,\chi^{-1})$) is an orthonormal basis of $\gp{K}_{\infty}$-isotypic and $\gp{K}_{\fin}$-finite vectors of $V_{\pi}$ (resp. $V_{\chi,\chi^{-1}}$) and the \emph{automorphic Fourier coefficients}, namely the above terms written via inner product such as $\Pairing{\varphi}{\eis(i\tau,f)}$, are given by the usual convergent integrals. Moreover, the dominant of the cuspidal part, namely 
	$$ \sideset{}{_{\pi \text{ cuspidal}}} \sum \sideset{}{_{e \in \Bas(\pi)}} \sum \extnorm{\Pairing{\varphi}{e} e(g)} $$
	uniformly converges in any Siegel domain and is of rapid decay with respect to $\Ht(g) \to \infty$.
\label{AutoFourInv}
\end{theorem}

	Let $\Norm[x]$ for $x \in [\PGL_2]$ be a metric which defines the Schwartz function space $\Sch([\PGL_2])$ in the sense of Wallach. For example, define $\Norm[\PGL_2(\F)g] = \Ht(g)$ for $g \in \PGL_2(\A)$ lying in a fixed fundamental domain contained in a standard Siegel domain (as small as possible) and let
    $$ \Sch([\PGL_2]) = \left\{ f \in \Cont^{\infty}([\PGL_2]) \ \middle| \ \forall N \in \N, X \in U(\mathfrak{g}), \extnorm{\rpR(X).f(g)} \ll_N \Ht(g)^{-N} \right\}, $$
    where $U(\mathfrak{g})$ is the univsersal enveloping algebra of the complexified Lie algebra of $\PGL_2(\A_{\infty})$. Then any $\varphi \in \Sch([\PGL_2])$ is obviously ``nice''. Let $f \in \Sch(\PGL_2(\A)) := \Sch(\PGL_2(\A_{\infty})) \otimes \Cont_c^{\infty}(\PGL_2(\A_{\fin}))$ and form the automorphic kernel function $K=K_f : [\PGL_2]^2 \to \C$ associated with $f$
\begin{equation}
	K(x,y) := \sum_{\gamma \in \PGL_2(\F)} f(x^{-1} \gamma y).
\label{AutoKern}
\end{equation} 
    Then for a fixed variable, say $x$, the function in $y$ of $K(x,y)$ is in $\Sch([\PGL_2])$. We may apply Theorem \ref{AutoFourInv} and obtain its automorphic Fourier inversion

\begin{align} 
    K(x,y) &= \sum_{\pi \text{ cuspidal}} \sum_{\phi \in \Bas(\pi)} \rpR(f)\phi(x) \overline{\phi(y)} + \label{AutoKernSpecD} \\
    &\quad \sum_{\chi \in \widehat{\R_+ \F^{\times} \backslash \A^{\times}}} \sum_{e \in \Bas(\chi,\chi^{-1})} \int_{-\infty}^{\infty} \rpR(f)\eis(i\tau,e)(x) \overline{\eis(i\tau,e)(y)} \frac{d\tau}{4\pi} + \nonumber \\
	&\quad \frac{1}{\Vol([\PGL_2])} \sideset{}{_{\substack{ \chi \in \widehat{\F^{\times} \backslash \A^{\times}} \\ \chi^2 = 1 }}} \sum \int_{\PGL_2(\A)} f(g) \chi(\det g) dg \cdot \chi(\det x) \overline{\chi(\det y)}. \nonumber
\end{align}

\begin{theorem}
    For $f \in \Sch(\PGL_2(\A))$ with the automorphic kernel function $K(x,y)$ defined by (\ref{AutoKern}), the equation (\ref{AutoKernSpecD}) holds with normal convergence in $(x,y) \in [\PGL_2]^2$. Moreover, the cuspidal part of $K(x,y)$
    $$ K_0(x,y) := \sum_{\pi} \sum_{\phi \in \Bas(\pi)} \rpR(f)\phi(x) \overline{\phi(y)} $$
    as well as its dominant are rapidly decreasing functions in $(x,y) \in [\PGL_2]^2$, and the right hand side converges absolutely and uniformly in any Siegel domain.
\label{AutoKernSpecDSch}
\end{theorem}
\begin{proof}
	Applying Theorem \ref{AutoFourInv} to $\rpR(f)\phi$, we may rewrite, first formally,
	$$ K_0(x,y) = \sum_{\pi} \sum_{\phi_1, \phi_2 \in \Bas(\pi)} \Pairing{\rpR(f)\phi_2}{\phi_1} \phi_1(x) \overline{\phi_2(y)}. $$
	The proof of Theorem \ref{AutoFourInv} given in \cite[\S 2.6]{W14} is based on some local Sobolev type estimation of the Whittaker functions, which implies that for any $N \in \Z_{\geq 0}$ there is $d=d(N) \in \Z_{\geq 0}$ such that
	$$ \norm[\phi_j(x)] \ll_N \lambda_{\phi_j,\infty}^d \Ht(x)^{-N}, $$
	where $\lambda_{\phi_j,\infty}$ is the eigenvalue of $\phi_j$ with respect to the elliptic operator 
	$$ \Delta_{\infty} = \sideset{}{_{v \mid \infty}} \prod (-\mathcal{C}_{\SL_2(\F_v)} - 2 \mathcal{C}_{\gp{K}_v}), $$
	and $\mathcal{C}_*$ is the Casimir element of the Lie group $*$. It follows that for any large and fixed $d_0 \in \Z_{\geq 0}$
	$$ \extnorm{\Pairing{\rpR(f)\phi_2}{\phi_1} \phi_1(x) \overline{\phi_2(y)}} \ll_N \int_{\PGL_2(\A)} \extnorm{\rpL(\Delta_{\infty}^{d+d_0}) \rpR(\Delta_{\infty}^{d+d_0}) f (g)} dg \cdot \lambda_{\phi_1,\infty}^{-d_0} \lambda_{\phi_2,\infty}^{-d_0} \cdot \Ht(x)^{-N} \Ht(y)^{-N}. $$
	Summing the right hand side and applying a (weak version of) Weyl's law, we justify the previous formal decomposition and conclude the desired rapid decay property.
\end{proof}

\begin{remark}
	$K(x,y)$ is not of rapid decay in the two variables $x$ and $y$, since 
	$$ \int_{[\PGL_2]} K(x,y) \eis(i\tau,e)(y) dy = \rpR(f) \eis(i\tau,e)(x) $$
	is not of rapid decay. Hence the decomposition (\ref{AutoKernSpecD}) is not an automorphic Fourier inversion in two variables. Thus Theorem \ref{AutoKernSpecDSch} is not a consequence of a two variable version of Theorem \ref{AutoFourInv}.
\end{remark}

\begin{remark}
	We need Theorem \ref{AutoKernSpecDSch} for the validity with Schwartz test functions of \emph{relative} trace formulae such as Jacquet-Zagier's trace formula. This should not be confused with the extension of the class of test functions in the trace formulae obtained by Finis, Lapid, M\"uller, Hoffman etc., such as \cite{FL11}. In particular, those results for trace formulae do not imply our pre-trace formula in Theorem \ref{AutoKernSpecDSch}, and can not be applied to our possible future work on the first moment of $L(1,\pi,\mathrm{Ad})$.
\end{remark}

\begin{remark}
	Careful readers may notice that the proof of automorphic Fourier inversion needs a weak version of the Weyl's law. If $f \in \Cont_c^{\infty}(\PGL_2(\A))$ is $\gp{K}$-finite for both left and right multiplication, then the decomposition of the associated kernel function $K(x,y)$ given in (\ref{AutoKernSpecD}) has another proof independent of Theorem \ref{AutoFourInv} or its $\gp{K}$-finite variant. This proof exploits a quantitative variant of the Dixmier-Malliavin's theorem due to Duflo-Labesse \cite[(\Rmnum{1}.1.11)]{DL71} and kernel functions associated with test functions of positive type. For a detailed treatment, see \cite[\S 6]{KL13}. Consequently, one can apply trace formulas for bi-$\gp{K}$-finite test functions in $\Cont_c^{\infty}(\PGL_2(\A))$, with extension to heat kernel to obtain the Weyl's law (see \cite{Pa12}). Hence invoking the Weyl's law in any automorphic Fourier inversion is not a circular reasoning. Alternatively, one may also use Donelly's work \cite{Do81, Do82} to establish a weak version of the Weyl's law, which does not involve trace formulae and suffices for the automorphic Fourier inversion.\footnote{We owe this clarification to Professor Werner M\"uller.}
\end{remark}

\subsection{Jacquet-Zagier's Trace Formula for $\mathrm{GL}_2$}

    Let $f \in \Sch(\PGL_2(\A))$ to which we associate the automorphic kernel function $K(x,y)$ by (\ref{AutoKern}) with its cuspidal part $K_0(x,y)$ given in Theorem \ref{AutoKernSpecDSch}. Let $\Phi \in \Sch(\A^2)$ and form the associated Eisenstein series as
	$$ \eis_{\Phi}(s,g) := \sideset{}{_{\gamma \in \gp{B}(\F) \backslash \GL_2(\F)}} \sum f_{\Phi}(s, \gamma g), $$
	where the Godement section is defined by
	$$ f_{\Phi}(s, g) := \norm[\det g]_{\A}^{1/2+s} \int_{\A^{\times}} \Phi((0,t)g) \norm[t]_{\A}^{1+2s} d^{\times}t. $$
	Jacquet-Zagier \cite{Z81, JZ87} considered the equality of two expansions of the following integral (absolutely convergent for any $s \in \C$ by Theorem \ref{AutoKernSpecDSch})
\begin{equation}
	I(s) = I(s; f, \Phi) := \int_{[\PGL_2]} K_0(x,x) \eis_{\Phi}(s, x) dx.
\label{MainIntDef}
\end{equation}
    One expansion is obtained by inserting the spectral expansion of $K_0(x,y)$, and leads to a sum of integral representations of the symmetric/adjoint square $L$-functions for $\PGL_2$. We shall call it the $\GL_2$ side. The other expansion is obtained by rearranging terms in
    $$ K_0(x,y) = K(x,y) - \textrm{(CSC)}, $$
    in a way different from the one leading to the geometric side of the Arthur-Selberg trace formula. In particular, they avoid Arthur's truncation and find an expression
\begin{equation} 
	I(s) = \sideset{}{_{\E}} \sum I_{\E}(s) + I_1(s) + I_2(s) + I_{\infty}'(s) + I_{\infty}''(s) + I_{\infty}'''(s), 
\label{JZGeom}
\end{equation}
    where the sum is over quadratic field extensions $\E/\F$. We recall the terms $I_{\E}(s)$ as follows. For every quadratic field extension $\E/\F$ we choose an embedding $\E \hookrightarrow \Mat_2(\F)$ of $\F$-algebras. We identify $\E^{\times}$ as a torus in $\GL_2$ defined over $\F$, whose adelic points is denoted by $\E_{\A}^{\times}$. Then
	$$ \frac{I_{\E}(s)}{\Lambda_{\E}(1/2+s)} = \frac{1}{2} \int_{\E_{\A}^{\times} \backslash \GL_2(\A)} \left( \sum_{1 \neq \lambda \in \E^{\times} / \F^{\times}} f(x^{-1} \lambda x) \right) \cdot \frac{\int_{\E_{\A}^{\times}} \Phi((0,1)tx) \norm[\det tx]_{\A}^{1/2+s} d^{\times}t }{\Lambda_{\E}(1/2+s)} dx. $$
	Since $t \mapsto \Phi((0,1)tx)$ is easily identified as a Schwartz function on $\A_{\E}$, $I_{\E}(s)$ is an integral representation of the complete Dedekind zeta function $\Lambda_{\E}(1/2+s)$. For this reason, we shall call this expansion the $\GL_1$ side. The equality of the two sides is an instance of \emph{spectral reciprocity} in the sense of \cite{BK19}.

\begin{remark}
    One of the original goals of Jacquet-Zagier is to deduce the Selberg trace formula from their equation (\ref{JZGeom}) by taking the residue at $s=1/2$. To identify this residue of the $\GL_1$ side with the geometric side of the trace formula is not at all a trivial task. In fact, it was even incomplete in \cite{JZ87}, which was completed only recently in \cite{W19}. For this reason, we prefer to replace the terminology ``geometric side'' with ``$\GL_1$ side''. In view of the difficulties in applications of the trace formula to concrete problems, it is unclear why the trace formula should be advantageous over Jacquet-Zagier's basic equation (\ref{JZGeom}).
\end{remark}

    The other terms in the $\GL_1$ side are irrelevant to the bias problem for cubic level. In fact, they turn out to be vanishing due to the following lemma.
    
\begin{lemma}
    Let $f_{\vp}$ be the component at a finite place $\vp$ of $f \in \Sch(\PGL_2(\A))$.
\begin{itemize}
    \item[(1)] If $f_{\vp}$ is in the space spanned by matrix coefficients of supercuspidal representations, then $I_{\infty}''(s)$ and $I_{\infty}'''(s)$ vanish identically.
    \item[(2)] If the support of $f_{\vp}$ is contained in the subset of $\F_{\vp}$-elliptic elements of $\PGL_2(\F_{\vp})$, then $I_1(s), I_2(s)$ and $I_{\infty}'(s)$ vanish identically.
\end{itemize}
\label{SimpleRSF}
\end{lemma}
\begin{proof}
    (1) $I_{\infty}'''(s)$ has a decomposition as
	$$ I_{\infty}'''(s) = -\sum_{\chi \in \widehat{\F^{\times} \R_+ \backslash \A^{\times}}} \frac{1}{2\pi} \int_0^{\infty} I_{\chi}(s,i\tau) d\tau, $$
	$$ I_{\chi}(s,i\tau) = \sum_{e \in \Bas_{i\tau}(\chi, \chi^{-1})} \int_{\gp{Z}(\A) \gp{N}(\A) \backslash \GL_2(\A)} \pi_{i\tau}(\chi, \omega \chi^{-1})(f)(W_e)(x) \overline{W_e(x)} f_{\Phi}(s,x) dx, $$
	where $\Bas_{i\tau}(\chi, \chi^{-1})$ is an orthonormal basis in the induced model of the principal series representation $\pi_{i\tau}(\chi, \chi^{-1})$ and $W_e$ is the Whittaker function of $e$. By \cite[Corollary 10.29]{KL72}, $\pi_{i\tau}(\chi_{\vp}, \chi_{\vp}^{-1})(f_{\vp}) = 0$, whence the vanishing of $I_{\infty}'''(s)$. $I_{\infty}''(s)$ is $-1/2$ times the sum over quadratic Hecke characters $\chi$ of
	$$ \int_{\gp{N}(\A)} \int_{\A^{\times}} \int_{\gp{K}} \int_{\gp{N}(\A)} f \left( \kappa^{-1} n_1 \begin{pmatrix} a^{-1} & 0 \\ 0 & 1 \end{pmatrix} wn_2 \kappa \right) \chi(a) \norm[a]_{\A}^{3/4+s/2} f_{\Phi}(s,\kappa) dn_1 d\kappa d^{\times}a dn_2. $$
	Under the assumption, we also have for any $g,h \in \GL_2(\F_{\vp})$,
	$$ \int_{\F_{\vp}} f_{\vp} \left( g \begin{pmatrix} 1 & u \\ 0 & 1 \end{pmatrix} h \right) du = 0, $$
	since the integral over any unipotent subgroup is vanishing on supercuspidal representations (see \cite[Proposition 4.4.1]{Bu98}), so is it for any linear combination of their matrix coefficients. Thus the inner most integral is identically vanishing, so is $I_{\infty}''(s)$.
	
\noindent (2) Observing the formulae
    $$ I_1(s) = \frac{1}{2} \int_{\gp{A}(\A) \backslash \GL_2(\A)} \left( \sum_{1 \neq \alpha \in \F^{\times}} f \left( x^{-1} \begin{pmatrix} \alpha & 0 \\ 0 & 1 \end{pmatrix} x \right) \right) \cdot \int_{\gp{A}(\A)} \Phi((1,1)ax) \norm[\det ax]_{\A}^{1/2+s} d^{\times}a dx, $$
    $$ I_2(s) = \int_{\gp{K}} \int_{\A^{\times}} f \left( \kappa^{-1} \begin{pmatrix} 1 & a \\ 0 & 1 \end{pmatrix} \kappa \right) \norm[a]_{\A}^{1/2+s} d^{\times}a \cdot \int_{\A^{\times}} \left( \int_{\A} \Phi((t,u)\kappa) du \right) \norm[t]_{\A}^{2s} d^{\times}t d\kappa, $$
    $$ I_{\infty}'(s) = \int_{\gp{K}} f_{\Phi}(s,\kappa) \cdot \int_{\A^{\times}} \left( \int_{\A} f \left( \kappa^{-1} \begin{pmatrix} 1 & u \\ 0 & 1 \end{pmatrix} \kappa \right) \psi(au) du \right) \norm[a]_{\A}^{1/2+s} d^{\times}a d\kappa, $$
    we see that they are distributions over $\GL_2(\F_{\vp})$ supported in the subset of split, resp. unipotent, resp. unipotent elements. Hence they are vanishing under the assumption on $f_{\vp}$.
\end{proof}

\begin{remark}
	If a test function $f \in \Sch(\PGL_2(\A))$ has (not necessarily different) two finite components $f_{\vp_1}$ and $f_{\vp_2}$ which satisfy the two conditions in Lemma \ref{SimpleRSF} respectively (this is the case for cubic level but not for square-free level in the bias problem), then we may also argue without referring to the concrete expressions of terms in Jacquet-Zagier's formula. Namely, by the condition in (1), we have $K_0(x,y) = K(x,y)$; while by the condition in (2), we have
	$$ K(x,x) = \sum_{\gamma \in \PGL_2(\F)^{\mathrm{ell}}} f(x^{-1} \gamma x) $$
	where $\PGL_2(\F)^{\mathrm{ell}}$ is the subset of $\F$-elliptic elements in $\PGL_2(\F)$. Thus we directly deduce the desired simplification
	$$ I(s) = \int_{[\PGL_2]} K(x,x) \eis_{\Phi}(s,x) dx = \sum_{\E} I_{\E}(s) $$
	by Rankin-Selberg unfolding.
\end{remark}

\section{Computation of Spectral Side}

	\subsection{Choice of Test Functions}
	
	Recall that we are considering a totally real number field $\F$ of degree $d$ with absolute different $\idl{D}_{\F}$ and absolute discriminant $D_{\F} = \Nr_{\Q}^{\F}(\idl{D}_{\F})$. For each place $v \mid \infty$, we have chosen an integer $k_v \geq 1$. We have taken a square-free integral ideal $\idl{N}$. Write $\vec{k} = (k_v)_{v \mid \infty} \in \N^d$. Let $\mathcal{A}(\vec{k},\idl{N}^3)$ be the set of cuspidal representations $\pi$ such that $\pi_v \simeq D_{2k_v}$ for any $v \mid \infty$, and that $\pi_{\vp}$ has conductor $\vp^3$ at $\vp \mid \idl{N}$ and is unramified at $\vp \nmid \idl{N}$. We now specify the test functions $f$ and $\Phi$ in (\ref{MainIntDef}) for the bias problem in this paper.
	
	We choose $\Phi$ to be the standard $\gp{K}$-invariant Schwartz function, that is
\begin{equation}
	\Phi(x,y) = \sideset{}{_v} \prod \Phi_v(x_v, y_v), \quad \Phi_v(x,y) := \left\{ \begin{matrix} e^{-\pi (x^2 + y^2)} & \F_v = \R \\ 1_{\vo_{\vp}}(x) \cdot 1_{\vo_{\vp}}(y) & v=\vp < \infty \end{matrix} \right. .
\label{SphPhi}
\end{equation}
\begin{lemma}
	The spherical Eisenstein series $\eis_{\Phi}(s,x)$ has a constant residue at $s=1/2$ given by
	$$ \Res_{s=1/2} \eis_{\Phi}(s,x) = \frac{\Lambda_{\F}^{(-1)}(1)}{2 D_{\F}}, $$
	where $\Lambda_{\F}(s)$ is the Dedekind zeta function of $\F$ and $\Lambda_{\F}^{(-1)}(x)$ is its residue at $s=x$.
\label{SphEisRes}
\end{lemma}
\begin{proof}
	The pole of $\eis_{\Phi}(s,x)$ is controlled by its constant term, which is easily identified as
	$$ f_{\Phi}(s,x) + f_{\widehat{\Phi}}(-s,x), $$
where $\widehat{\Phi}$ is the twisted Fourier transform defined by
	$$ \widehat{\Phi}(x,y) := \int_{\A^2} \Phi(u,v) \psi\left( - (u,v) w \begin{pmatrix} x \\ y \end{pmatrix} \right) dudv, \quad w := \begin{pmatrix} 0 & -1 \\ 1 & 0 \end{pmatrix}. $$
	For our choice of $\Phi$ (\ref{SphPhi}), it is easy to see
	$$ \widehat{\Phi}(x,y) = \sideset{}{_v} \prod \widehat{\Phi_v}(x_v, y_v), \quad \widehat{\Phi_v}(x,y) := \left\{ \begin{matrix} e^{-\pi (x^2 + y^2)} & \F_v = \R \\ D_{\vp}^{-1} \cdot 1_{\idl{D}_{\vp}^{-1}}(x) \cdot 1_{\idl{D}_{\vp}^{-1}}(y) & v=\vp < \infty \end{matrix} \right. , $$
	where $\idl{D}_{\vp}$ is the component of $\idl{D}_{\F}$ at $\vp$, and $D_{\vp} = \Nr_{\Q}^{\F}(\idl{D}_{\vp})$. We thus get
	$$ \Res_{s=\frac{1}{2}} \eis_{\Phi}(s,x) = \Res_{s=\frac{1}{2}} f_{\widehat{\Phi}}(-s,1) = \Res_{s=\frac{1}{2}} D_{\F}^{-\frac{1}{2}-2s} \Lambda_{\F}(1-2s) = \Res_{s=\frac{1}{2}} D_{\F}^{-1} \Lambda_{\F}(2s) = \frac{\Lambda_{\F}^{(-1)}(1)}{2 D_{\F}}. $$
\end{proof}
	
	We specify $f = \otimes_v' f_v$ place-by-place as follows:
	
\noindent At a real place $v \mid \infty$, we introduce  a non-zero fractional linear function on $\GL_2(\R)$
\begin{equation}
	J(g) := \frac{\det(g)}{(-b+c+(a+d)i)^{2}}, \quad g = \begin{pmatrix} a & b \\ c & d \end{pmatrix}.
\label{ArchInv}
\end{equation}
	Let $\delta_v > 0$. Take $h_v=h_{\delta_v} \in \Cont^{\infty}(\R_{\geq 0})$ to be $h_v(t) = \exp(-\delta_v t)$. We take the test function of the form
\begin{equation}
	f_v(g) := \left\{ \begin{matrix} \frac{(2k-1)2^{2k-2}}{\pi} \cdot J(g)^k \cdot h_v( \norm[J(g)]^{-1} ) & \text{if } \det(g) > 0 \\ 0 & \text{otherwise} \end{matrix} \right.. 
\label{TFArch}
\end{equation}
	We shall write $\vec{\delta}=(\delta_v)_{v \mid \infty}$. We have the pointwise limit
\begin{equation}
	\lim_{\delta_v \to 0^+} f_v(g) = \varepsilon(1/2,D_{2k},\psi_{\R}) \cdot \Pairing{D_{2k}(g)v_0}{v_0},
\label{TFArchLim}
\end{equation}  
	where $\varepsilon(1/2,D_{2k},\psi_{\R})=(-1)^k$ (see \cite[(3.7)]{Kn94}) and $v_0$ is a unitary lowest weight vector in $D_{2k}$, the discrete series of lowest weight $2k$ (see \cite[(6)]{KL06}). Note that the above limit also holds in $\intL^2(\PGL_2(\R))$ for $k \geq 1$ and in $\intL^1(\PGL_2(\R))$ for $k > 1$.
	
\noindent At a finite place $\vp$, we write $q=\Nr_{\Q}^{\F}(\vp)$ for simplicity. Let $\gp{Z}_{\vp}$ be the center of $\GL_2(\F_{\vp})$. Let $\RI_\Fp$ be the Iwahoric subgroup of $\GL_2(\vo_{\vp})$ consisting of upper triangular elements modulo $\vp$. If $\vp \nmid \idl{N}$, we choose $f_{\vp} = \mathbbm{1}_{\PGL_2(\vo_\vp)}$. If $\vp \mid \idl{N}$, we take $f_{\vp}$ to be supported in $\begin{pmatrix} & 1 \\ \varpi_{\vp} & \end{pmatrix} \cdot 
\gp{Z}_\Fp \cdot \RI_\Fp$, in which it is given by the formula of $\wt{f}^b_\Fp$ in Lemma \ref{lem:app:Iwahoriaveragefb}, i.e.,
\begin{equation}
	f_{\vp} \left( \begin{pmatrix} 0 & 1 \\ \vpi_\Fp & 0 \end{pmatrix} z \begin{pmatrix} x_1 & r_1 \\ \vpi_\Fp r_2 & x_2 \end{pmatrix} \right) = (q+1) \cdot \bigg\{ \begin{matrix} q-1 & r_1+r_2\in \Fp \\ -1 & \text{otherwise}
\end{matrix}.
\label{TFNArch}
\end{equation}
for any $x_1,x_2\in \Fo^\times_\Fp$, $r_1,r_2\in \Fo_\Fp$.
\begin{proposition}
	If $\pi$ is a unitary irreducible representation of $\PGL_2(\F_\vp)$, then $\pi(f_{\vp})$ is a non-zero operator only if the conductor of $\pi$ is $\vp^3$. In this case, we have 
	$$ \Tr(\pi(f_{\vp})) = \Vol(\PGL_2(\vo_\vp)) \cdot \varepsilon \left( \frac{1}{2},\pi,\psi_{\vp} \right). $$
\label{NArchC}
\end{proposition}
\begin{proof}
	By construction, $f_{\vp} = \wt{f}^b_\Fp$ is a normalized average of $f^b_\Fp$ under the action of $\RI_\vp$ by conjugation. Hence it suffices to check the stated equality for $f^b_\Fp$ instead of $\wt{f}^b_\Fp$, for which we refer the reader to Appendix \ref{app:simpsupcusp}, in particular the equation (\ref{eq:app:wholetestfun:projectiontoline}).
\end{proof}

	\subsection{Computation and Estimation at Archimedean Places}
	
	In this subsection we work over a fixed (real) place $v \mid \infty$, hence we abbreviate $k=k_v$ for simplicity. The function $J(g)$ (see (\ref{ArchInv})) satisfies
	$$ J \left( \begin{pmatrix} \cos \theta & \sin \theta \\ -\sin \theta & \cos \theta \end{pmatrix} g \right) = J \left( g \begin{pmatrix} \cos \theta & \sin \theta \\ -\sin \theta & \cos \theta \end{pmatrix} \right) = e^{-2i\theta}J(g). $$
	Hence for any admissible irreducible representation $\pi$ of $\GL_2(\R)$, $\pi(f_v)$ (see (\ref{TFArch})) picks up the (unique if it exists) weight $2k$ vector in $\pi$ and multiplies it by $\Tr(\pi(f_v))$. The goal of this subsection is to give information about $\Tr(\pi(f_v))$ for all possible $\pi$ by essentially computing all of them.
	
\begin{lemma}
	If $k>1$, then $\lim_{\delta \to 0^+} \Tr(\pi(f_v)) = 0$ if $\pi \not\simeq D_{2k}$ and $=(-1)^k$ if $\pi \simeq D_{2k}$.
\end{lemma}	
\begin{proof}
	By (\ref{TFArchLim}) in the sense of $\intL^1(\PGL_2(\R))$, we have $\lim_{\delta \to 0^+} \Tr(\pi(f_v)) = \Tr \left( \lim_{\delta \to 0^+} f_v \right)$. We then conclude by Schur's lemma.
\end{proof}

\begin{definition}
	For integers $k \geq 1$, define
	$$ P_k(s) := \sum_{j=0}^k (-1)^j \binom{2k}{2j} \frac{\Gamma \left( s+k+j-\frac{1}{2} \right) \Gamma \left( k-j+\frac{1}{2} \right)}{\Gamma \left( s+2k \right)}. $$
\label{WtkPoly}
\end{definition}

\begin{lemma}
	Let $\pi \simeq \pi(\norm^{i\tau}, \norm^{-i\tau}) \otimes \sgn^l$ with $l \in \{ 0,1 \}$ and $\tau \in \R \cup (-\RamC,\RamC)i$ be a principal or complementary series representation. Then as $\delta_v > 0$ and $\tau$ vary we have for any $-1/2+\RamC < c < 0$ and any $C \gg 1$
	$$ \Tr(\pi(f_v)) \ll_{c,C} \delta_v^{-c} (1+\norm[\tau])^{-C}. $$
\end{lemma}
\begin{proof}
	By the character formula of principal series \cite[Proposition 7.6]{JL70} and the formula for $f_v$ in (\ref{TFArch}), we have
\begin{align*}
	\Tr(\pi(f_v)) &= \int_{\R^{\times}} \int_{\R} f_v \begin{pmatrix} y & x \\ 0 & 1 \end{pmatrix} \norm[y]^{i\tau-\frac{1}{2}} \sgn^l(y) dx d^{\times}y \\
	&= \frac{(2k-1)2^{2k-2}}{\pi} \int_0^{\infty} \int_{-\infty}^{\infty} \frac{y^{k+i\tau-\frac{1}{2}}}{(-x+(y+1)i)^{2k}} h_{\delta} \left( \frac{x^2+(y+1)^2}{y} \right) dx d^{\times}y.
\end{align*}
	Inserting the Mellin inversion formula $h_{\delta}(t) = \int_{(c)} (\delta t)^{-s} \Gamma(s) \frac{ds}{2\pi i}$ and changing the order of integrations, we get for $c \gg 1$
\begin{align*}
	\Tr(\pi(f_v)) &= \frac{(2k-1)2^{2k-2}}{\pi} \int_0^{\infty} \int_{-\infty}^{\infty} \frac{y^{k+i\tau-\frac{1}{2}}}{(-x+(y+1)i)^{2k}} \int_{(c)} \delta_v^{-s} \extnorm{\frac{x^2+(y+1)^2}{y}}^{-s} \Gamma(s) \frac{ds}{2\pi i} dx d^{\times}y \\
	&= \frac{(2k-1)2^{2k-2}}{\pi} \int_{(c)} \delta_v^{-s} \Gamma(s) \int_0^{\infty} \int_{-\infty}^{\infty} \frac{y^{s+k+i\tau-\frac{1}{2}}}{(-x+(y+1)i)^{2k} (x^2+(y+1)^2)^s} dx d^{\times}y \frac{ds}{2\pi i}.
\end{align*}
	Changing the variable $x \to x(y+1)$ in the innermost integral, the inner double integral splits as
	$$ \int_0^{\infty} \frac{y^{s+k+i\tau-\frac{1}{2}}}{(y+1)^{2s+2k-1}} \frac{dy}{y} \cdot \int_{-\infty}^{\infty} \frac{1}{(x-i)^{2k}(x^2+1)^s} dx. $$
	The first integral in $y$ is equal to $B\left( s+k+i\tau-\frac{1}{2}, s+k-i\tau - \frac{1}{2} \right)$. By the change of variable $x=\tan \theta$, the second integral becomes
\begin{align}
	&\quad \int_{-\infty}^{\infty} \frac{1}{(x-i)^{2k}(x^2+1)^s} dx \label{PIntR} \\
	&= 2 \cdot (-1)^k \int_0^{\frac{\pi}{2}} (\cos \theta)^{2s+2k-2} \sum_{j=0}^k \binom{2k}{2j} (-1)^{k-j} (\cos \theta)^{2j} (\sin \theta)^{2k-2j} d\theta \nonumber \\
	&= \sum_{j=0}^k \binom{2k}{2j} (-1)^j B\left( s+k+j-\frac{1}{2}, k-j+\frac{1}{2} \right) = P_k(s). \nonumber
\end{align}
	Thus we obtain
\begin{equation}
	\Tr(\pi(f_v)) = \frac{(2k-1)2^{2k-2}}{\pi} \int_{(c)} \delta_v^{-s} \Gamma(s) B\left( s+k+i\tau-\frac{1}{2}, s+k-i\tau - \frac{1}{2} \right) P_k(s) \frac{ds}{2\pi i}.
\label{LocTrA}
\end{equation}
	$P_k(s)$ has polynomial growth in $s$ in any vertical domain of the form $a \leq \Re(s) \leq b$, and is holomorphic in $\Re(s) > 1/2-k$. Moreover, from its integral representation (\ref{PIntR}), we deduce
\begin{equation}
	P_k(0) = \int_{-\infty}^{\infty} \frac{1}{(x-i)^{2k}} dx = \left. \frac{1}{(1-2k)(x-i)^{2k-1}} \right|_{x=-\infty}^{x=+\infty} = 0.
\label{Pk0}
\end{equation}
	Hence we can shift the contour to any $-1/2+\RamC<c<0$. In this region, by Stirling's estimation (see for example \cite[(B.8)]{Iw02}) we have
	$$ B\left( s+k+i\tau-\frac{1}{2}, s+k-i\tau - \frac{1}{2} \right) \ll_c \left\{ \begin{matrix} (1+\norm[\Im s])^{-\frac{1}{2}} & \text{if } \norm[\Im s] \geq \norm[\tau] \\ \frac{(1+\norm[\tau])^{2c+2k-2}}{(1+\norm[\Im s])^{2c+2k-\frac{3}{2}}} e^{\pi (\norm[\Im s] - \norm[\tau])} & \text{if } \norm[\Im s] \leq \norm[\tau] \end{matrix} \right. . $$
	Writing $y = \Im s$, we deduce that
\begin{align*}
	\extnorm{\Tr(\pi(f_v))} &\ll \delta_v^{-c} \left( \int_{\norm[y] \geq \norm[\tau]} (1+\norm[y])^A e^{-\frac{\pi}{2}\norm[y]} dy + (1+\norm[\tau])^{2c+2k-2} e^{-\pi \norm[\tau]} \int_{\norm[y] \leq \norm[\tau]} (1+\norm[y])^B e^{\frac{\pi}{2}\norm[y]} dy \right)
\end{align*} 
	for some constants $A,B$ depending only on $c$ and $k$. The desired bound follows readily.
\end{proof}

\begin{lemma}
	If $k=1$, then $\lim_{\delta \to 0^+} \Tr(D_{2}(f_v)) = -1$.
\end{lemma}
\begin{proof}
	The discrete series $D_2$ has a model as the unique subspace of $\Ind_{\gp{B}(\R)}^{\GL_2(\R)}(\norm^{\frac{1}{2}}, \norm^{-\frac{1}{2}}) \otimes \sgn^l$ with $l \in \{ 0,1 \}$, which contains the unique weight $2$ vector. Thus $\Tr(D_2(f_v))$ is given by specializing $k=1$ and by replacing $i\tau$ with $1/2$ in (\ref{LocTrA}), namely
\begin{align*}
	\Tr(D_2(f_v)) &= \int_{\R^{\times}} \int_{\R} f_v \begin{pmatrix} y & x \\ 0 & 1 \end{pmatrix} \sgn^l(y) dx d^{\times}y \\
	&= \frac{1}{\pi} \int_{(c)} \delta_v^{-s} \Gamma(s) B\left( s+1, s \right) \frac{-s \Gamma \left( \frac{1}{2} \right) \Gamma \left( s+\frac{1}{2} \right)}{\Gamma \left( s+2 \right)} \frac{ds}{2\pi i}.
\end{align*}
	Shifting the contour to $-1/2 < c < 0$ we cross a pole at $s=0$ and get
	$$ \Tr(D_2(f_v)) = -1 + \frac{1}{\pi} \int_{(c)} \delta_v^{-s} \Gamma(s) B\left( s+1, s \right) \frac{-s \Gamma \left( \frac{1}{2} \right) \Gamma \left( s+\frac{1}{2} \right)}{\Gamma \left( s+2 \right)} \frac{ds}{2\pi i} = -1 + O(\delta_v^{-c}). $$
	The desired formula for the limit follows readily.
\end{proof}

\begin{proposition}
	For any $k \geq 1$, we have
	$$ \lim_{\delta \to 0^+} \Tr(\pi(f_v)) = \left\{ \begin{matrix} 0 & \text{if } \pi \not\simeq D_{2k} \\ \varepsilon(1/2,D_{2k},\psi_{\R})=(-1)^k & \text{if } \pi \simeq D_{2k} \end{matrix} \right. . $$ 
	Moreover, if $\pi \simeq \pi(\norm^{i\tau}, \norm^{-i\tau})$ or $\pi(\norm^{i\tau}, \norm^{-i\tau}) \otimes \sgn$ is a principal or complementary series representation with $\norm[i\tau] < \RamC$, then for any $-1/2+\RamC < c < 0$ and any $C \gg 1$ we have
	$$ \Tr(\pi(f_v)) \ll_{c,C} \delta_v^{-c} (1+\norm[\tau])^{-C}. $$
\label{ArchCE}
\end{proposition}
\begin{proof}
	All assertions are already established except for the case $k=1$ and $\pi \simeq \pi(\norm^r, \norm^{-r})$ with $0 < r < 1/2$ being a complementary series representation. In this exceptional case, $\Tr(\pi(f_v))$ is given by specializing $k=1$ and replacing $i\tau$ with $r$ in (\ref{LocTrA}). The same argument following (\ref{LocTrA}) gives $\extnorm{\Tr(\pi(f_v))} \ll_{r,c} \delta_v^{-c}$ for any $r-1/2<c<0$, concluding the proof.
\end{proof}

	\subsection{Global Computation and Estimation}	
	
	By Lemma \ref{SphEisRes} and inserting the expansion of $K_0(x,y)$ in Theorem \ref{AutoKernSpecDSch}, we have
	$$ \frac{2 D_{\F}}{\Lambda_{\F}^{(-1)}(1)}\Res_{s=\frac{1}{2}} I(s) = \int_{[\PGL_2]} K_0(x,x) dx = \sum_{\pi} \sum_{\phi \in \Bas(\pi)} \int_{[\PGL_2]} \frac{\rpR(f)\phi(x) \overline{\phi(x)}}{\Norm[\phi]^2} dx = \sum_{\pi} \Tr(\pi(f)). $$
	The non-vanishing contribution of $\pi$ are those with usual conductor $\Cond(\pi_\fin) = \idl{N}^3$ by Proposition \ref{NArchC}. We regroup such $\pi$ according to its archimedean components as follows:
\begin{itemize}
	\item[(1)] There is a subset $I \subset V_{\infty}$ such that $\pi_v \simeq D_{2l_v}$ or $D_{2l_v} \otimes \sgn$ for some $l_v \in \Z_{\geq 1}$ at each $v \in I$,
	\item[(2)] while at each $v \mid \infty$ \& $v \notin I$, $\pi_v$ is a non-discrete series representation.
\end{itemize} 
	The non-trivial contribution comes from those for which $l_v \leq k_v$ at each $v \in I$, since $\pi_v(f_v)$ picks up the weight $2k_v$ vector. We write $\vec{l} \leq \vec{k}$ in this case, and denote by $\mathcal{A}(I,\vec{l}; \idl{N}^3)$ the set of cuspidal representations $\pi$ satisfying (1), (2) and $\Cond(\pi_{\fin}) = \idl{N}^3$. Here $\Cond(\pi_{\fin})$ is the usual conductor/level ideal of $\pi$. It has an archimedean counterpart $\Cond(\pi_{\infty}) = \sideset{}{_{v \mid \infty}} \prod \Cond(\pi_v)$, where
	$$ \Cond(\pi_v) = \left\{ \begin{matrix} (1+\norm[\tau])^2 & \text{if } \pi_v \simeq \pi(\norm^{i\tau}, \norm^{-i\tau}) \text{ or } \pi(\norm^{i\tau}, \norm^{-i\tau}) \otimes \sgn \\ 1+k^2 & \text{if } \pi_v \simeq  D_{2k} \text{ or } D_{2k} \otimes \sgn \end{matrix} \right. . $$
	Note $\mathcal{A}(V_{\infty},\vec{k};\idl{N}^3)=\mathcal{A}(\vec{k};\idl{N}^3)$. By Proposition \ref{ArchCE}, we have for any $0 < c < 1/2-\vartheta$ and $C \gg 1$,
	$$ \sum_{\pi \in \mathcal{A}(I, \vec{l}; \idl{N}^3)} \Tr(\pi(f)) = \sum_{\pi \in \mathcal{A}(I, \vec{l}; \idl{N}^3)} \prod_{v \mid \infty, v \notin I} \delta_v^c \prod_{v \in I} \Tr(\pi_v(f_v)) \prod_{\vp < \infty} \varepsilon \left( \frac{1}{2},\pi_{\vp}, \psi_{\vp} \right) \cdot O(\Cond(\pi_{\infty})^{-C}). $$
	The quantity $\prod_{v \in I} \Tr(\pi_v(f_v)) = \prod_{v \in I} \Tr(D_{2l_v}(f_v))$ depends only on $\vec{l}$. By (weak) Weyl's law, the sum over $\pi$ of $\Cond(\pi_{\infty})^{-A}$ is absolutely convergent for some absolute constant $A > 0$. Hence we get
\begin{align}
	\lim_{\vec{\delta} \to \vec{0}^+} \Res_{s=\frac{1}{2}} I(s) &= \frac{\Lambda_{\F}^{(-1)}(1)}{2 D_{\F}} \lim_{\vec{\delta} \to \vec{0}^+} \sum_{I, \vec{l} \leq \vec{k}} \sum_{\pi \in \mathcal{A}(I, \vec{l}; \idl{N}^3)} \Tr(\pi(f)) \nonumber \\
	&= \frac{\Lambda_{\F}^{(-1)}(1)}{2 D_{\F}} \lim_{\vec{\delta} \to \vec{0}^+} \sum_{\vec{l} \leq \vec{k}} \sum_{\pi \in \mathcal{A}(V_{\infty}, \vec{l}; \idl{N}^3)} \Tr(\pi(f)) \nonumber \\
	&= \frac{\Lambda_{\F}^{(-1)}(1)}{2 D_{\F}^2} \sum_{\pi \in \mathcal{A}(\vec{k},\idl{N}^3)} \varepsilon \left( \frac{1}{2},\pi \right), \label{ArchFinal}
\end{align}
	where in the last equation we applied Proposition \ref{NArchC} \& \ref{ArchCE} and $\Vol(\sideset{}{_{\vp < \infty}} \prod\PGL(\vo_{\vp})) = D_{\F}^{-1}$.

\section{Computation of $\mathrm{GL}_1$ Side}

	\subsection{Reparametrization and First Reductions}
	
	We turn to the computation of $I(s)$ (see (\ref{JZGeom})) for our test function. First of all we notice that our test function $f$ satisfies both conditions in Lemma \ref{SimpleRSF}. In fact, by the explicit formula (\ref{TFNArch}), we see that $f_{\vp}$ for $\vp \mid \idl{N}$ satisfies Lemma \ref{SimpleRSF} (2); since $f_{\vp} = \widetilde{f_{\vp}^b}$ for $\vp \mid \idl{N}$ is an average of conjugation of $f_{\vp}^b$, which satisfies Lemma \ref{SimpleRSF} (1), $f_{\vp}$ satisfies it, too. As a consequence, we get
\begin{equation}
	I(s) = \sum_{\E} I_{\E}(s) = \sum_{\E} \sum_{1 \neq \lambda \in \E^{\times}/\F^{\times}} I(s; \lambda; f, \Phi),
\label{EllTermDef}
\end{equation}
	where for every quadratic field extension $\E/\F$ we have identified $\E^{\times}$ with its image under a fixed embedding $\E \to \Mat_2(\F)$ as a non-split torus, and the summand is defined to be
	$$ I(s; \lambda; f, \Phi) = \frac{1}{2} \cdot \int_{\E_{\A}^{\times} \backslash \GL_2(\A)} f(x^{-1} \lambda x) \cdot \int_{\E_{\A}^{\times}} \Phi((0,1)tx) \norm[\det tx]_{\A}^{1/2+s} d^{\times}t dx. $$
	
\begin{lemma}
	$I(s;\lambda;f,\Phi)$ is independent of the embedding $\E \hookrightarrow \Mat_2(\F)$.
\label{ConjInv}
\end{lemma}
\begin{proof}
	Any embedding $\iota: \E \to \Mat_2(\F)$ is realized by choosing an $\F$-basis $e_1,e_2$ of $\E$ and taking
	$$ (xe_1,xe_2) = (e_1,e_2) \iota(x), \quad \forall x \in \E. $$
	In this way, both $\E^{\times}$ and $\GL_2(\F)$ acts on $\E^{\times}$ transitively. Since the stabilizer of $e_1$ in $\GL_2(\F)$ is contained in $\gp{B}(\F)$, we get the decomposition
	$$ \GL_2(\F) = \E^{\times} \gp{B}(\F). $$
	Writing $\GL_2(\F) \ni g_0 = t_0 b_0$ for $t_0 \in \E^{\times}, b_0 \in \gp{B}(\F)$, we get by the change of variables $x \mapsto g_0 x$
\begin{align*}
	I(s;\lambda;f,\Phi) &= \frac{1}{2} \int_{g_0^{-1}\E_{\A}^{\times}g_0 \backslash \GL_2(\A)} f(x^{-1} g_0^{-1} \lambda g_0 x) \int_{\E_{\A}^{\times}} \Phi((0,1)tg_0x) \norm[\det tx]_{\A}^{1/2+s} d^{\times}t dx \\
	&= \frac{1}{2} \int_{b_0^{-1}\E_{\A}^{\times}b_0 \backslash \GL_2(\A)} f(x^{-1} b_0^{-1} \lambda b_0 x) \int_{\E_{\A}^{\times}} \Phi((0,1)b_0^{-1}tb_0x) \norm[\det tx]_{\A}^{1/2+s} d^{\times}t dx \\
	&= I(s; b_0^{-1}\lambda b_0; f, \Phi) = I(s; g_0^{-1}\lambda g_0; f, \Phi).
\end{align*}
	This is the required independence of embeddings since any two such embeddings are conjugate by an element of $\GL_2(\F)$ by Skolem-Noether's theorem.
\end{proof}
\begin{corollary}
	If $\bar{\lambda}$ denotes the conjugate of $\lambda$ in $\E$, then we have
	$$ I(s;\lambda;f,\Phi) = I(s;\bar{\lambda};f,\Phi). $$
\label{ConjID}
\end{corollary}
\begin{proof}
	By Skolem-Noether's theorem again, the conjugation on $\E$ is induced by conjugation by an element in $\GL_2(\F)$. Thus the above equality follows from Lemma \ref{ConjInv}.
\end{proof}
\begin{remark}
	The elliptic orbital integrals are independent of the embeddings. Lemma \ref{ConjInv} is an analogue of this result.
\end{remark}
	
	The above parametrization in terms of elements in quadratic field extensions is not convenient for compuation or estimation. We need to give a more convenient one in (\ref{EllTermDef}). Let
	$$ E_2(\F) = \left\{ (a,b) \in \F \times \F^{\times} \ \middle| \ a^2-4b \text{ is not a square in } \F^{\times} \right\}, $$
	and consider an action of $\F^{\times}$ on $E_2(\F)$ given by $t \circ (a,b) := (ta,t^2b)$. We have a surjective map
	$$ \sideset{}{_{\E}} \bigcup \left( \E^{\times}/\F^{\times} - \{ 1 \} \right) \to \F^{\times} \backslash E_2(\F) =: [E_2(\F)], \quad \lambda \mapsto [(\Tr(\lambda), \Nr(\lambda))]. $$
	The pre-image of every element $[a,b]:=\F^{\times} \circ (a,b)$ consists of
\begin{itemize}
    \item precisely two elements $\F^{\times}\lambda$ and $\F^{\times}\bar{\lambda}$ with $\Tr(\lambda)=a$ and $\Nr(\lambda)=b$, if $a \neq 0$;
    \item only one element $\F^{\times}\lambda = \F^{\times}\bar{\lambda}$ if $a=0$.
\end{itemize}
\begin{remark}
	For a nontrivial $\GL_2(\F)$ conjugacy class $C$ inside $\PGL_2(\F) = \gp{Z}(\F)\bs \GL_2(\F)$ with representative $\lambda\in \GL_2(\F)$, set for $x \in \PGL_2(\A)$
$$
K_{C}(x,x) = \sum_{\lambda\in C,\lambda\notin \gp{B}(\F)/\gp{Z}(\F)}
f(x^{-1}\lambda x)
$$
which corresponds to the $C$-part of of the kernel function $K(x,x)$. In \cite[\S 2.2]{JZ87}, the authors claim that 
\begin{align}\label{eq:kernel:correction:1}
K_{C}(x,x) =
\sum_{
\substack{
\gamma\in \gp{G}_\lambda(\F)\bs \GL_2(\F),
\\
\gamma^{-1}\lambda \gamma\notin \gp{B}(\F)\backslash \gp{Z}(\F)}
}f(x^{-1}\gamma^{-1}\lambda \gamma x)
\end{align}
where $\gp{G}_\lambda(\F)$ is the centralizer of $\lambda$ inside $\GL_2(\F)$. This is actually NOT true in general. In fact, by definition we have
$$
K_C(x,x) = 
\sum_{
\substack{
\gamma\in \mathrm{Cent}_\lambda(\F)\backslash \PGL_2(\F),
\\
\gamma^{-1}\lambda \gamma\notin \gp{B}(\F)\backslash \gp{Z}(\F)}
}f(x^{-1}\gamma^{-1}\lambda \gamma x)
$$
where $\mathrm{Cent}_\lambda(\F)$ is the centralizer of $\lambda$ in $\PGL_2(\F)$ defined by 
\begin{align*}
\mathrm{Cent}_\lambda(\F) 
=& 
\{
g\in \PGL_2(\F)
|\quad 
g^{-1}\lambda g\equiv \lambda \mod \gp{Z}(\F)
\}
\\
=&
\{
g\in \PGL_2(\F)
|\quad 
g^{-1}\lambda g= z\lambda \text{ for some $z\in \gp{Z}(\F)$}
\}.
\end{align*}
The condition in the last line holds if and only if 
\begin{align}\label{eq:kernel:correction:2}
\tr(\lambda) = z\tr(\lambda),\quad \det(\lambda) = z^2\det(\lambda).
\end{align}
When $\tr(\lambda) \neq 0$, \eqref{eq:kernel:correction:2} holds only when $z=1$. Hence
$$
\mathrm{Cent}_\lambda(\F) = \gp{Z}(\F)\backslash \gp{G}_\lambda(\F)\quad \text{and}
\quad 
\mathrm{Cent}_\lambda(\F)\backslash \PGL_2(\F) \simeq 
\gp{G}_\lambda(\F)\backslash \GL_2(\F).
$$
In this case the equality \eqref{eq:kernel:correction:1} holds and the discussion of \cite[\S 2.2]{JZ87} applies to the conjugacy class $C$. When $\tr(\lambda) = 0$, \eqref{eq:kernel:correction:2} holds when $z^2 = 1 \Leftrightarrow z=\pm 1$. In this case, by Skolem-Noether's theorem, there exists $g_0\in \GL_2(\F)$ such that 
$$
g_0^{-1}\lambda g_0 = -\lambda.
$$
In other words, within $\PGL_2(\F)$, $\lambda$ and $-\lambda$ are conjugate to each other. Therefore 
$$
\mathrm{Cent}_\lambda(\F) = 
\gp{Z}(\F)\backslash 
\bigg\{
\gp{G}_\lambda(\F)
\bigsqcup
g_0^{-1}\gp{G}_\lambda(\F)g_0
\bigg\} \quad \Rightarrow \quad [\mathrm{Cent}_\lambda(\F): \gp{Z}(\F)\backslash \gp{G}_\lambda(\F)] = 2.
$$
In this case the discussion in \cite[\S 2.2]{JZ87} still applies, except that we need to divide the corresponding elliptic and hyperbolic contribution by $2$.
\end{remark}
\begin{definition}
	For $[a,b] \in [E_2(\F)] $, we shall write
	$$ I(s;[a,b];f,\Phi) := \sideset{}{_{\substack{\lambda \in \F^{\times} \backslash \E^{\times} \\ [\Tr(\lambda), \Nr(\lambda)] = [a,b]}}} \sum I(s;\lambda;f,\Phi). $$
	By Corollary \ref{ConjID} and the above discussion, we have
	$$ I(s;[a,b];f,\Phi) = 2^{-\mathbbm{1}_{a=0}} \int_{\E_{\A}^{\times} \backslash \GL_2(\A)} f(x^{-1} \lambda x) \cdot \int_{\E_{\A}^{\times}} \Phi((0,1)tx) \norm[\det tx]_{\A}^{1/2+s} d^{\times}t dx $$
	for any $\lambda$ with $(\Tr(\lambda),\Nr(\lambda))=(a,b)$.
\end{definition}

\noindent We deduce the following re-parametrized version of (\ref{EllTermDef})
\begin{equation}
	I(s) = \sum_{[a,b] \in [E_2(\F)]} I(s; [a,b]; f, \Phi).
\label{EllTermRP}
\end{equation}

\begin{lemma}
	Assume $f(x^{-1}\lambda x) \neq 0$ for some $\F$-elliptic $\lambda$ and $x \in \PGL_2(\A)$. Write $\E=\F[\lambda]$ for the quadratic field generated by $\lambda$ and let $(\Tr(\lambda),\Nr(\lambda))=(a,b)$. Then we have
\begin{itemize}
	\item[(1)] $\idl{N}$ is a square in the narrow class group $\Cl_+(\F)$ and every prime ideal $\vp \mid \idl{N}$ is ramified in $\E$;
	\item[(2)] For any fractional ideal $\idl{J}$ such that $\idl{J}^{-2}\idl{N}$ is a principal ideal, we can find $z \in \F^{\times}$ such that $z^2 b$ is a totally positive generator of $\idl{J}^{-2}\idl{N}$ and $a/(zb) \in \idl{J}$.
\end{itemize} 
\label{GL1NV}
\end{lemma}
\begin{proof}
	By our choice of local test functions (\ref{TFNArch}), at any $\vp \mid \idl{N}$ we can find $z_{\vp} \in \F_{\vp}^{\times}$ so that $z_{\vp} \lambda$ is conjugate to an element in $\begin{pmatrix} \vp \vo_{\vp} & \vo_{\vp}^{\times} \\ \varpi_{\vp} \vo_{\vp}^{\times} & \vp \vo_{\vp} \end{pmatrix}$. Hence we have
	$$ \Tr(z_{\vp} \lambda) \in \vp \vo_{\vp}, \quad \Nr(z_{\vp} \lambda) \in \varpi_{\vp} \vo_{\vp}^{\times}, \quad \forall \vp \mid \idl{N}. $$
	In particular, $z_{\vp}\lambda$ is a root of an Eisenstein polynomial over $\F_{\vp}$. Hence $\vp \mid \idl{N}$ is (totally) ramified in $\E$.
	
\noindent Similarly, at $\vp \nmid \idl{N}$, we can find $z_{\vp} \in \F_{\vp}^{\times}$ so that
	$$ \Tr(z_{\vp} \lambda) \in \vo_{\vp}, \quad \Nr(z_{\vp} \lambda) \in \vo_{\vp}^{\times}. $$
	Since $\F^{\times} \backslash \A_{\fin}^{\times} / \widehat{\vo}^{\times} \simeq \Cl(\F)$, there is a fractional ideal $\idl{J}$ represented by a finite idele $(\delta_{\vp})_{\vp}$, an $z \in \F^{\times}$ such that $z^{-1}z_{\vp} \in \delta_{\vp} \vo_{\vp}^{\times}$ for all $\vp < \infty$. It follows that $\Nr(z\lambda) \in \delta_{\vp}^{-2} \vp \vo_{\vp}^{\times}$ for $\vp \mid \idl{N}$, and $\Nr(z\lambda) \in \delta_{\vp}^{-2} \vo_{\vp}^{\times}$ for $\vp \nmid \idl{N}$. Hence $\idl{J}^{-2} \idl{N} = (\Nr(z\lambda))$ is a principal ideal, i.e., $\idl{N}$ is a square in $\Cl(\F)$, and $z^2 b = \Nr(z\lambda)$ is a generator of $\idl{J}^{-2}\idl{N}$. Since the support of $f_v$ at $v \mid \infty$ is contained in the subgroup of positive determinant elements, we deduce that $b=\det(x^{-1}\lambda x)$, hence $z^2b$ is totally positive. Moreover, we have
	$$ \Tr(z\lambda) \in \Tr(\delta_{\vp}^{-1}z_{\vp}\lambda \vo_{\vp}^{\times}) \in \left\{ \begin{matrix} \delta_{\vp}^{-1} \vp \vo_{\vp} & \text{if } \vp \mid \idl{N} \\ \delta_{\vp}^{-1} \vo_{\vp} & \text{if } \vp \nmid \idl{N} \end{matrix} \right. . $$
	Thus $za = \Tr(z\lambda) \in \idl{J}^{-1} \idl{N} = z^2b \idl{J}$, i.e., $a/(zb) \in \idl{J}$.
\end{proof}

\begin{corollary}[Theorem \ref{MainGenF} (1)]
	If $\idl{N}$ is not a square in the narrow class group $\Cl_+(\F)$, then we have $I(s)=0$ and there is no bias.
\label{NSNB}
\end{corollary}
\begin{proof}
	By Lemma \ref{GL1NV} (1), we have $I(s;[a,b];f,\Phi) = 0$ identically. Hence $I(s)=0$.
\end{proof}

	From now on, we assume $\idl{N}$ is a square in $\Cl_+(\F)$ and fix a fractional ideal $\idl{J}$ such that $\idl{J}^{-2}\idl{N}$ is principal with a chosen totally positive generator $N$. Let $U_+$ be the group of totally positive units of $\F$. Then all other totally positive generators of $\idl{J}^{-2}\idl{N}$ are of the form $u N$ for some $u \in U_+$. If $[a_1 N,u_1 N] = [a_2 N, u_2 N]$ for some $a_1,a_2 \in \idl{J}$ and $u_1,u_2 \in U_+$, then there exists $z \in \vo^{\times}$ such that $u_2 = u_1 z^2$ and $a_2 = a_1 z$. We choose a system of representatives $U_2$ for the quotient $U_+/(\vo^{\times})^2$. Then the non-zero terms in (\ref{EllTermRP}) are for $[a,b] = [nN, u N]$ with $u \in U_2$ and $n \in \idl{J}$. Moreover, $[n_1N, u_1 N] = [n_2 N, u_2 N]$ if and only if $u_1 = u_2$ and $n_1 = \pm n_2$. We can rewrite
\begin{equation}
	I(s) = \sum_{u \in U_2} \sum_{\substack{n \in \idl{J}/\{ \pm 1 \} \\ n^2u^2N^2-4uN \neq \square}} I(s; [nuN,uN]; f, \Phi).
\label{EllTermRRP}
\end{equation}

\begin{remark}
	For notational simplicity, we shall omit $u$ and $N$ in the local computation and write 
	$$ I(s;n;f,\Phi) = I(s;[nuN,uN];f,\Phi)$$
in the rest of this section. Obviously, we have a factorization of
	$$ I(s;n;f,\Phi) = \prod_v I_v(s;n;f_v,\Phi_v) $$
as product of local components.
\end{remark}

\begin{remark}
	For definiteness, we choose $\lambda$ re-parametrized by $[nuN,uN]$ to be
	$$ \lambda = \begin{pmatrix} nuN & -uN \\ 1 & 0 \end{pmatrix}. $$
	We will write $N$ instead of $uN$ in the local computation.
\end{remark}

\begin{remark}
	We can choose $\idl{J}$ to be an integral ideal since $\Cl(\F)$ is finite.
\end{remark}

	\subsection{Archimedean Places}
	
	Recall that a family of functions $h=h_{\delta} \in \Cont^{\infty}(\R_{\geq 0})$ are chosen to define the test function $f=f_v$ (we omit the subscript $v$ in this subsection for simplicity). The local component in question is
\begin{align*}
	\frac{I(s;n;f;\Phi)}{\Gamma_{\R}(2s+1)} &= \int_{\R} \int_{\R^{\times}} f \left( \begin{pmatrix} 1 & -x \\ 0 & 1 \end{pmatrix} \begin{pmatrix} nN & -y^{-1}N \\ y & 0 \end{pmatrix} \begin{pmatrix} 1 & x \\ 0 & 1 \end{pmatrix} \right) \norm[y]^{s+\frac{1}{2}} d^{\times}y dx \\
	&= \frac{(2k-1)(4N)^k}{4\pi} \left\{ I\left( h_{\frac{\delta}{N}}; n^2N^2-4N, nN; s+\frac{1}{2} \right) + I \left( h_{\frac{\delta}{N}}; n^2N^2-4N, -nN; s+\frac{1}{2} \right) \right\},
\end{align*}
	where $I(h; \Delta, t; s)$ for $\Delta < t^2$ is defined by
	$$ I(h; \Delta, t; s) := \int_0^{\infty} \int_{-\infty}^{\infty} \frac{y^{s+2k-2}}{(x^2+y^2+ity-\frac{\Delta}{4})^{2k}} h \left( \extnorm{\frac{x^2+y^2+ity-\frac{\Delta}{4}}{y}}^2 \right) dxdy. $$
	The computation is thus reduced to $I(h;\Delta,t;1)+I(h;\Delta,-t;1)$. We first consider the case $\Delta > 0$.
	
\begin{lemma}
	Assume $0 < \Delta < t^2$. Recall $P_k(s)$ given in Definition \ref{WtkPoly}. We have for $c=\Re(s) > 0$ and for any $h \in \Cont^{\infty}(\R_{\geq 0})$ of rapid decay at $\infty$
	$$ I(h; \Delta, t; 1) + I(h; \Delta, -t; 1) = \int_{(c)} \norm[t]^{1-2k-2s} P_k(s) \Mellin{h}(s) \frac{ds}{2\pi i}. $$
\label{GenAC}
\end{lemma}
\begin{proof}
	In the defining integral of $I(\cdot)$, we change the variable $x$ to
	$$ u := \frac{x^2-\frac{\Delta}{4}}{y} + y \quad \Leftrightarrow \quad x^2 = \frac{\Delta}{4} + uy - y^2. $$
	Note that $x^2 \geq 0$ implies $y \leq (u+\sqrt{u^2+\Delta})/2$. Thus we get
	$$ I(h; \Delta, t; s) = \int_{-\infty}^{\infty} \int_0^{\frac{u+\sqrt{u^2+\Delta}}{2}} \frac{h(u^2+t^2)}{(u+it)^{2k}} \frac{y^{s-1}}{\sqrt{\frac{\Delta}{4}+uy-y^2}} dy du. $$
	It follows that
\begin{align*}
	I(h; \Delta, t; s) &+ I(h; \Delta, -t; s) = \int_0^{\infty} \left\{ \frac{1}{(u+it)^{2k}} + \frac{1}{(u-it)^{2k}} \right\} h(u^2+t^2) \cdot \\
	&\qquad \left\{ \int_0^{\frac{u+\sqrt{u^2+\Delta}}{2}} \frac{y^{s-1}}{\sqrt{\frac{\Delta}{4}+uy-y^2}} dy + \int_0^{\frac{-u+\sqrt{u^2+\Delta}}{2}} \frac{y^{s-1}}{\sqrt{\frac{\Delta}{4}-uy-y^2}} dy\right\} du.
\end{align*}
	At $s=1$, we have
\begin{align} 
	&\quad \int_0^{\frac{u+\sqrt{u^2+\Delta}}{2}} \frac{dy}{\sqrt{\frac{\Delta}{4} + uy - y^2}} + \int_0^{\frac{-u+\sqrt{u^2+\Delta}}{2}} \frac{dy}{\sqrt{\frac{\Delta}{4} - uy - y^2}} \nonumber \\
	&= \int_{-\frac{\sqrt{u^2+\Delta}}{2}}^{\frac{\sqrt{u^2+\Delta}}{2}} \frac{dy}{\sqrt{\frac{u^2+\Delta}{4}-y^2}} = \int_{-1}^1 \frac{dy}{\sqrt{1-y^2}} = \pi. \label{InnInt}
\end{align}
	Applying the Mellin inversion formula for $h$, we obtain
\begin{align*}
	I(h; \Delta, t; 1) &+ I(h; \Delta, -t; 1) = \frac{\pi}{\norm[t]^{2k-1}} \int_0^{\infty} \left\{ \frac{1}{(u+i)^{2k}} + \frac{1}{(u-i)^{2k}} \right\} h(t^2(u^2+1)) du \\
	&= \frac{\pi}{\norm[t]^{2k-1}} \int_0^{\infty} \left\{ \frac{1}{(u+i)^{2k}} + \frac{1}{(u-i)^{2k}} \right\} \int_{(c)} \Mellin{h}(s) (t^2(u^2+1))^{-s} \frac{ds}{2\pi i} du.
\end{align*}
	For $c > 0$, the above integral is absolutely convergent. Hence we can change the order of integrations and conclude by the following equation (see (\ref{PIntR}))
	$$ \int_0^{\infty} \left\{ \frac{1}{(u+i)^{2k}} + \frac{1}{(u-i)^{2k}} \right\} (u^2+1)^{-s} du = P_k(s). $$
\end{proof}

\begin{lemma}
	Assume $\Delta \leq 0$. We have
	$$\lim_{\delta \to 0^+} \left( I(h_{\delta}; \Delta, t; 1) + I(h_{\delta}; \Delta, -t; 1) \right) = \frac{2\pi}{2k-1} \Re \left( \frac{1}{(\sqrt{\norm[\Delta]}+it)^{2k-1}} \right). $$
	Consequently, we have in the case $n^2N^2-4N \leq 0$
	$$ \lim_{\delta \to 0^+} I \left( \frac{1}{2};n;f,\Phi \right) = \frac{(4N)^k}{2 \pi} \Re \left( \frac{1}{(\sqrt{\norm[n^2N^2-4N]}+inN)^{2k-1}} \right). $$
\label{SpAC}
\end{lemma}
\begin{proof}
	The same change of variables as in the proof of Lemma \ref{GenAC} leads to
	$$ I(h; \Delta, t; s) = \int_{\sqrt{\norm[\Delta]}}^{\infty} \int_{\frac{u-\sqrt{u^2+\Delta}}{2}}^{\frac{u+\sqrt{u^2+\Delta}}{2}} \frac{h(u^2+t^2)}{(u+it)^{2k}} \frac{y^{s-1}}{\sqrt{\frac{\Delta}{4}+uy-y^2}} dy du. $$
	Hence we get by (\ref{InnInt}) and the dominated convergence theorem
	$$ I(h_{\delta}; \Delta, t; 1) = \pi \int_{\sqrt{\norm[\Delta]}}^{\infty} \frac{h_{\delta}(u^2+t^2)}{(u+it)^{2k}} du \to \pi \int_{\sqrt{\norm[\Delta]}}^{\infty} \frac{1}{(u+it)^{2k}} du = \frac{\pi}{2k-1} \frac{1}{(\sqrt{\norm[\Delta]}+it)^{2k-1}}, \quad \delta \to 0^+. $$
	It follows that
	$$ \lim_{\delta \to 0^+} \left( I(h_{\delta}; \Delta, t; 1) + I(h_{\delta}; \Delta, -t; 1) \right) = \frac{\pi}{2k-1} \left( \frac{1}{(\sqrt{\norm[\Delta]}+it)^{2k-1}} + \frac{1}{(\sqrt{\norm[\Delta]}-it)^{2k-1}} \right). $$
\end{proof}	


	\subsection{Non-Archimedean Places}
	
		\subsubsection{Method of Computation}
		
	We turn to the computation of $I_{\vp}(s;n;f_{\vp},\Phi_{\vp})$ at a place $\vp < \infty$. We continue to omit the subscript $\vp$ for simplicity of notation. The method was previously developed in \cite[\S 4.1-4.3]{CWZ21}. We shall recall it with an extension suitable for the current situation.
	
	First of all, at the basis of the method there are two families of embeddings of quadratic algebra $\E$ over $\F$ into $\Mat_2(\F)$, which we call \emph{standard}, resp. \emph{well-positioned}.
	
\begin{definition}
	Let $\E$ be a quadratic (algebra) extension of $\F$ with an embedding $\iota: \E \hookrightarrow \Mat_2(\F)$ as $\F$-algebras. It is standard with respect to an element $\F \not\ni \theta \in \vo_{\E}$ in the ring of integers of $\E$, so that
	$$ \theta^2 - \fb \theta + \fa = 0, \quad \iota(\theta) = \begin{pmatrix} \fb & -\fa \\ 1 & 0 \end{pmatrix}, \quad \text{and} \quad \fa \notin \varpi^2 \vo_{\F}, $$
	where $\varpi$ is a uniformizer of the valuation ring $\vo_{\F}$. A standard embedding as above is called well-positioned if in addition we have
	$$ \vo_{\E} = \vo_{\F}[\theta]. $$
	In particular, well-positioned embeddings are optimal with respect to $\vo_{\E}$ and $\Mat_2(\vo_{\F})$. We also say that $\E$ is a standard or well-positioned if we identify $\E$ with its image $\iota(\E)$.
\label{WPEmb}
\end{definition}

\begin{proposition}
	Let $\E$ be a well-positioned quadratic extension of $\F$ with respect to $\theta$ contained in $\Mat_2(\F)$. If it is non-split, then we have
	$$ \E^{\times} \subset \F^{\times} \GL_2(\vo_{\F}), \quad \GL_2(\F) = \sideset{}{_{r=0}^{\infty}} \bigsqcup \E^{\times} a(\varpi^{-r}) \GL_2(\vo_{\F}), \quad a(t) := \begin{pmatrix} t & 0 \\ 0 & 1 \end{pmatrix}. $$
	If it is split, we identify $\theta$ with an element in $\vo_{\F}$, as well as for $\bar{\theta} := \fb - \theta$. Then we have
	$$ \GL_2(\F) = \sideset{}{_{r=0}^{\infty}} \bigsqcup \begin{pmatrix} \theta & \bar{\theta} \\ 1 & 1 \end{pmatrix} \gp{A} n(\varpi^{-r}) \GL_2(\vo_{\F}). $$
\label{IwaQE}
\end{proposition}
\begin{proof}
	We only need to consider the non-split case, since the other one is just the usual Iwasawa decomposition. We regard elements of $\E$ as row vectors in $\F \oplus \F$ via
	$$ \E \to \F \oplus \F, \quad x \theta + y \mapsto (x,y). $$
	Then multiplication by $x \in \E$ is identified with multiplication by $x \in \Mat_2(\F)$ on $\F \oplus \F$. $\GL_2(\F)$ acts on $\F \oplus \F$, hence it also acts on the set $\mathcal{L}$ of $\vo_{\F}$-lattices in $\F \oplus \F$. Since $\vo_{\F}$ is DVR, every $L \in \mathcal{L}$ is principal for some order $\vO \subset \vo_{\E}$ (see \cite[Proposition 7.4]{Cox89} plus \cite[Proposition \Rmnum{1}.12.4]{Ne99}). By \cite[Lemma 7.2]{Cox89}, all the orders are parametrized by $r \in \Z_{\geq 0}$ as
	$$ \vO_r = \vo_{\F} \varpi^r \theta + \vo_{\F} = \vO_0 \cdot a(\varpi^r). $$
	Taking into account that $\GL_2(\F)$ acts transitively on $\mathcal{L}$ with stabilizer group $\GL_2(\vo_{\F})$, we get
	$$ \GL_2(\F) = \sideset{}{_{r=0}^{\infty}} \bigsqcup \GL_2(\vo_{\F}) a(\varpi^r) \E^{\times}. $$
	The desired decomposition follows by taking inverse.
\end{proof}

	Next, let $\iota_1$ (resp. $\iota_2$) be a standard (resp. well-positioned embedding) with respect to $\theta_1$ (resp. $\theta_2$) for the same quadratic algebra $\E$. In particular, $\theta_1 \in \vo_{\E} = \vo_{\F}[\theta_2]$. Hence there exist $u,v \in \vo_{\F}$ with $v \neq 0$ such that $\theta_1 = v \theta_2 + u$, or equivalently
	$$ \begin{pmatrix} \theta_1 \\ 1 \end{pmatrix} = \begin{pmatrix} v & u \\ 0 & 1 \end{pmatrix} \begin{pmatrix} \theta_2 \\ 1 \end{pmatrix} \quad \Rightarrow \quad \begin{pmatrix} x\theta_1 \\ x \end{pmatrix} = \begin{pmatrix} v & u \\ 0 & 1 \end{pmatrix} \begin{pmatrix} x\theta_2 \\ x \end{pmatrix} \quad \forall x \in \E. $$
	Moreover, we have by definition
	$$ \begin{pmatrix} x \theta_j \\ x \end{pmatrix} = \iota_j(x) \begin{pmatrix} \theta_j \\ 1 \end{pmatrix}, \quad \forall x \in \E \quad \Rightarrow \quad \iota_2(x) = \begin{pmatrix} v & u \\ 0 & 1 \end{pmatrix}^{-1} \iota_1(x) \begin{pmatrix} v & u \\ 0 & 1 \end{pmatrix}, \quad \forall x \in \E. $$
	In particular, we get
	$$ \iota_2(\theta_1) = \begin{pmatrix} v & u \\ 0 & 1 \end{pmatrix}^{-1} \iota_1(\theta_1) \begin{pmatrix} v & u \\ 0 & 1 \end{pmatrix} = \begin{pmatrix} * & * \\ v & * \end{pmatrix}. $$
	Let $f \in \Cont_c^{\infty}(\GL_2(\F))$. Fix a meromorphic section $e_{0,s} \in \pi(\norm_{\F}^s, \norm_{\F}^{-s})$ invariant by $\GL_2(\vo_{\F})$. Define
	$$ I_j := \int_{\F^{\times} \backslash \GL_2(\F)} f(x^{-1}\iota_j(\theta_1)x) e_{0,s}(x) dx. $$
	We deduce easily the following relation
\begin{equation}
	I_1 = \int_{\F^{\times} \backslash \GL_2(\F)} f \left( x^{-1}\begin{pmatrix} v & u \\ 0 & 1 \end{pmatrix}\iota_2(\theta_1)\begin{pmatrix} v & u \\ 0 & 1 \end{pmatrix}^{-1}x \right) e_{0,s}(x) dx =  \norm[v]_{\F}^{\frac{1}{2}+s} I_2.
\label{WellP}
\end{equation} 
	
\begin{lemma}
	Let $a$ be the conductor exponent of the order $\vo_{\F}[\theta_1]$. Then we have $\norm[v]_{\F} = \norm[\varpi^a]_{\F}$.
\label{WellPExp}
\end{lemma}
\begin{proof}
	Recall the conductor exponent is defined via
	$$ \vo_{\F}[\theta_1] = \vo_{\F} + \varpi^a \vo_{\E} = \vo_{\F} + \varpi^a \vo_{\F}[\theta_2] = \vo_{\F}[\varpi^a \theta_2]. $$
	But $\vo_{\F}[\theta_1] = \vo_{\F}[v \theta_2]$ by $\theta_1 = v \theta_2 + u$. It follows that
	$$ [\vo_{\E}:\vo_{\F}[\theta_1]] = [\vo_{\F} : \varpi^a \vo_{\F}] = [\vo_{\F} : v\vo_{\F}] \quad \Rightarrow \quad v \in \varpi^a \vo_{\F}^{\times}. $$
\end{proof}

\begin{definition}
	We call $I_2$ a \emph{well-positionization} of $I_1$ associated with $(\iota_2, \theta_2)$.
\end{definition}

	Finally, we specialize the spherical section $e_{0,s}$ to the case of Godement sections
	$$ e_{0,s}(x) = \norm[\det x]_{\F}^{\frac{1}{2}+s} \int_{\F^{\times}} \Phi((0,t)x) \norm[t]_{\F}^{1+2s} d^{\times}t, \quad \Phi = \mathbbm{1}_{\vo_{\F} \times \vo_{\F}}. $$
	Proposition \ref{IwaQE} implies the following decomposition
	$$ I_2 = \zeta_{\E} \left( s+\frac{1}{2} \right) \cdot \sum_{r=0}^{\infty} \mathrm{RS}-\vO_{\theta_1}(r,f) \cdot \mathrm{wt}(s,r,\E/\F), $$
	where the \emph{normalized Rankin-Selberg orbital integral} $\mathrm{RS}-\vO_{\theta_1}(r,f)$ is defined in terms of the notation of Proposition \ref{IwaQE} by (according to $\E/\F$ non-split or split)
	$$ \mathrm{RS}-\vO_{\theta_1}(r,f) = \frac{1}{\Vol(\GL_2(\vo_{\F}))} \left\{ \begin{matrix} \int_{\GL_2(\vo_{\F})} f(\kappa^{-1} a(\varpi^r) \iota_2(\theta_1) a(\varpi^{-r}) \kappa) d\kappa \\ \int_{\GL_2(\vo_{\F})} f \left( \kappa^{-1} n(-\varpi^{-r}) \begin{pmatrix} \theta_1 & \bar{\theta}_1 \\ 1 & 1 \end{pmatrix}^{-1} \iota_2(\theta_1) \begin{pmatrix} \theta_1 & \bar{\theta}_1 \\ 1 & 1 \end{pmatrix} n(\varpi^{-r}) \kappa \right) d\kappa \end{matrix} \right. ; $$
	and the \emph{Rankin-Selberg weights} $\mathrm{wt}(s,r,\E/\F)$ are computed in \cite[Proposition 4.5 \& \S 4.3]{CWZ21} (for a slightly different normalization), which we adapt (and correct) as follows.
	
\begin{proposition}
	We write $q$ for the cardinality of $\vo_{\F}/\varpi\vo_{\F}$ and $Z := q^s$, and define
	$$ L_{\vp}(1, \eta_{\E / \F}) = \left\{ \begin{matrix} (1+q^{-1})^{-1} & \text{if } \E / \F \text{ is unramified} \\ 1 & \text{if } \E / \F \text{ is ramified} \\ (1-q^{-1})^{-1} & \text{if } \E / \F \text{ is split} \end{matrix} \right. , $$
	where $\eta_{\E / \F}$ is the quadratic character associated with the quadratic extension $\E / \F$. Then we have
	$$ \frac{\mathrm{wt}(s; 0, \E / \F)}{\Vol(\GL_2(\vo_{\F}))} = 1. $$
	While for $r \geq 1$, we have:
\begin{itemize}
	\item[(1)] If $\E / \F$ is unramified, then
	$$ L_{\vp}(1, \eta_{\E / \F}) \frac{\mathrm{wt}(s; r, \E / \F)}{\Vol(\GL_2(\vo_{\F}))} = \frac{ (Z-q^{-1}Z^{-1}) (q^{\frac{1}{2}}Z)^r - (Z^{-1}-q^{-1}Z) (q^{\frac{1}{2}}Z^{-1})^r }{Z-Z^{-1}}. $$
	\item[(2)] If $\E / \F$ is ramified, then
	$$ L_{\vp}(1, \eta_{\E / \F}) \frac{\mathrm{wt}(s; r, \E / \F)}{\Vol(\GL_2(\vo_{\F}))} = \frac{ (Z-q^{-\frac{1}{2}}) (q^{\frac{1}{2}}Z)^r - (Z^{-1}-q^{-\frac{1}{2}}) (q^{\frac{1}{2}}Z^{-1})^r }{Z-Z^{-1}}. $$
	\item[(3)] If $\E / \F$ is split, then
	$$ L_{\vp}(1, \eta_{\E / \F}) \frac{\mathrm{wt}(s; r, \E / \F)}{\Vol(\GL_2(\vo_{\F}))} = \frac{ (Z+q^{-1}Z^{-1}-2q^{-\frac{1}{2}}) (q^{\frac{1}{2}}Z)^r - (Z^{-1}+q^{-1}Z-2q^{-\frac{1}{2}}) (q^{\frac{1}{2}}Z^{-1})^r }{Z-Z^{-1}}. $$
\end{itemize}
\label{ExpWts}
\end{proposition}

		\subsubsection{Computation at $\vp \mid \idl{N}$}
		
	By definition, we need to compute
	$$ I(s;n;f,\Phi) = \int_{\E^{\times} \backslash \GL_2(\F)} f(x^{-1}\iota(\lambda)x) \int_{\E^{\times}} \Phi((0,1)\iota(t)x) \norm[\det \iota(t)x]_{\F}^{\frac{1}{2}+s} d^{\times}t dx, $$
	where $\E=\F[\lambda]$ is a ramified extension over $\F$ and the embedding $\iota$ is given by
	$$ \iota(\lambda) = \begin{pmatrix} nN & -N \\ 1 & 0 \end{pmatrix}. $$
	Recall that $N\vo_{\F} = \idl{J}^{-2}\idl{N}, n \in \idl{J}$ and $\idl{J}=\delta \vo_{\F}$. Hence $\iota$ is not necessarily standard. Note that
	$$ \iota_1(x) := \begin{pmatrix} \delta & \\ & 1 \end{pmatrix} \iota(x) \begin{pmatrix} \delta & \\ & 1 \end{pmatrix}^{-1} \quad \Rightarrow \quad \iota_1(\delta \lambda) = \begin{pmatrix} n\delta N & -\delta^2 N \\ 1 & 0 \end{pmatrix}, $$
	which is the root of an Eisenstein polynomial in $\vo_{\F}[X]$. Hence $\vo_{\E} = \vo_{\F}[\delta \lambda]$ and $\iota_1$ is well-positioned with respect to $\delta \lambda$. Since $f$ is invariant by $\F^{\times}$ by definition (\ref{TFNArch}), we get
\begin{equation}
	I(s;n;f,\Phi) = \norm[\delta]_{\F}^{\frac{1}{2}+s} I_1,
\label{Stand}
\end{equation} 
	where $I_1$ is defined and has a decomposition as
\begin{align*}
	I_1 &:= \int_{\iota_1(\E^{\times}) \backslash \GL_2(\F)} f(x^{-1} \iota_1(\delta \lambda) x) \int_{\E^{\times}} \Phi((0,1) \iota_1(t)x) \norm[\det \iota_1(t)x]_{\F}^{\frac{1}{2}+s} d^{\times}t dx \\
	&= \zeta_{\E} \left( s+\frac{1}{2} \right) \cdot \sum_{r=0}^{\infty} \mathrm{RS}-\vO_{\delta \lambda}(r,f) \cdot \mathrm{wt}(s,r,\E/\F).
\end{align*} 

\begin{lemma}
	Write $q$ for the cardinality of $\vo_{\F}/\varpi\vo_{\F}$. The above Rankin-Selberg orbital integral $\mathrm{RS}-\vO_{\delta \lambda}(r,f)$ is non-vanishing only if $r=0$, and we have
	$$ \mathrm{RS}-\vO_{\delta \lambda}(0,f) = \left\{ \begin{matrix} q-1 & \text{if } n \in \varpi \idl{J} \\ -1 & \text{otherwise} \end{matrix} \right. . $$
\end{lemma}
\begin{proof}
	The coset decomposition
	$$ \gp{K} / \gp{K}_0[\vp] = \gp{K}_0[\vp] \sqcup \sideset{}{_{\alpha \in \vo_{\F} / \varpi \vo_{\F}}} \bigsqcup w \begin{pmatrix} 1 & 0 \\ \alpha & 1 \end{pmatrix} \gp{K}_0[\vp] $$
	implies the equation
\begin{align}
	\mathrm{RS}-\vO_{\delta \lambda}(r,f) &= \frac{1}{q+1} f(a(\varpi^r) \iota_1(\delta \lambda) a(\varpi^{-r})) + \label{RSORNeq} \\
	&\quad \frac{1}{q+1} \sideset{}{_{\alpha \in \vo_{\F} / \varpi \vo_{\F}}} \sum f \left( \begin{pmatrix} 1 & 0 \\ -\alpha & 1 \end{pmatrix} w^{-1} a(\varpi^r) \iota_1(\delta \lambda) a(\varpi^{-r}) w \begin{pmatrix} 1 & 0 \\ \alpha & 1 \end{pmatrix} \right). \nonumber
\end{align}
	For any $r \geq 0$, we have
	$$ a(\varpi^r) \iota_1(\delta \lambda) a(\varpi^{-r}) = \begin{pmatrix} n \delta N & -\varpi^r \delta^2 N \\ \varpi^{-r} & 0 \end{pmatrix} \left( \notin \begin{pmatrix} \varpi \vo_{\F} & \vo_{\F}^{\times} \\ \varpi\vo_{\F}^{\times} & \varpi\vo_{\F} \end{pmatrix} \right), $$
which never lies in the support of $f$. Hence the first term in (\ref{RSORNeq}) vanishes identically. Similarly,
\begin{multline*} 
	\begin{pmatrix} 1 & 0 \\ -\alpha & 1 \end{pmatrix} w^{-1} a(\varpi^r) \iota_1(\delta \lambda) a(\varpi^{-r}) w \begin{pmatrix} 1 & 0 \\ \alpha & 1 \end{pmatrix} = \\
	\begin{pmatrix} -\alpha \varpi^{-r} & -\varpi^{-r} \\ \varpi^r \delta^2 N + n \delta N \alpha + \varpi^{-r} \alpha^2 & n \delta N + \alpha \varpi^{-r} \end{pmatrix} \in \begin{pmatrix} \varpi\vo_{\F} & \vo_{\F}^{\times} \\ \varpi\vo_{\F}^{\times} & \varpi\vo_{\F} \end{pmatrix} 
\end{multline*}
	if and only if $r=0$ and $\alpha=0$, in which case we get by (\ref{TFNArch})
	$$ f \begin{pmatrix} 0 & -1 \\ \delta^2 N & n \delta N \end{pmatrix} = (q+1) \cdot \left\{ \begin{matrix} q-1 & \text{if } n \in \varpi \idl{J} \\ -1 & \text{otherwise} \end{matrix} \right. $$
	and conclude the desired formula.
\end{proof}

\begin{corollary}
	At a place $\vp \mid \idl{N}$ with cardinality of residue field equal to $q$, we have
	$$ I_{\vp}(s;n;f_{\vp}, \Phi_{\vp}) = \Vol(\GL_2(\vo_{\vp})) q^{-\left( \frac{1}{2}+s \right)\mathrm{ord}_{\vp}(\idl{J})} \zeta_{\E_{\vp}} \left( \frac{1}{2}+s \right) \cdot \left\{ \begin{matrix} q-1 & \text{if } \mathrm{ord}_{\vp}(n) \geq \mathrm{ord}_{\vp}(\idl{J})+1 \\ -1 & \text{otherwise} \end{matrix} \right. . $$
\label{NANPart}
\end{corollary}

		\subsubsection{Computation at $\vp \nmid \idl{N}$}
		
	As in the previous case, we associate to the embedding $\iota$ another embedding $\iota_1$ with the same formula. This time, $\iota_1$ is standard but not necessarily well-positioned. The relation (\ref{Stand}) still holds. Let $I_2$ be a well-positionization of $I_1$ associated with $(\iota_2, \theta_2)$. Then we have by (\ref{WellP}) and Lemma \ref{WellPExp}
	$$ I_1 = \norm[\varpi]_{\F}^{\left( \frac{1}{2}+s \right) a} I_2, $$
	where $a$ is the conductor exponent of the order $\vo_{\F}[\delta \lambda]$. By computing the discriminant of $\vo_{\F}[\delta \lambda]$ (see \cite[Lemma 4.1]{CWZ21} for detail), we get
	$$ \delta^2 \left( (nN)^2-4N \right) \vo_{\F} = \varpi^{2a} D_{\E/\F} $$
	where $D_{\E/\F}$ is the (local) discriminant of $\vo_{\E}$. The computation of $I_2$ is given in \cite[\S 4.4]{CWZ21}, which should be considered as a standard result for the unramified places. We simply recall the formula as
	$$ I_2 = \Vol(\GL_2(\vo_{\F})) \zeta_{\E} \left( \frac{1}{2}+s \right) q^{\frac{a}{2}} \frac{(q^{(a+1)s}-q^{-(a+1)s}) - \varepsilon_{\E} q^{-\frac{1}{2}} (q^{as}-q^{-as})}{q^s-q^{-s}}, $$
	where $\varepsilon_{\E}$ is defined by
	$$ \varepsilon_{\E} = \left\{ \begin{matrix} -1 & \text{if } \E/\F \text{ is unramified} \\ 0 & \text{if } \E/\F \text{ is ramified} \\ 1 & \text{if } \E/\F \text{ is split} \end{matrix} \right.. $$
	We summarize the above discussion in the following lemma.

\begin{lemma}
	Let $\E$ be the quadratic extension determined by $\lambda = \begin{pmatrix} nN & -N \\ 1 & 0 \end{pmatrix}$ with discriminant $D_{\E/\F}$. Let $D_{\vp}$ be the local component of $D_{\E/\F}$ at a place $\vp \nmid \idl{N}$, of which the cardinality of residue field is equal to $q$. Let $a \in \Z_{\geq 0}$ be defined by
	$$ \delta_{\vp}^2 \left( (nN)^2 - 4N \right) \vo_{\vp} = \varpi_{\vp}^{2a} D_{\vp}. $$
	Then we have
\begin{align*}
	I_{\vp}(s;n;f_{\vp},\Phi_{\vp}) &= \Vol(\GL_2(\vo_{\vp})) q^{-\left( \frac{1}{2}+s \right)\mathrm{ord}_{\vp}(\idl{J})} \zeta_{\E} \left( \frac{1}{2}+s \right) \cdot \\
	&\quad q^{-as} \frac{(q^{(a+1)s}-q^{-(a+1)s}) - \varepsilon_{\E,\vp} q^{-\frac{1}{2}} (q^{as}-q^{-as})}{q^s-q^{-s}}.
\end{align*} 
\label{NAUnr}
\end{lemma}

	\subsection{A (Multiple) Dirichlet Series}

	Comparing Corollary \ref{NANPart} \& Lemma \ref{NAUnr} \& (\ref{ZagLFVar}) and taking Remark \ref{MaxCpMeas} into account, we get
	$$ I(s;n;f,\Phi) = D_{\F}^{-\frac{3}{2}} I_{\infty}(s;n;f_{\infty},\Phi_{\infty}) \cdot \zeta_{\F} \left( s+\frac{1}{2} \right) L^{(\idl{N})} \left( s+\frac{1}{2},(nuN)^2-4uN; \idl{J} \right) \cdot A(n,\idl{N}) $$ 
	where we have written $I_{\infty}(\cdots) = \prod_{v \mid \infty} I_v(\cdots)$ for simplicity, and recall
	$$ A(n,\idl{N}) := \prod_{\vp \mid \idl{N}} \left\{ \begin{matrix} \Nr(\vp)-1 & \text{if } n \in \vp \idl{J} \\ -1 & \text{otherwise} \end{matrix} \right. . $$
	Consequently, we get
\begin{equation}
	\frac{\Res_{s=\frac{1}{2}} I(s)}{D_{\F}^{-\frac{3}{2}} \zeta_{\F}^{(-1)}(1)} = \sum_{u \in U_2} \sum_{n \in \idl{J}/\{ \pm 1 \}} 2^{-\mathbbm{1}_{n=0}} I_{\infty} \left(
\frac{1}{2};n;f_{\infty},\Phi_{\infty} \right) L^{(\idl{N})}(1, (nuN)^2-4uN; \idl{J}) A(n,\idl{N}).
\label{IndBasis}
\end{equation} 

	It turns out that $L^{(\idl{N})}(1, (nuN)^2-4uN; \idl{J}) A(n,\idl{N})$ for $n \in \idl{J}/\{ \pm 1 \}$ is the finite part of the individual terms on the geometric side of the trace formula for our chosen test functions $f_{\vp}$ given in (\ref{TFNArch}). We will need the analytic continuation of several Dirichlet series associated to these $L$-values, especially in the case that some component of the weight vector $k_v=1$. For this purpose, we will apply the trace formula in the opposite direction, and will need to extend the validity of the trace formula to a sufficiently large class of text functions $f_{\infty}$ beyond $\Sch(\PGL_2(\R)^d // \PO_2(\R)^d)$, the space of bi-$\PO_2(\R)^d$-invariant Schwartz functions (see Theorem \ref{AutoKernSpecDSch}). 
	
	In this subsection, we present the simplest case $\F=\Q$ to illustrate the essential idea. To this end, we review and adapt a method given by Hejhal in \cite[\S 1.7]{He76}. At the infinite place, the Selberg transforms can be summarized by ($\Phi_{\infty}$ is different from other parts of the paper)
\begin{equation} 
	\left\{ \begin{matrix} f_{\infty} \begin{pmatrix} \lambda & \\ & 1 \end{pmatrix} = \Phi_{\infty}(\lambda + \lambda^{-1} - 2), & Q_{\infty}(x) = 2 \int_0^{\infty} \Phi_{\infty}(x+\nu^2) d\nu, \\
	g_{\infty}(u) = Q_{\infty}(e^u+e^{-u}-2), & h_{\infty}(r) = \int_{\R} g_{\infty}(u) e^{iru} du; \end{matrix} \right.
\label{SelbergTrans}
\end{equation}
	while the orbital integral is given by
\begin{equation} 
	\int_{\R} f_{\infty} \left( \begin{pmatrix} 1 & -x \\ & 1 \end{pmatrix} \begin{pmatrix} y & \\ & 1 \end{pmatrix} \begin{pmatrix} 1 & x \\ & 1 \end{pmatrix} \right) dx = \frac{g_{\infty}(\log y)}{\sqrt{y}-\sqrt{y}^{-1}}, \quad y > 1. 
\label{OrbIntR}
\end{equation}
	Hejhal \cite{He76_Survey, He76} considered even $\Phi_{\infty} \in \Cont_c^{\infty}(\R)$. This is not the correct class of functions\footnote{This is the source of ``a slight problem'' in the proof of \cite[Theorem 1.7.5]{He76}, which however does not seem to be ``easily handled''.} obtainable from the first equation in (\ref{SelbergTrans}), where the smoothness of $f_{\infty} \in \Cont_c^{\infty}(\PGL_2(\R) // \PO_2(\R))$ is defined by the differential operator $\lambda \partial_{\lambda}$. Equivalently, the corresponding class of $\Phi_{\infty}$ is
	$$ \mathcal{D}(\R_+) := \left\{ \Phi \in \Cont^{\infty}((0,\infty)) \ \middle| \ \exists \varphi \in \Cont_c^{\infty}(\R), \varphi(u)=\varphi(-u) \text{ such that } \Phi(e^u+e^{-u}-2) = \varphi(u), u > 0 \right\}. $$
	Observing another form of the orbital integral (\ref{OrbIntR})
\begin{align*}
	g_{\infty}(\log y) &= \extnorm{\sqrt{y}-\sqrt{y}^{-1}} \int_{\R} f_{\infty} \left( \begin{pmatrix} y & \\ & 1 \end{pmatrix} \begin{pmatrix} 1 & (1-y^{-1})x \\ & 1 \end{pmatrix} \right) dx \\
	&= \sqrt{y} \int_{\R} f_{\infty} \left( \begin{pmatrix} y & \\ & 1 \end{pmatrix} \begin{pmatrix} 1 & x \\ & 1 \end{pmatrix} \right) dx, \qquad y \neq 1,
\end{align*}
	we see that $g_{\infty}(u) \in \Cont_c^{\infty}(\R)$ if $f_{\infty} \in \Cont_c^{\infty}(\PGL_2(\R) // \PO_2(\R))$. The converse is also true, and follows directly from the following result due to Helgason and Gangolli, called the Paley-Wiener theorem for the spherical Fourier (or Harish-Chandra) transform.
	
\begin{theorem}
	The transform $f_{\infty} \to h_{\infty}$ is an isomorphism of $\Cont_c^{\infty}(\PGL_2(\R) // \PO_2(\R))$ and the Paley-Wiener space of even functions, i.e., the space of even entire functions $h$ such that there is $r > 0$ and $\gamma_N$ for any $N \in \Z_{\geq 0}$ such that $\norm[h(z)] \leq \gamma_N (1+\norm[z])^{-N} e^{r \norm[\Im z]}$.
\end{theorem}
\begin{proof}
	See \cite[Introduction, Theorem 4.2]{Hel84} or \cite[Theorem 6.6.8]{GV88}.
\end{proof}

	Let $(g_{\infty},h_{\infty})$ be a Fourier pair of functions related by the fourth equation in (\ref{SelbergTrans}). First assume $g_{\infty} \in \Cont_c^{\infty}(\R)$ and even. Then it is given by a $f_{\infty} \in \Cont_c^{\infty}(\PGL_2(\R) // \PO_2(\R))$ via (\ref{SelbergTrans}). Inserting this $f_{\infty}$ and $f_p$ given in (\ref{TFNArch}) below in the trace formula, we get
\begin{align}
	&\quad \sum_{r} h_{\infty}(r) \varepsilon_{r} \nonumber \\
	&= \sum_n  g_{\infty} \left( \log \left( \frac{nN+\sqrt{(nN)^2-4N}}{2} \right) \right) \frac{ L^{(N)}(1, (nN)^2-4N) A(n,N) }{\sqrt{\frac{(nN)^2-4N}{4N} }},
\label{MaassBiasF}
\end{align}
	where in the notation
\begin{itemize}
	\item[(1)] $r$ runs over those elements in $\C$ for which $\pi(\norm^{ir}, \norm^{-ir})$ appears as the infinite component of a cuspidal representation $\pi_r$ (with multiplicity) of usual conductor $\idl{N}^3$, 
	\item[(2)] $\varepsilon_r \in \{ \pm 1 \}$ is the epsilon factor of $\pi_r$,
	\item[(3)] $n \in \Z/\{ \pm 1 \}$ such that $(nN)^2-4N$ is not a square and is positive.
\end{itemize}
	We want to extend the validity of (\ref{MaassBiasF}) to a larger class of pairs of functions $(g_{\infty}, h_{\infty})$.
	
\begin{proposition}
	Let $\delta, \xi > 0$ be fixed and arbitrary. Let $(g_{\infty},h_{\infty})$ be a Fourier pair of even functions related by the fourth equation in (\ref{SelbergTrans}). Assume that $h(r)$ is holomorphic on $\norm[\Im r] \leq \frac{1}{2}+\delta$ and satisfy $ \norm[h_{\infty}(r)] \ll (1+\norm[r])^{-2-\xi}$. Then the trace formula (\ref{MaassBiasF}) holds for such pair $(g_{\infty}, h_{\infty})$.
\label{TFExtQ}
\end{proposition}
\begin{remark}
	Note that by a routine contour shift argument, we have $\norm[g_{\infty}(u)] \ll e^{-\left( \frac{1}{2}+\delta \right)\norm[u]}$. By Weyl's law, Lemma \ref{GenZagLBd} and $\norm[A(n,N)] \leq \varphi(N)$, both sides of (\ref{MaassBiasF}) are absolutely convergent under the assumption. 
\end{remark}
\begin{proof}
	First consider the special case that $h_{\infty}(r)$ satisfies the required bound for \emph{any} $\xi > 0$. Hence $g_{\infty}^{(n)}(u)$ satisfies the bound mentioned in the above remark for any integer $n \geq 0$. Let $\phi(x) \in \Cont^{\infty}(\R_{\geq 0})$ such that: (a) $0 \leq \phi(x) \leq 1$; (b) $\phi(x)=1$ for $0 \leq x \leq 1$; (c) $\phi(x) = 0$ for $x \geq 2$; (d) $\phi(x)$ is monotonically decreasing. Define $g_m(u) := g_{\infty}(u)\phi_m(u)$ for
	$$ \phi_m(x) = \left\{ \begin{matrix} 1, & \norm[x] \leq m \\ \phi(\norm[x]-m), & \norm[x] > m \end{matrix} \right. . $$
	Then $g_m \in \Cont_c^{\infty}(\R)$ and is even. Hence the trace formula (\ref{MaassBiasF}) holds for the Fourier pairs $(g_m,h_m)$. We of course have $\norm[g_m(u)] \leq \norm[g_{\infty}(u)]$. By integration by parts and the bounds $\norm[\phi_m^{(n)}] \ll_n 1$ uniformly in $m$, we also have: (1) $h_m(r) = \int_{\R} g_{\infty}(u) \phi_m(u) e^{iru} du \ll (1+\norm[r])^{-3}$; (2) $h_m(r) \to h(r)$ as $m \to \infty$ uniformly for $\norm[\Im(r)] \leq 1/2$. Hence the dominated convergence theorem applies to both sides of (\ref{MaassBiasF}) for $(g_m,h_m)$, and gives the required trace formula for $(g_{\infty}, h_{\infty})$. For the general case, we consider the family of Fourier pairs $(g_{\epsilon},h_{\epsilon})$ for $0 < \epsilon < 1$ defined by
	$$ g_{\epsilon}(u) := \int_{\R} g_{\infty}(x) \cdot \frac{1}{\sqrt{2\pi \epsilon}} e^{-\frac{(x-u)^2}{4\epsilon}} dx, \quad h_{\epsilon}(r) = h(r) e^{-\epsilon r^2}. $$
	Then $(g_{\epsilon},h_{\epsilon})$ lies in the above special case and the trace formula (\ref{MaassBiasF}) holds for them. For any $N \geq 1$, we have uniformly in $\epsilon$
\begin{align*}
	g_{\epsilon}(u) &= \int_{\norm[x-u] \geq \norm[u]/N} g_{\infty}(x) \cdot \frac{1}{\sqrt{2\pi \epsilon}} e^{-\frac{(x-u)^2}{4\epsilon}} dx + \int_{\norm[x-u] \leq \norm[u]/N} g_{\infty}(x) \cdot \frac{1}{\sqrt{2\pi \epsilon}} e^{-\frac{(x-u)^2}{4\epsilon}} dx \\
	&\ll_N \max_{x \in \R} \norm[g_{\infty}(x)] \cdot \int_{\norm[x] \geq \norm[u]/(N\sqrt{\epsilon})} \frac{1}{\sqrt{2\pi}} e^{-\frac{x^2}{4}} dx + \max_{\norm[x-u] \leq \norm[u]/N} \norm[g_{\infty}(x)] \\
	&\ll_N e^{-\frac{u^2}{4N^2}} + e^{-\left( \frac{1}{2}+\delta \right) (1-N^{-1}) \norm[u]} \ll_N e^{-\left( \frac{1}{2}+\delta \right) (1-N^{-1}) \norm[u]}.
\end{align*}
	Taking $N$ large enough but fixed, the right hand side becomes $e^{-\left( \frac{1}{2}+\delta' \right) \norm[u]}$ for some $\delta' > 0$. Thus the geometric sides converge to the geometric side for $g_{\infty}$ by the dominated convergence theorem. The spectral sides obviously converge to the spectral side for $h_{\infty}$, too. We conclude the proof.
\end{proof}	
	
\begin{corollary}
	The following Dirichlet series
	$$ D_N(s) := \sideset{}{_{\substack{n=1 \\ (nN)^2-4N \neq \square}}^{\infty}} \sum n^{-s} L^{(N)}(1, (nN)^2-4N) A(n,N) $$
	is absolutely convergent for $\Re s > 1$ and has a meromorphic continuation which is regular at $s=1$.
\end{corollary}
\begin{proof}
	By Proposition \ref{TFExtQ}, we apply the trace formula (\ref{MaassBiasF}) with the Fourier pair \cite[(7.9) \& (7.10)]{He76_Survey}
	$$ g_{\infty}(u) = \frac{1}{2s} e^{-s \norm[u]} - \frac{1}{2\beta} e^{-\beta \norm[u]}, \quad h_{\infty}(r) = \frac{1}{r^2+s^2} - \frac{1}{r^2+\beta^2}, $$
	where $\beta > 2$ is large and fixed. We obtain
\begin{align*} 
	&\quad \frac{1}{2} \left( \frac{1}{2s} - \frac{1}{2\beta} \right) L^{(N)}(1,-4N) \varphi(N) + \sideset{}{_{\substack{n \geq 1 \\ (nN)^2-4N \neq \square}}} \sum \\
	&\quad \left\{ \frac{1}{2s} \left( \frac{n \sqrt{N} + \sqrt{\norm[n^2N-4]}}{2} \right)^{-s} - \frac{1}{2\beta} \left( \frac{n \sqrt{N} + \sqrt{\norm[n^2N-4]}}{2} \right)^{-\beta} \right\} \cdot \frac{L^{(N)}(1,(nN)^2-4N) A(n,N)}{\sqrt{\extnorm{\frac{(nN)^2-4N}{4N}}}} \\
	&= \sum_k \left( \frac{1}{r_k^2+s^2} - \frac{1}{r_k^2+\beta^2} \right) \varepsilon_k,
\end{align*}
	where the sum of $k$ runs over the Maass new forms of Lapacian eigenvalue $1/4+r_k^2$ and level $N^3$, and $\varepsilon_k \in \{ \pm 1 \}$ is the root number. By Weyl's law the right hand side has meromorphic continuation to $s \in \C$, which is regular at $s=1/2$ by any bound towards the Ramanujan-Petersson conjecture. This implies readily the required properties of $D_N(s+1/2)$ since their difference is absolutely convergent hence analytic at $s=1/2$.
\end{proof}

	\subsection{Global Computation}
	
	Recall that $f_{\infty}=\otimes_{v \mid \infty} f_v$, and each $f_v$ has a parameter $\delta_v > 0$. For convenience of notation, we number the archimedean places as $\left\{ v_j \ \middle| \ 1 \leq j \leq d \right\}$, and abbreviate the subscript $v_j$ as $j$. We shall let $\delta_j \to 0^+$ consecutively for $1 \leq j \leq d$ in the above equation, and establish the following lemma by induction.
\begin{lemma}
	For $0 \leq l \leq d$, we have
\begin{align*}
	&\quad \lim_{\delta_1 \to 0^+} \cdots \lim_{\delta_l \to 0^+} \frac{\Res_{s=\frac{1}{2}} I(s)}{D_{\F}^{-\frac{3}{2}} \zeta_{\F}^{(-1)}(1)} \\
	&= \sum_{u \in U_2} \sum_{n \in \idl{J}/\{ \pm 1 \}} 2^{-\mathbbm{1}_{n=0}} \left( \prod_{j=1}^l \frac{(4N_j)^{k_j}}{2\pi} \Re \left( \frac{\mathbbm{1}_{\geq 0}((4uN-(nuN)^2)_j)}{\left( \sqrt{(4uN-(nuN)^2)_j} +i (nuN)_j \right)^{2k_j-1}} \right) \right) \cdot \\
	&\quad \prod_{j=l+1}^d I_j \left( \frac{1}{2}; n; f_j; \Phi_j \right) \cdot L^{(\idl{N})}(1, (nuN)^2-4uN; \idl{J}) A(n,\idl{N}).
\end{align*}
\end{lemma}
\begin{proof}
	The case $l=0$ is (\ref{IndBasis}). Assume the validity of the desired equation for some $0 \leq l < d$. Writing $1 = \mathbbm{1}_{\geq 0}((4uN-(nuN)^2)_{l+1}) + \mathbbm{1}_{< 0}((4uN-(nuN)^2)_{l+1})$, we have
\begin{align*}
	&\quad \lim_{\delta_1 \to 0^+} \cdots \lim_{\delta_l \to 0^+} \frac{\Res_{s=\frac{1}{2}} I(s)}{D_{\F}^{-\frac{3}{2}} \zeta_{\F}^{(-1)}(1)} \\
	&= \sum_{u \in U_2} \sum_{n \in \idl{J}/\{ \pm 1 \}} 2^{-\mathbbm{1}_{n=0}} \left( \prod_{j=1}^l \frac{(4N_j)^{k_j}}{2\pi} \Re \left( \frac{\mathbbm{1}_{\geq 0}((4uN-(nuN)^2)_j)}{\left( \sqrt{(4uN-(nuN)^2)_j} +i (nuN)_j \right)^{2k_j-1}} \right) \right) \cdot \\
	&\quad \mathbbm{1}_{\geq 0}((4uN-(nuN)^2)_{l+1}) \prod_{j=l+1}^d I_j \left( \frac{1}{2}; n; f_j; \Phi_j \right) \cdot L^{(\idl{N})}(1, (nuN)^2-4uN; \idl{J}) A(n,\idl{N}) + \\
	&\quad \sum_{u \in U_2} \sum_{n \in \idl{J}/\{ \pm 1 \}} 2^{-\mathbbm{1}_{n=0}} \left( \prod_{j=1}^l \frac{(4N_j)^{k_j}}{2\pi} \Re \left( \frac{\mathbbm{1}_{\geq 0}((4uN-(nuN)^2)_j)}{\left( \sqrt{(4uN-(nuN)^2)_j} +i (nuN)_j \right)^{2k_j-1}} \right) \right) \cdot \\
	&\quad \mathbbm{1}_{< 0}((4uN-(nuN)^2)_{l+1}) \prod_{j=l+1}^d I_j \left( \frac{1}{2}; n; f_j; \Phi_j \right) \cdot L^{(\idl{N})}(1, (nuN)^2-4uN; \idl{J}) A(n,\idl{N}) \\
	&=: S_1 + S_2.
\end{align*}
	By Lemma \ref{SpAC}, we have
\begin{align*}
	\lim_{\delta_{l+1} \to 0^+} S_1 &= \sum_{u \in U_2} \sum_{n \in \idl{J}/\{ \pm 1 \}} 2^{-\mathbbm{1}_{n=0}} \left( \prod_{j=1}^{l+1} \frac{(4N_j)^{k_j}}{2\pi} \Re \left( \frac{\mathbbm{1}_{\geq 0}((4uN-(nuN)^2)_j)}{\left( \sqrt{(4uN-(nuN)^2)_j} +i (nuN)_j \right)^{2k_j-1}} \right) \right) \cdot \\
	&\quad \prod_{j=l+2}^d I_j \left( \frac{1}{2}; n; f_j; \Phi_j \right) \cdot L^{(\idl{N})}(1, (nuN)^2-4uN; \idl{J}) A(n,\idl{N}),
\end{align*} 
	which is the desired right hand side for $l+1$. Note that we can take the limit in the sum because the summand is a function in $n$ of uniformly rapid decay in every direction in $\R^d$. It remains to show $S_2 \to 0$ as $\delta_{l+1} \to 0^+$. Since $\Mellin{h_{\delta_{l+1}}}(s) = \delta_{l+1}^{-s} \Gamma(s)$, we may apply Lemma \ref{GenAC} to get
	$$ S_2 = \int_{(c)} \delta_{l+1}^{-s} D(2s+2k_{l+1}-1) P_{k_{l+1}}(s) \Gamma(s) \frac{ds}{2\pi i}, $$
	where the Dirichlet series $D(s)$ is defined for $\Re (s) > 1$ by
\begin{align*}
	D(s) &= \sum_{u \in U_2} \sum_{n \in \idl{J}/\{ \pm 1 \}} 2^{-\mathbbm{1}_{n=0}} \left( \prod_{j=1}^l \frac{(4N_j)^{k_j}}{2\pi} \Re \left( \frac{\mathbbm{1}_{\geq 0}((4uN-(nuN)^2)_j)}{\left( \sqrt{(4uN-(nuN)^2)_j} +i (nuN)_j \right)^{2k_j-1}} \right) \right) \cdot \\
	&\quad \mathbbm{1}_{< 0}((4uN-(nuN)^2)_{l+1}) \cdot \norm[nu]_{l+1}^{-s} \cdot \prod_{j=l+2}^d I_j \left( \frac{1}{2}; nu; f_j; \Phi_j \right) \cdot L^{(\idl{N})}(1, (nuN)^2-4uN; \idl{J}) A(n,\idl{N}).
\end{align*}
	We are reduced to show that $D(s)$ has meromorphic continuation to $-\epsilon< \Re(s) + 1 - 2k_{l+1} < 0$ for some $\epsilon > 0$, which is regular at $\Re(s) = 2k_{l+1}-1$. 
	
\noindent Note that the hypothesis of induction is the geometric side of a trace formula, whose spectral side can easily be deduced from (\ref{ArchFinal}). Namely, let $I_l=\left\{ j \ \middle| \ 1 \leq j \leq l \right\}$, and write $\mathcal{A}_+(I_l;\idl{N}^3)$ for the set of automorphic representations $\pi = \otimes_v' \pi_v$ of level $\idl{N}^3$ such that $\pi_j \simeq D_{2k_j}$ for $1 \leq j \leq l$. Then we have
	$$ \lim_{\delta_1 \to 0^+} \cdots \lim_{\delta_l \to 0^+} \frac{\Res_{s=\frac{1}{2}} I(s)}{D_{\F}^{-\frac{3}{2}} \zeta_{\F}^{(-1)}(1)} = \frac{D_{\F}}{2} \sum_{\pi \in \mathcal{A}_+(I_l; \idl{N}^3)} \varepsilon \left( \frac{1}{2},\pi \right) \prod_{j=l+1}^d \varepsilon \left( \frac{1}{2}, \pi_j \right)^{-1} \Tr(\pi_j(f_j)). $$
	We note that the equality of the geometric and spectral sides is an equality of distributions on $\otimes_{j=l+1}^d f_j \in \Sch(\PGL_2(\R))^{d-l}$, in particular on $f_{l+1} \in \Sch(\PGL_2(\R))$. We extend the validity of this trace formula as in the case of $\Q$ given in Proposition \ref{TFExtQ}, and obtain for $\Re s \gg 1$ and a fixed $\beta \gg 1$
\begin{align*}
	&\quad \sum_{u \in U_2} \sum_{n \in \idl{J}/\{ \pm 1 \}} 2^{-\mathbbm{1}_{n=0}} \left( \prod_{j=1}^l \frac{(4N_j)^{k_j}}{2\pi} \Re \left( \frac{\mathbbm{1}_{\geq 0}((4uN-(nuN)^2)_j)}{\left( \sqrt{(4uN-(nuN)^2)_j} +i (nuN)_j \right)^{2k_j-1}} \right) \right) \cdot \\
	&\quad \left\{ \frac{1}{2s} \left( \frac{n \sqrt{uN} + \sqrt{\norm[n^2uN-4]}}{2} \right)_{l+1}^{-s} - \frac{1}{2\beta} \left( \frac{n \sqrt{uN} + \sqrt{\norm[n^2uN-4]}}{2} \right)_{l+1}^{-\beta} \right\} \left( \sqrt{\extnorm{\frac{(nuN)^2-4uN}{4uN}}_{l+1}} \right)^{-1} \cdot \\
	&\quad \prod_{j=l+2}^d I_j \left( \frac{1}{2}; nu; f_j; \Phi_j \right) \cdot L^{(\idl{N})}(1, (nuN)^2-4uN; \idl{J}) A(n,\idl{N}) \\
	&= \frac{D_{\F}}{2} \sum_{\pi \in \mathcal{A}_+(I_l; \idl{N}^3)} \varepsilon \left( \frac{1}{2},\pi \right) \prod_{j=l+2}^d \varepsilon \left( \frac{1}{2}, \pi_j \right)^{-1} \Tr(\pi_j(f_j)) \cdot \varepsilon \left( \frac{1}{2},\pi_{l+1} \right) \left( \frac{1}{r(\pi_{l+1})^2+s^2} - \frac{1}{r(\pi_{l+1})^2+\beta^2} \right),
\end{align*}
	where $r=r(\pi_{l+1}) \in \R \cup (-1/2+\RamC, 1/2-\RamC)i$ is the spectral parameter of $\pi_{l+1} \simeq \pi(\norm^{ir}, \norm^{-ir})$. The right hand side has meromorphic continuation which is regular in $\Re(s) > 1/2-\RamC$. It readily implies the meromorphic continuation of $D(s)$ regular in $\Re (s) > 1-\RamC$, and we are done.
\end{proof}

	It is clear that Theorem \ref{MainGenF} (2) follows from the above lemma by comparing the case $l=d$ with (\ref{ArchFinal}).

\section{Concrete Examples}

	\subsection{The Case Over $\Q$}

\begin{proof}[Proof of Corollary \ref{MainQ}]
	(1) In this case, we can take $\idl{J}=\Z$ and $U_2=\{ 1 \}$. The only $n \in \Z/\{ \pm 1 \}$ such that $4N-n^2N^2 > 0$ is $n=0$. Hence we get
	$$ B(k,N^3) = \frac{\sqrt{N}}{\pi} L^{(N)}(1,-4N) \varphi(N). $$
	If $N \equiv 3 (4)$, then $-4N=Dl^2$ with $D=-N, l =2$. We have 
	$$ \eta_{-N}(2)= \left\{ \begin{matrix} 1 & \text{if } N \equiv 7 (8) \\ -1 & \text{if } N \equiv 3 (8) \end{matrix} \right.. $$
	Hence by Dirichlet's class number formula we conclude the first two formulas with
	$$ L^{(N)}(1,-4N) = L(1,\eta_{-N}) \cdot 2^{-\frac{1}{2}} \frac{(2-2^{-1})-\eta_{-N}(2)2^{-\frac{1}{2}}(2^{\frac{1}{2}}-2^{-\frac{1}{2}})}{2^{\frac{1}{2}}-2^{-\frac{1}{2}}} = \frac{2\pi h(-N)}{\omega(-N) \sqrt{N}} \cdot \left\{ \begin{matrix} 1 & \text{if } N\equiv -1 (8) \\ 2 & \text{if } N \equiv 3 (8) \end{matrix} \right. , $$
	where $\omega(-N)=2$ is the number of units in $\Q[\sqrt{-N}]$. If $N \equiv 1,2(4)$, then $-4N = Dl^2$ with $D=-4N, l=1$. The desired formula follows by the same argument.
	
\noindent (2) The only $n \in \Z/\{ \pm 1 \}$ such that $4N-n^2N^2 > 0$ are $n=0,1$. Hence we get
	$$ B(k,8) = \frac{1\sqrt{2}}{\pi} L^{(2)}(1,-8) - \frac{2\sqrt{2}}{\pi} \Re \left( e^{-i\frac{\pi}{4}(2k-1)} \right) L^{(2)}(1,-4). $$
	Since $-8$ and $-4$ are both fundamental discriminants with $h(-8)=h(-4)=1$ \& $\omega(-8)=2, \omega(-4)=4$, we get
	$$ L^{(2)}(1,-8) = L(1,\eta_{-8}) = \frac{2\pi h(-8)}{\omega(-8)\sqrt{8}} = \frac{\pi}{2\sqrt{2}}, \quad L^{(2)}(1,-4) = \frac{2\pi h(-4)}{\omega(-4)\sqrt{4}} = \frac{\pi}{4}. $$
	The desired formula follows readily.
	
\noindent (3) The only $n \in \Z/\{ \pm 1 \}$ such that $4N-n^2N^2 > 0$ are $n=0,1$. Hence we get
	$$ B(k,27) = \frac{2\sqrt{3}}{\pi} L^{(3)}(1,-12) - \frac{2\sqrt{3}}{\pi} \Re \left( e^{-i\frac{\pi}{3}(2k-1)} \right) L^{(3)}(1,-3). $$
	We have $-12=Dl^2$ with $D=-3, l=2$, and $\eta_{-3}(2)=-1$. Hence
	$$ L^{(3)}(1,-12) = L(1,\eta_{-3}) \cdot 2 = \frac{2\pi}{3\sqrt{3}}, \quad L^{(3)}(1,-3) = L(1,\eta_{-3}) = \frac{\pi}{3\sqrt{3}}. $$
	The desired formula follows readily.
\end{proof}

	\subsection{The Case Over $\Q(\sqrt{2})$}
	
	The ring of integers of $\F=\Q(\sqrt{2})$ is $\Z[\sqrt{2}]$, which is a principal ideal domain. Hence we can choose $\idl{J}=\vo_{\F}$. The discriminant is $D_{\F} = 8$. The group of units is $\vo_{\F}^{\times} = \left\{ \pm (1+\sqrt{2})^k \ \middle| \ k \in \Z \right\}$, hence the number of roots of unity is $\omega_{\F}=2$ and the group of totally positive units is $U_+ = \left\{ (1+\sqrt{2})^{2k} \ \middle| \ k \in \Z \right\} = (\vo_{\F}^{\times})^2$. In particular $U_2 = \{ 1 \}$. An odd rational prime $p$ is split in $\Q(\sqrt{2})$ if and only if $p \equiv \pm 1 (8)$, otherwise it is inert.
	
\begin{lemma}
	If $N > 3$, then all $n \in \vo_{\F}$ such that $4N-(nN)^2$ is totally positive are reduced to $n=0$; while if $N=3$, all such $n$ are $0, \pm 1$.
\end{lemma}
\begin{proof}
	If $N > 3$, then $\Nr(n)^2 < (4/N)^2 \leq 1$, hence $\Nr(n)=0$ since it is an integer. Thus $n=0$. If $N=3$, we have an extra case $\Nr(n)=\pm 1$, or $n = \pm (1+\sqrt{2})^k$. But $4/3 < (1+\sqrt{2})^2$, thus $k=0$.
\end{proof}

\noindent Note that by definition $A(0,N) = \varphi_{\F}(N)$ and $A(1,3) = -1$. We can simplify the formula as
\begin{equation}
	B(\vec{k},N^3) = \frac{2\sqrt{2}}{\pi^2} N \varphi_{\F}(N) L^{(N)}(1,-4N) - \mathbbm{1}_{N=3} \frac{12\sqrt{2}}{\pi^2} L^{(3)}(1,-3) \prod_{v \mid \infty} \Re \left( \frac{1}{2}+\frac{\sqrt{3}}{2}i \right)^{1-2k_v}.
\label{BiasQSq2Sim}
\end{equation} 

	By definition, $L^{(N)}(1,-4N)$ has two parts. To explicitly write them, we must first determine the relative discriminant $D_{\E/\F}$ for the bi-quadratic field $\E := \F[\sqrt{-4N}] = \Q[\sqrt{2},\sqrt{-N}]$. By the formula \cite[(7.23)]{FT91}, we have
	$$ \norm[D_{\E}] = \norm[ D_{\Q(\sqrt{2})} D_{\Q(\sqrt{-N})} D_{\Q(\sqrt{-2N})} ] = \left\{ \begin{matrix} (8N)^2 & \text{if } N \equiv 3 (4) \\ 4(8N)^2 & \text{if } N \equiv 1 (4) \end{matrix} \right. . $$
	From the formula $\norm[D_{\E}] = \Nr(D_{\E/\F}) \norm[D_{\F}]^2$, we deduce
	$$ \Nr(D_{\E/\F}) = \left\{ \begin{matrix} N^2 & \text{if } N \equiv 3 (4) \\ 4N^2 & \text{if } N \equiv 1 (4) \end{matrix} \right.. $$
	Note that $\vo[\sqrt{-N}] \subset \vo_{\E}$, which is a $\vo$-sub-order of discriminant $4N \vo$. Thus $D_{\E/\F} \mid 4N \vo$. Hence
	$$ D_{\E/\F} = \left\{ \begin{matrix} N \vo & \text{if } N \equiv 3 (4) \\ 2N \vo & \text{if } N \equiv 1 (4) \end{matrix} \right. . $$
	Writing $4N\vo = D_{\E/\F} l^2$, we get
	$$ l = \left\{ \begin{matrix} 2 \vo & \text{if } N \equiv 3 (4) \\ \sqrt{2} \vo & \text{if } N \equiv 1 (4) \end{matrix} \right.. $$
	We can thus write (with $q := \Nr(\vp)$)
\begin{align}
	L^{(N)}(1,-4N) &= L(1,\eta_{\E/\F}) \prod_{\vp^e \parallel l, \vp \nmid N} q^{-\frac{e}{2}} \frac{(q^{\frac{e+1}{2}} - q^{-\frac{e+1}{2}}) - \eta_{\E/\F}(\varpi_{\vp}) q^{-\frac{1}{2}} (q^{\frac{e}{2}} - q^{-\frac{e}{2}})}{q^{\frac{1}{2}}-q^{-\frac{1}{2}}} \nonumber \\
	&= L(1,\eta_{\E/\F}) \cdot \left\{ \begin{matrix} \frac{5}{2} & \text{if } N \equiv 3 (8) \\ 1 & \text{if } N \equiv 7 (8) \\ \frac{3}{2} & \text{if } N \equiv 1 (4) \end{matrix} \right.. \label{LN1}
\end{align}
	Note that in the case $N \equiv 1 (4)$, $\vp=\sqrt{2}\vo$ is ramified in $\E$ since $D_{\E/\F} = 2N \vo$; while in the case $N \equiv 3 (4)$, we distinguish whether $\vp$ is split or inert in $\E$ according as $2$ is split or inert in $\Q(\sqrt{-N})$. By Dirichlet's class number formula, we have
\begin{equation}
	L(1,\eta_{\E/\F}) = \frac{\zeta_{\E}^{(-1)}(1)}{\zeta_{\F}^{(-1)}(1)} = \pi^2 \cdot \frac{\omega_{\F}}{\omega_{\E}} \cdot \frac{R_{\E}}{R_{\F}} \cdot \frac{h_{\E}}{h_{\F}} \cdot \sqrt{\frac{\norm[D_{\F}]}{\norm[D_{\E}]}},
\label{L1}
\end{equation} 
	where $\omega_*$, resp. $R_*$, resp. $h_*$ is the cardinality of the group $\mu_*$ of roots of unity, resp. the regulator, resp. the class number of $*$. By \cite[\Rmnum{5}.(2.4)]{FT91}, we have $\vo_{\E}^{\times} = \vo_{\F}^{\times} \mu_{\E}$ since $\Nr(1+\sqrt{2})=-1$. Hence by \cite[Theorem 4.12 \& Proposition 4.16]{Wa82}, we deduce
\begin{equation} 
	R_{\E}/R_{\F} = 2. 
\label{RegulatorRatio}
\end{equation}
	From \cite[\Rmnum{5}.(2.1)]{FT91}, all possible cyclotomic fields embeddable in a bi-quadratic field are $\Q[\zeta_*]$ with $* \in \{ 2,4,6,8,12 \}$. They have degrees $1,2,2,4,4$ over $\Q$, and are given explicitly by $\Q, \Q[i], \Q[\sqrt{-3}], \Q[\sqrt{2},i]$ and $\Q[\sqrt{-3},i]$. Thus we get
\begin{equation}
	\frac{\omega_{\F}}{\omega_{\E}} = \left\{ \begin{matrix} 1 & \text{if } N > 3 \\ 1/3 & \text{if } N = 3 \end{matrix} \right..
\label{RootUnityRatio}
\end{equation}
	We have seen $\vo_{\E}^{\times} = \vo_{\F}^{\times} \mu_{\E}$. By \cite[Theorem 42]{FT91} the fundamental unit $\epsilon_{\F} := 1+\sqrt{2}$ of $\F$ is also a fundamental unit $\epsilon_{\E}$ of $\E$. A theorem of Herglotz \cite{He22} (see also \cite[(1)]{BP73}) implies
\begin{equation}
	\frac{h_{\E}}{h_{\F}} = \frac{h(-N)h(-2N)}{\lambda_0} = \frac{h(-N)h(-2N)}{2},
\label{ClNumRatio}
\end{equation}
	since $\lambda_0 = 2$ is determined by
	$$ \Nr_{\E/\F}(\epsilon_{\E}) = \epsilon_{\F}^{\lambda_0}. $$

\begin{proof}[Proof of Corollary \ref{MainQSq2}]
	Inserting (\ref{RegulatorRatio})-(\ref{ClNumRatio}) to (\ref{L1}) and (\ref{LN1}), we get for $N > 3$
	$$ L^{(N)}(1,-4N) = \pi^2 h(-N)h(-2N) \cdot \left\{ \begin{matrix} \frac{5}{4\sqrt{2}N} & \text{if } N \equiv 3 (8) \\ \frac{1}{2\sqrt{2}N} & \text{if } N \equiv 7 (8) \\ \frac{3}{8\sqrt{2}N} & \text{if } N \equiv 1 (4) \end{matrix} \right. . $$
	The assertion in (1) readily follows from the above formula and (\ref{BiasQSq2Sim}). For $N=3$, we similarly get
	$$ L^{(3)}(1,-12) = \frac{5}{36\sqrt{2}} \pi^2 h(-3)h(-6) = \frac{5}{18\sqrt{2}} \pi^2. $$ 
	Moreover, we have $L^{(3)}(1,-3) = L(1,\eta_{\E/\F}) = \sqrt{2} \pi^2/ 18$. The assertion in (2) follows readily from
	$$ B(\vec{k},27) = \frac{48\sqrt{2}}{\pi^2} L^{(3)}(1,-12) - \frac{12\sqrt{2}}{\pi^2} L^{(3)}(1,-3) \prod_{v \mid \infty} \Re \left( \frac{1}{2}+\frac{\sqrt{3}}{2}i \right)^{1-2k_v}. $$
\end{proof}

\appendix

\section{Simple Supercuspidal Representations of $\PGL_2$}\label{app:simpsupcusp}

In this appendix, we are going to review the theory of simple supercuspidal representations for $\PGL_2$, hence provides necessary information to complete the proof of Proposition \ref{NArchC}. 

Fix $\Fp\in |\F|_f$ with local field $\F_\Fp\supset \Fo_\Fp$ and $\Fo_\Fp/\Fp =\kappa_\Fp\simeq \BF_q$.
Set $\gp{H}_\Fp = \gp{Z}_\Fp\cdot \gp{K}^\p_\Fp$ with
$$
\gp{K}_\Fp^\p =
\left(
\begin{smallmatrix}
1+\Fp & \Fo_\Fp\\
\Fp & 1+\Fp
\end{smallmatrix}
\right)\subset \gp{K}_\Fp = \GL_2(\Fo_\Fp).
$$

Following \cite[\S 2]{KL15}, there is a class of characters of $\gp{H}_\Fp$ called the \emph{affine generic characters} with trivial central character. More precisely, let $\psi_\Fp$ be the additive character of $\F_\Fp$ that is trivial on $\Fo_\Fp$ but nontrivial on $\vpi^{-1}_\Fp\cdot\Fo_\Fp$. It turns out that the reduction-modulo-$\Fp$ of
$
\wt{\psi}_\Fp = \psi_\Fp(\vpi^{-1}_\Fp\cdot)
$
can be viewed as a non-trivial additive character of $\kappa_\Fp$. For any $t_1,t_2\in \kappa_\Fp$, define 
$\chi:\gp{H}_\Fp\to \BC^\times$ as follows, which, from \cite[\S 2]{KL15}, is a well-defined continuous homomorphism,
\begin{equation}\label{app:const:affgeneric}
\chi(zk) = \wt{\psi}_\Fp (t_1r_1+t_2r_2),\quad z\in \Z_\Fp,k = 
\left(
\begin{smallmatrix}
x_1 & r_1\\
\vpi_\Fp\cdot r_2 & x_2
\end{smallmatrix}
\right)\in \gp{K}^\p_\Fp.
\end{equation}
For an affine generic character $\chi$, define the compactly-induced smooth representation
\begin{equation}\label{app:const:pichi}
\pi_\chi = c\text{-}\Ind^{\gp{G}_\Fp}_{\gp{H}_\Fp}(\chi),\quad \gp{G}_\Fp = \GL_2(\F_\Fp).
\end{equation}

It is known from \cite[Proposition~3.1]{KL15} and \cite[Th.~4.4]{KL15} that $
\pi_\chi\simeq \bigoplus_{\zeta^2 = 1}
\sigma^\zeta_\chi.
$ with $\{\sigma^\zeta_\chi\}_{\zeta^2=1}$ two inequivalent supercuspidal representations of $\GL_2(\F_\Fp)$ with trivial central character.
Following Gross and Reeder, these supercuspidal representations are called the \emph{simple supercuspidal representations} of $\GL_2(\F_\Fp)$ with trivial central character. The set of equivalent classes of simple supercuspidal representations of $\GL_2(\F_\Fp)$ is parametrized by 
\begin{equation}\label{app:const:simplecuspparametrize}
\CS_\Fp = \{(t,\zeta)\in \kappa_\Fp^\times \times \BC^\times\mid\quad \zeta^2 = 1\}.
\end{equation}
In particular, for each pair $(t,\zeta)\in \CS_\Fp$, following \eqref{app:const:affgeneric}, the corresponding affine generic character is given by 
$
\chi(zk) = \wt{\psi}_\Fp(r_1+tr_2).
$

\subsection{Normalized matrix coefficient}

For $(t,\zeta)\in \CS_\Fp$, set
$
(V_{\pi_\chi})^\chi=\{ f\in V_{\pi_\chi}|\quad \pi_\chi(h)v = \chi(h)\cdot f \text{ for any $h\in \gp{H}_\Fp$} \}.
$
There is the following nonzero element in $(V_{\pi_\chi})^\chi$, 
\begin{equation}\label{app:normalizedmat:f0}
f_0(h) = \bigg\{
\begin{matrix}
\chi(h), & h\in \gp{H}_\Fp\\
0, & \text{otherwise}.
\end{matrix}
\end{equation}

Set 
$
g_\chi = \left(
\begin{smallmatrix}
0 & t\\
\vpi_\Fp& 0
\end{smallmatrix}
\right).
$
Let $L$ be the operator $f_0(\cdot)\mapsto f_0(g^{-1}_\chi\cdot)$. By \cite[(4.3)]{KL15}, the vector
\begin{equation}\label{app:normalizedmat:f0zetdef}
f^\zeta_0 = f_0+(\zeta\cdot L)f_0\quad \text{lies in } \sigma^\zeta_\chi.
\end{equation}
From \cite[Lemma~3.2]{KL15}, $g_\chi$ normalizes the character $\chi$. Moreover, by a simple calculation, $g_\chi$ also normalizes $\gp{H}_\Fp$. Therefore 
\begin{equation}\label{app:normalizedmat:f0zetact}
\pi_\chi(g_\chi)f_0^\zeta = \zeta\cdot f^\zeta_0.
\end{equation}

Set 
\begin{equation}\label{app:normalizedmat:Cchizet}
C^{\chi,\zeta}_0(g):=
\frac{\langle \sigma^\zeta_\chi(g)f^\zeta_0,f^\zeta_0\rangle}{
\langle f_0^\zeta,f^\zeta_0\rangle
}.
\end{equation}

\begin{lem}\label{lem:app:normalizedmatformula}
$$
C^{\chi,\zeta}_0(g)=\chi(g)\cdot \mathbbm{1}_{\gp{H}_\Fp}(g)+\zeta\cdot \chi(g_\chi g)\cdot \mathbbm{1}_{\gp{H}_\Fp}(g_\chi g).
$$
\end{lem}
\begin{proof}
By definition,
\begin{align*}
\langle \sigma^\zeta_\chi(g)f^\zeta_0,f^\zeta_0\rangle
=&
\langle \pi_\chi(g)f_0,f_0\rangle
+
\langle \pi_\chi(g)f_0(g^{-1}_\chi\cdot),f_0(g^{-1}_\chi\cdot)\rangle
\\
&+
\zeta\cdot\langle \pi_\chi(g)f_0,f_0(g^{-1}_\chi\cdot)\rangle
+
\zeta\cdot\langle \pi_\chi(g)f_0(g^{-1}_\chi\cdot),f_0\rangle
\\
=&
2(\chi(g)\cdot \mathbbm{1}_{\gp{H}_\Fp}(g)+\zeta\cdot \chi(g_\chi g)\cdot \mathbbm{1}_{\gp{H}_\Fp}(g_\chi g)).
\end{align*}
Here $\mathbbm{1}_{\gp{H}_\Fp}$ is the characteristic function of $\gp{H}_\Fp$, and $\chi(g)$ is defined to be zero whenever $g\notin \gp{H}_\Fp$. Therefore 
$
\langle f^\zeta_0,f^\zeta_0\rangle = 2
$
and the result follows immediately.
\end{proof}

\subsection{Local root number}

\begin{lem}\label{lem:app:whittakerformula}
The following linear functional $W_\chi(\cdot,1):\pi_\chi\to \BC$ defined via
\begin{align*}
f\in \pi_\chi\mapsto \int_{\F_\Fp}
f
\left(
\begin{smallmatrix}
1 & x\\
0 & 1
\end{smallmatrix}
\right)\wt{\psi}_\Fp(-x)\ud x
\end{align*}
restricts to a non-zero Whittaker function on each $\sigma^\zeta_\chi$. Set $w = 
\left(
\begin{smallmatrix}
0 & 1\\
-1 & 0
\end{smallmatrix}
\right).
$ Then the following equalities hold w.r.t. $\wt{\psi}_\Fp$.
\begin{align*}
W_\chi(f^\zeta_0,
\left(
\begin{smallmatrix}
a & 0\\
0 & 1
\end{smallmatrix}
\right)
) = \mathbbm{1}_{1+\Fp}(a),
\quad 
W_\chi(f^\zeta_0,
\left(
\begin{smallmatrix}
a & 0\\
0 & 1
\end{smallmatrix}
\right)w
) =\zeta \cdot
\mathbbm{1}_{1+\Fp}(at^{-1}\vpi_\Fp).
\end{align*}
\end{lem}
\begin{proof}
It suffices to establish the two equalities. By definition,
\begin{align*}
W_\chi(f^\zeta_0,
\left(
\begin{smallmatrix}
a & 0\\
0 & 1
\end{smallmatrix}
\right))
=
\int_{\F_\Fp}
f_0^\zeta
(
\left(
\begin{smallmatrix}
1 & x\\
0  & 1
\end{smallmatrix}
\right)
\cdot
\left(
\begin{smallmatrix}
a & 0\\
0  & 1
\end{smallmatrix}
\right))
\wt{\psi}_\Fp(-x)\ud x
=
\int_{\F_\Fp}
f_0^\zeta
\left(
\begin{smallmatrix}
a & x\\
0  & 1
\end{smallmatrix}
\right)
\wt{\psi}_\Fp(-x)\ud x.
\end{align*}
Since
$
g^{-1}_\chi\cdot 
\left(
\begin{smallmatrix}
a & x\\
0  & 1
\end{smallmatrix}
\right) = 
\left(
\begin{smallmatrix}
0 & \vpi^{-1}_\Fp\\
at^{-1}  & t^{-1}x
\end{smallmatrix}
\right)\notin \gp{H}_\Fp,
$
we get
$$
W_\chi(f^\zeta_0,\left(
\begin{smallmatrix}
a & x\\
0  & 1
\end{smallmatrix}
\right)) = 
\int_{\F_\Fp}
f_0
\left(
\begin{smallmatrix}
a & x\\
0 & 1
\end{smallmatrix}
\right)\wt{\psi}_\Fp(-x)\ud x.
$$
Note that 
$
\left(
\begin{smallmatrix}
a & x\\
0  & 1
\end{smallmatrix}
\right)\in \gp{H}_\Fp
$
if and only if 
$
x\in \Fo_\Fp$ and $a\in 1+\Fp.$ When it occurs,
$$
\int_{\F_\Fp}
f_0
\left(
\begin{smallmatrix}
a & x\\
0 & 1
\end{smallmatrix}
\right)\wt{\psi}_\Fp(-x)\ud x = 
\int_{\Fo_\Fp}\wt{\psi}_\Fp(x)\wt{\psi}_\Fp(-x)\ud x = \vol(\Fo_\Fp) = 1.
$$
Hence the first identity holds.

For the second identity, since
$
w = 
\left(
\begin{smallmatrix}
0 & 1\\
-1 & 0 
\end{smallmatrix}
\right) = 
\left(
\begin{smallmatrix}
t^{-1} & 0\\
0 & -\vpi^{-1}_\Fp
\end{smallmatrix}
\right)
g_\chi,
$
we get
$$
W_\chi(f^\zeta_0,
\left(
\begin{smallmatrix}
a & 0\\
0 & 1
\end{smallmatrix}
\right)w
)=
\int_{\F_\Fp}
f^\zeta_0
(
\left(
\begin{smallmatrix}
at^{-1} & -x\vpi^{-1}_\Fp\\
0 & -\vpi^{-1}_\Fp
\end{smallmatrix}
\right)
g_\chi
)
\wt{\psi}_\Fp(-x)\ud x.
$$
By \eqref{app:normalizedmat:f0zetact},
$
f^\zeta_0
(
\left(
\begin{smallmatrix}
at^{-1} & -x\vpi^{-1}_\Fp\\
0 & -\vpi^{-1}_\Fp
\end{smallmatrix}
\right)
g_\chi
) = \zeta \cdot 
f^\zeta_0
\left(
\begin{smallmatrix}
at^{-1} & -x\vpi^{-1}_\Fp\\
0 & -\vpi^{-1}_\Fp
\end{smallmatrix}
\right).
$
Hence
$$
W_\chi(f^\zeta_0,
\left(
\begin{smallmatrix}
a & 0\\
0 & 1
\end{smallmatrix}
\right)w
)=
\zeta \cdot
\int_{\F_\Fp}
f^\zeta_0
\left(
\begin{smallmatrix}
at^{-1} & -x\vpi^{-1}_\Fp\\
0 & -\vpi^{-1}_\Fp
\end{smallmatrix}
\right)
\wt{\psi}_\Fp(-x)\ud x.
$$
Following the same argument as the first identity, 
$
\left(
\begin{smallmatrix}
at^{-1} & -x\vpi^{-1}_\Fp\\
0 & -\vpi^{-1}_\Fp
\end{smallmatrix}
\right)\in \gp{H}_\Fp
$
if and only if there exists $z\in \F^\times_\Fp$, such that
$
z\vpi^{-1}_\Fp\in 1+\Fp, x\in \Fo_\Fp, at^{-1}\vpi_\Fp\in 1+\Fp.
$
When it occurs, 
$$
\int_{\F_\Fp}
f^\zeta_0
\left(
\begin{smallmatrix}
at^{-1} & -x\vpi^{-1}_\Fp\\
0 & -\vpi^{-1}_\Fp
\end{smallmatrix}
\right)
\wt{\psi}_\Fp(-x)\ud x
= 
\int_{\F_\Fp}
\wt{\psi}_\Fp(x)\wt{\psi}_\Fp(-x)\ud x = 1.
$$
It follows that 
$
W_\chi(f^\zeta_0,
\left(
\begin{smallmatrix}
a & 0 \\
0 & 1
\end{smallmatrix}
\right)w
)=\zeta 
\cdot
\mathbbm{1}_{1+\Fp}(at^{-1}\vpi_\Fp)
$
and the second formula holds.
\end{proof}

From \cite[\S 6]{BuH06},
$
L(s,\sigma^\zeta_\chi) = 1.
$
Hence in order to compute $\varepsilon(\frac{1}{2},\sigma^\zeta_\chi,\wt{\psi}_\Fp)$, it suffices to compute the Hecke integrals in \cite[Theorem~2.18]{JL70}, which is the following lemma.

\begin{lem}\label{lem:app:localrootnumHeckeintegral}
The following equalities hold. In particular
$
\varepsilon(\frac{1}{2},\sigma^\zeta_\chi,\wt{\psi}_\Fp) = 
\zeta.
$

$$ \int_{\F^\times_\Fp}
W_\chi(f^\zeta_0,
\left(
\begin{smallmatrix}
a & 0\\
0 & 1
\end{smallmatrix}
\right)
)\ud^\times a 
= 1, \quad
\int_{\F^\times_\Fp}
W_\chi(f^\zeta_0,
\left(
\begin{smallmatrix}
a & 0\\
0 & 1
\end{smallmatrix}
\right)w
)
\ud^\times a =\zeta.
$$
\end{lem}
\begin{proof}
From Lemma \ref{lem:app:whittakerformula}, the following identities hold
\begin{align*}
\int_{\F^\times_\Fp}
W_\chi(f^\zeta_0,
\left(
\begin{smallmatrix}
a & 0\\
0 & 1
\end{smallmatrix}
\right)
)|a|_\Fp^{s-\frac{1}{2}}\ud^\times a
&=
\int_{\F^\times_\Fp}
\mathbbm{1}_{1+\Fp}(a)|a|^{s-\frac{1}{2}}_\Fp \ud^\times a
=1
\\
\int_{\F^\times_\Fp}
W_\chi(f^\zeta_0,
\left(
\begin{smallmatrix}
a & 0\\
0 & 1
\end{smallmatrix}
\right)w
)|a|_\Fp^{\frac{1}{2}-s}
\ud^\times a
=&
\zeta
\cdot 
\int_{\F_\Fp^\times}
\mathbbm{1}_{1+\Fp}(at^{-1}\vpi_\Fp)
|a|_\Fp^{\frac{1}{2}-s}
\ud^\times a.
\end{align*}
Therefore
$
\int_{\F^\times_\Fp}
W_\chi(f^\zeta_0,
\left(
\begin{smallmatrix}
a & 0\\
0 & 1
\end{smallmatrix}
\right)w
)
\ud^\times a=
\zeta.
$
It follows that the lemma has been established.
\end{proof}
Based on \cite[p.38]{JL70} and $\wt{\psi}_\Fp(\cdot) = \psi_\Fp(\vpi^{-1}_\Fp\cdot)$, the following corollary holds.
\begin{cor}\label{cor:app:localrootnumlevel0}
$
\varepsilon(\frac{1}{2},\sigma^\zeta_\chi,\psi_\Fp) = 
\zeta.
$
\end{cor}

\subsection{Test function}\label{app:subsec:testfun}

Let the formal degree of $\sigma^\zeta_\chi$ be $d_{\zeta,\chi}$. By \cite[Proposition~6.1]{KL15} and the proof of \cite[Corollary~6.5]{KL15},
\begin{equation}\label{eq:app:formaldeg}
d_{\zeta,\chi} = \frac{q^2-1}{2}.
\end{equation}
From \cite[Lemma~10.23]{KL06} and \cite[Corollary~10.26]{KL06}, for any irreducible unitary representation $\pi_\Fp$ of $\PGL_2(\F_\Fp)$ and any vector $v\in \pi_\Fp$,
\begin{equation}\label{eq:app:testfun:projectiontoline}
\pi_\Fp(\zeta\cdot d_{\zeta,\chi}\cdot \wb{C^{\chi,\zeta}_0})
v = 
\bigg\{
\begin{matrix}
\zeta \cdot v,& \sigma^\zeta_\chi\simeq \pi_\Fp, v\in \BC \cdot f_0^\zeta\\
0, &\text{otherwise}
\end{matrix}
\end{equation}

\begin{defin}\label{defin:app:testfun}
Define
$
f^b_\Fp := 
\sum_{\chi,\zeta}
\zeta\cdot d_{\zeta,\chi}\cdot \wb{C^{\chi,\zeta}_0}.
$
\end{defin}
By definition $f^b_\Fp\in \CC^\infty_c(\PGL_2(\F_\Fp))$.
Moreover, by \eqref{eq:app:testfun:projectiontoline} and Corollary \ref{cor:app:localrootnumlevel0}, 
\begin{equation}\label{eq:app:wholetestfun:projectiontoline}
\pi_\Fp(f^b_\Fp)
v = 
\bigg\{
\begin{matrix}
\varepsilon(\frac{1}{2},\sigma^\zeta_\chi,\psi_\Fp) \cdot v,& \sigma^\zeta_\chi\simeq \pi_\Fp, v\in \BC \cdot f_0^\zeta\\
0, &\text{otherwise}
\end{matrix}
\end{equation}
for any irreducible unitary representation $\pi_\Fp$ of $\PGL_2(\F_\Fp)$ and $v\in \pi_\Fp$.

\begin{lem}\label{lem:app:supportoffb}
$f^b_\Fp$ is supported on 
$
\left(
\begin{smallmatrix}
0 & 1\\
\vpi_\Fp & 0
\end{smallmatrix}
\right)
\cdot 
\Z_\Fp \cdot \RI_\Fp
$
where $\RI_\Fp$ is the standard Iwahori subgroup of $\GL_2(\Fo_\Fp)$.
Moreover, for any $z\in F_\Fp^\times$, $x_1,x_2\in \Fo^\times_\Fp$, $r_1,r_2\in \Fo_\Fp$,
$$
f^b_\Fp
(
\left(
\begin{smallmatrix}
0 & 1\\
\vpi_\Fp & 0
\end{smallmatrix}
\right)
z
\left(
\begin{smallmatrix}
x_1 & r_1\\
\vpi_\Fp r_2& x_2
\end{smallmatrix}
\right)
)=
(q^2-1)\cdot
\wt{\psi}_\Fp(-\frac{r_1+r_2}{x_1}).
$$
\end{lem}
\begin{proof}
From Definition \ref{defin:app:testfun}, 
$
f^b_\Fp = 
\sum_{\chi,\zeta}
\zeta\cdot d_{\zeta,\chi}\cdot \wb{C^{\chi,\zeta}_0}.
$
By \eqref{eq:app:formaldeg} and Lemma \ref{lem:app:normalizedmatformula},
\begin{align*}
f^b_\Fp(g) 
&= \frac{q^2-1}{2}\cdot\sum_{\chi,\zeta}
\zeta \cdot (\wb{\chi}(g)\cdot \mathbbm{1}_{\gp{H}_\Fp}(g)+\wb{\zeta}\cdot \wb{\chi}(g_\chi\cdot g)\cdot \mathbbm{1}_{\gp{H}_\Fp}(g_\chi\cdot g))
\\
&=(q^2-1)\cdot \sum_{\chi}\chi^{-1}(g_\chi\cdot g)\cdot \mathbbm{1}_{\gp{H}_\Fp}(g_\chi \cdot g).
\end{align*}
By definition,
\begin{align*}
g^{-1}_\chi\cdot \gp{H}_\Fp 
= 
\left(
\begin{smallmatrix}
0 & \vpi^{-1}_\Fp\\
1 & 0
\end{smallmatrix}
\right)
\left(
\begin{smallmatrix}
t^{-1} & 0\\
0 & 1
\end{smallmatrix}
\right)
\Z_\Fp
\left(
\begin{smallmatrix}
1+\Fp & \Fo_\Fp\\
\Fp & 1+\Fp
\end{smallmatrix}
\right)=
\left(
\begin{smallmatrix}
0 & \vpi_\Fp\\
1 & 0
\end{smallmatrix}
\right)
\Z_\Fp
\left(
\begin{smallmatrix}
t^{-1}+\Fp & \Fo_\Fp\\
\Fp & 1+\Fp
\end{smallmatrix}
\right).
\end{align*}
The disjoint union over $t\in \Fo^\times_\Fp/(1+\Fp)$ is exactly given by 
$
\left(
\begin{smallmatrix}
0 & \vpi_\Fp\\
1 & 0
\end{smallmatrix}
\right)
\cdot
\Z_\Fp
\cdot
\RI_\Fp.
$
For $x_1\in t^{-1}+\Fp$,
\begin{align*}
\chi^{-1}
(g_\chi
\left(
\begin{smallmatrix}
0 & \vpi_\Fp\\
1 & 0
\end{smallmatrix}
\right)
z
\left(
\begin{smallmatrix}
x_1 & r_1\\
\vpi_\Fp r_2 & x_2
\end{smallmatrix}
\right)
)
=
\chi^{-1}
\left(
\begin{smallmatrix}
tx_1 & tr_1\\
\vpi_\Fp r_2 & x_2
\end{smallmatrix}
\right)
=
\wt{\psi}_\Fp(-tr_1-tr_2).
\end{align*}
It follows that the lemma has been proved.
\end{proof}

\begin{lem}\label{lem:app:Iwahoriaveragefb}
Set 
$$
\wt{f}^b_\Fp(g):=\frac{1}{\vol(\RI_\Fp)}
\int_{\RI_\Fp}f^b_\Fp(k^{-1}gk)\ud k.
$$
Then $\wt{f}^b_\Fp$ has the same support as $f^b_\Fp$, and for any $x_1,x_2\in \Fo^\times_\Fp$, $r_1,r_2\in \Fo_\Fp$, 
$$
\wt{f}^b_\Fp(
\left(
\begin{smallmatrix}
0 & 1\\
\vpi_\Fp & 0
\end{smallmatrix}
\right)
z
\left(
\begin{smallmatrix}
x_1 & r_1\\
\vpi_\Fp r_2 & x_2
\end{smallmatrix}
\right)
) = 
(q+1)
\cdot 
\big\{
\begin{smallmatrix}
q-1, & r_1+r_2\in \Fp\\
-1, & \text{otherwise}
\end{smallmatrix}
$$
\end{lem}
\begin{proof}
For any 
$
k=
\left(
\begin{smallmatrix}
a & b\\
c & d
\end{smallmatrix}
\right)\in \RI_\Fp,\quad k^{-1} = \frac{1}{ad-bc}
\left(
\begin{smallmatrix}
d & -b\\
-c & a
\end{smallmatrix}
\right),
$
the following identity holds
$$
\left(
\begin{smallmatrix}
0 & \vpi^{-1}_\Fp\\
1 & 0
\end{smallmatrix}
\right)
k^{-1}
\left(
\begin{smallmatrix}
0 & 1\\
\vpi_\Fp & 0
\end{smallmatrix}
\right)
=
\frac{1}{ad-bc}
\left(
\begin{smallmatrix}
a & -\vpi^{-1}_\Fp c\\
-\vpi_\Fp b & d
\end{smallmatrix}
\right).
$$
Set 
$
k^{-1}
\left(
\begin{smallmatrix}
0 & 1\\
\vpi_\Fp & 0
\end{smallmatrix}
\right)z
\left(
\begin{smallmatrix}
x_1 & r_1\\
\vpi_\Fp r_2 & x_2
\end{smallmatrix}
\right)k=
\left(
\begin{smallmatrix}
0 & 1\\
\vpi_\Fp & 0
\end{smallmatrix}
\right)z
\left(
\begin{smallmatrix}
X_1 & R_1\\
\vpi_\Fp R_2& X_2
\end{smallmatrix}
\right).
$
By straightforward computation,
\begin{align*}
(ad-bc)\cdot
\left(
\begin{smallmatrix}
X_1 & R_1\\
\vpi_\Fp R_2& X_2
\end{smallmatrix}
\right)
=
\left(
\begin{smallmatrix}
a^2x_1+ac(r_1-r_2)-\vpi^{-1}_\Fp c^2x_2 & abx_1-bcr_2+adr_1-\vpi^{-1}_\Fp cdx_2\\
-\vpi_\Fp abx_1+\vpi_\Fp adr_2-\vpi_\Fp bcr_1+cdx_2 & 
-\vpi_\Fp b^2x_1 - \vpi_\Fp bd(r_1-r_2)+d^2x_2
\end{smallmatrix}
\right).
\end{align*}
Therefore 
$
R_1+R_2 = r_1+r_2,\quad X_1\equiv ad^{-1}x_1\mod \Fp.
$
It follows that by Lemma \ref{lem:app:supportoffb},
$$
f^b_\Fp(k^{-1}
\left(
\begin{smallmatrix}
0 & 1\\
\vpi_\Fp & 0
\end{smallmatrix}
\right)
z
\left(
\begin{smallmatrix}
x_1 & r_1\\
\vpi_\Fp r_2& x_2
\end{smallmatrix}
\right)k
)=(q^2-1)
\cdot
\wt{\psi}_\Fp
(-\frac{r_1+r_2}{ad^{-1}x_1}).
$$
When $r_1+r_2\in \Fp$, 
$
\wt{\psi}_\Fp
(-\frac{r_1+r_2}{ad^{-1}x_1}) = 1
$
identically.
When $r_1+r_2\in \Fo^\times_\Fp$, for fixed $a\in \Fo^\times_\Fp$, 
$
\sum_{d\in \Fo^\times/(1+\Fp)}
\wt{\psi}_\Fp
(-\frac{r_1+r_2}{ad^{-1}x_1}) = -1.
$
Then the formula follows from straightforward computation.
\end{proof}

\section*{Acknowledgement}

	During the whole preparation of this paper, Qinghua Pi is supported by Shandong Provincial Natural Science Foundation (Grant No. ZR2020MA001), Han Wu is supported by the Leverhulme Trust Research Project Grant RPG-2018-401. Han Wu would like to thank Prof. Ian Petrow for a discussion on automorphic Fourier inversion. We thank the anonymous referee for many suggestions which improve a lot the readability of the paper.

\bibliographystyle{acm}

\bibliography{Biasbib}

\bigskip

\noindent Zhilin Luo \\
Department of Mathematics \\
University of Chicago \\
Chicago IL, 60637, USA \\
zhilinchicago@uchicago.edu

\bigskip

\noindent Qinghua Pi \\
School of Mathematics and Statistics \\
Shandong Univeristy, Weihai \\
Weihai 264209, China \\
qhpi@sdu.edu.cn

\bigskip

\noindent Han Wu \\
School of Mathematical Sciences \\
Queen Mary University of London \\
Mile End Road \\
E1 4NS, London, UK \\
wuhan1121@yahoo.com

\end{document}